\documentclass[a4paper,12pt]{amsart}

\usepackage[
  margin=30mm,
  marginparwidth=25mm,     
  marginparsep=2mm,       
  bottom=25mm,
  ]{geometry}

\usepackage[latin1]{inputenc}
\usepackage[T1]{fontenc}
\usepackage[english]{babel}
\usepackage{amsmath,amssymb,amsthm,mathtools}
\usepackage[makeroom]{cancel}
\expandafter\let\expandafter\xbf\csname bfseries \endcsname
\expandafter\let\expandafter\xmd\csname mdseries \endcsname
\let\xbar\bar
\let\bar\xbar
\expandafter\let\csname bfseries \endcsname\xbf
\expandafter\let\csname mdseries \endcsname\xmd

\usepackage{latexsym}
\usepackage{delarray}
\usepackage{bbm}
\usepackage{scrextend}
\usepackage{hyperref}
\usepackage[pdftex,usenames,dvipsnames]{xcolor}
\usepackage{datetime}
\usepackage{mathrsfs}
\usepackage{enumitem}
\usepackage{tikz}

\usepackage{MnSymbol,wasysym}

\usepackage{multicol}

\newtheorem{theorem}{Theorem}[section]

\newtheorem{lemma}{Lemma}[section]

\newtheorem{remark}{Remark}[section]

\newtheorem*{Rem*}{Remark}

\newcommand{\R}{\mathbb{R}}
\newcommand{\N}{\mathbb{N}}

\newcommand{\C}{\mathbb{C}}
\newcommand{\Z}{\mathbb{Z}}

\newcommand{\OO}{\mathcal{O}}
\newcommand{\V}{\mathcal{V}}

\newcommand{\edproof}{ $\hfill {\Box}$}

\DeclareMathOperator{\supp}{supp}

\newcommand{\red}[1]{\textcolor{black}{#1}}
\newcommand{\blue}[1]{\textcolor{black}{#1}}

\allowdisplaybreaks 

\title
[Variation inequalities]
{Variation inequalities for Riesz transforms and Poisson semigroups associated with Laguerre polynomial expansions}

\author[J. J. Betancor]{J. J. Betancor}
\address{Jorge J. Betancor\newline
	Departamento de An\'alisis Matem\'atico, Universidad de La Laguna,\newline
	Campus de Anchieta, Avda. Astrof\'isico S\'anchez, s/n,\newline
	38721 La Laguna (Sta. Cruz de Tenerife), Spain}
\email{jbetanco@ull.es}

\author[M. De Le\'on-Contreras]{M. De Le\'on-Contreras$^*$}
\address{\newline
       Marta De Le\'on-Contreras \newline
       Department of  Mathematical Sciences, Norwegian University of Science and Technology (NTNU),  \newline NO-7491 Trondheim, Norway
       }
\email{marta.deleoncontreras@ntnu.no }


\subjclass[2010]
{}

\thanks{$^*$Corresponding author}
\begin{document}

\begin{abstract}
In this paper we establish $L^p$-boundedness properties for variation, oscillation and jump operators associated with Riesz transforms and Poisson semigroups related to Laguerre polynomial expansions.
\end{abstract}
\maketitle

\setcounter{secnumdepth}{3}
\setcounter{tocdepth}{3}


\section{Introduction}
Let $\alpha>-1.$ For every $k\in\N$ we denote by $L_k^\alpha$ the $k$-th Laguerre polynomial defined by  
$$
L_k^\alpha(x) =\frac{1}{\Gamma(k+1)}x^{-\alpha}e^x\frac{d^k}{dx^k}(e^{-x}x^{k+\alpha}), \quad x\in (0,\infty),
$$
see \cite{Leb} and \cite{Sz}.
The sequence $\left\{ \tilde{L}_k^\alpha=\sqrt{\frac{\Gamma(k+1){\Gamma(\alpha+1)}}{\Gamma(k+\alpha+1)}}{L}_k^\alpha\right\}_{k\in\N}$ is an orthonormal basis in $L^2((0,\infty), \mu_\alpha)$ where $ d\mu_\alpha(x)=\frac{x^{\alpha}e^{-x}dx}{{\Gamma(\alpha+1)}}$ on $(0,\infty).$

We consider the Laguerre differential operator $\mathbf{L}_\alpha$ given by 
$$
\mathbf{L}_\alpha=x\frac{d^2}{dx^2}+(\alpha+1-x)\frac{d}{dx}, \:\;\:\quad\text{ in } (0,\infty).
$$
We have that $\mathbf{L}_\alpha({L}_k^\alpha)=k\;{L}_k^\alpha$, $k\in\N.$

\red{Let $n\in \N$, $n\ge 1$}. Suppose that $\alpha=(\alpha_1,\dots,\alpha_n)\in(-1,\infty)^n$. For every $k=(k_1,\dots,k_n)\in\N^n,$ we define the $k$-th Laguerre polynomial in $(0,\infty)^n$ by
$$
L_k^\alpha(x)=\prod_{i=1}^nL_{k_i}^{\alpha_i}(x_i), \quad x=(x_1,\dots,x_n)\in (0,\infty)^n.
$$
The family $\left\{\prod_{i=1}^n\sqrt{\frac{\Gamma(k_i+1){\Gamma(\alpha_i+1)}}{\Gamma(k_i+\alpha_i+1)}}{L}_k^{{\alpha}}\right\}_{k=(k_1,\dots,k_n)\in\N^n}$ is an orthonormal basis in $L^2((0,\infty)^n,\mu_\alpha)$, where $d\mu_\alpha(x)=\prod_{i=1}^n\frac{x_i^{\alpha_i}e^{-x_i}}{{\Gamma(\alpha_i+1)}}dx_i$ on $(0,\infty)^n.$

We consider the Laguerre operator $\mathbf{L}_\alpha=\sum_{i=1}^n\mathbf{L}_{\alpha_i}$ in $(0,\infty)^n$. We have that
$$
\mathbf{L}_\alpha(L_k^{{\alpha}})=\sum_{i=1}^nk_iL_{k}^{{\alpha}}\,\quad k=(k_1,\dots,k_n)\in\N^n.
$$
For every $f\in L^2((0,\infty)^n, \mu_\alpha)$ we define
$$
c_k^\alpha(f)=\int_{(0,\infty)^n}f(x)L_k^\alpha(x)d\mu_\alpha(x)\prod_{i=1}^n\frac{\Gamma(k_i+1){\Gamma(\alpha_i+1)}}{\Gamma(k_i+\alpha_{{i}}+1)}.
$$
We define the operator $\mathcal{L}_\alpha$ as follows
$$
\mathcal{L}_\alpha(f)=\sum_{k\in\N^n}\lambda_k c_k^\alpha(f)L_k^\alpha,\quad f\in\mathcal{D}(\mathcal{L}_\alpha),
$$
where $\lambda_k=\sum_{i=1}^n k_i,$ $k=(k_1,\dots,k_n)\in\N^n,$ and
$$
\mathcal{D}(\mathcal{L}_\alpha)=\left\{f\in L^2((0,\infty)^n,\mu_\alpha):\quad \sum_{k\in\N^n}\Bigg|\lambda_k c_k^\alpha(f)\prod_{i=1}^n\sqrt{\frac{\Gamma(k_i+\alpha_i+1)}{\Gamma(k_i+1)\Gamma(\alpha_i+1)}}\Bigg|^2<\infty \right\}
$$
is the domain of $\mathcal{L}_\alpha$. If $f\in C^\infty_c((0,\infty)^n)$, the space of smooth functions with compact support in $(0,\infty)^n$, then $f\in \mathcal{D}(\mathcal{L}_\alpha)$ and $\mathcal{L}_\alpha f=\mathbf{L}_\alpha f.$

The operator $\mathcal{L}_\alpha$ is positive and symmetric. Furthermore, $-\mathcal{L}_\alpha$ generates a semigroup of operators $\{W_t^\alpha \}_{t>0}$ in $L^2((0,\infty)^n,\mu_\alpha)$ where, for every $t>0,$
$$
W_t^\alpha(f)=\sum_{k\in\N^n}e^{-\lambda_kt}c_k^\alpha(f)L_k^\alpha, \quad f\in L^2((0,\infty)^n, \mu_\alpha).
$$
According to the {Hille}-Hardy formula, see \cite[(4.17.6)]{Leb}, we have that
\begin{align*}
    \sum_{k\in\N^n}&e^{-\lambda_kt}L_k^\alpha(x)L_k^\alpha(y)\prod_{i=1}^n\frac{\Gamma(k_i+1){\Gamma(\alpha_i+1)}}{\Gamma(k_i+\alpha_i+1)}\\
    &=(1-e^{-t})^{-n}\prod_{i=1}^n{\Gamma(\alpha_i+1)}\exp\left(-\frac{e^{-t}}{1-e^{-t}}(x_i+y_i) \right)(e^{-t}x_iy_i)^{-\alpha_i/2}I_{\alpha_i}\left(\frac{2\sqrt{e^{-t}x_i y_i}}{1-e^{-t}}\right), \\&\qquad x=(x_1,\dots,x_n), \: y=(y_1,\dots,y_n)\in (0,\infty)^n \text{ and } t>0.
\end{align*}
Here $I_\nu$ denotes the modified Bessel function of the first kind and order $\nu$, see \cite[(5.7.1)]{Leb}. We can write, for every $f\in L^2((0,\infty)^n, \mu_\alpha),$

\begin{equation}\label{eq1.1}
    W_t^\alpha(f)(x)=\int_{(0,\infty)^n}W_t^\alpha (x,y)f(y)d\mu_\alpha(y), \quad x\in (0,\infty)^n,
\end{equation}
where
\begin{align*}
    W_t^\alpha(x,y)&=(1-e^{-t})^{-n}\prod_{i=1}^n{\Gamma(\alpha_i+1)}\exp\left(-\frac{e^{-t}}{1-e^{-t}}(x_i+y_i) \right)(e^{-t}x_iy_i)^{-\alpha_i/2}I_{\alpha_i}\left(\frac{2\sqrt{e^{-t}x_i y_i}}{1-e^{-t}}\right),
     \\&\qquad x=(x_1,\dots,x_n), \: y=(y_1,\dots,y_n)\in (0,\infty)^n \text{ and } t>0.
\end{align*}
The integral in \eqref{eq1.1} is absolutely convergent for every $f\in L^p((0,\infty)^n, \mu_\alpha)$ with  $1\le p\le \infty.$ We define, for every $t>0$ and $1\le p\le \infty,$ $W_t^\alpha$ on $L^p((0,\infty)^n, \mu_\alpha)$ by \eqref{eq1.1}. Thus, $\{W_t^\alpha \}_{t>0}$ is a symmetric diffusion semigroup in the Stein's sense, see \cite{StLP}, on the measure space $((0,\infty)^n,\mu_\alpha)$.

The Poisson semigroup $\{ P_t^\alpha\}_{t>0}$ associated with $\mathcal{L}_\alpha$ is defined by using the subordination principle, that is, for every $t>0$ and $f\in L^p((0,\infty)^n,\mu_\alpha)$, with $1\le p\le \infty$,
$$
P_t^\alpha(f)(x)=\frac{t}{2\sqrt{\pi}}\int_0^\infty \frac{e^{-\frac{t^2}{4u}}}{u^{3/2}} W_u^\alpha (f)(x)du, \quad x\in (0,\infty)^n.
$$
$\{P_t^\alpha\}_{t>0}$ is a symmetric diffusion semigroup in $((0,\infty)^n,\mu_\alpha)$.

The study of harmonic analysis in the Laguerre setting was begun by Muckenhoupt. He established $L^p$-boundedness properties for the maximal operator defined by the Laguerre semigroup $\{{W_t^\alpha}\}_{t>0}$, see \cite[Theorem 3]{Mu2}, and for the Riesz transformation associated with $\mathcal{L}_\alpha$, see \cite[Theorem 3]{Mu1}, both in one dimension. The results in \cite[Theorem 3]{Mu2} were extended to any higher dimension in \cite[Theorem 1]{Di}. Multidimensional Laguerre Riesz transforms were studied by Guti\'errez, Incognito and Torrea, \cite{GIT}, by using a transference method. They obtained $L^p$, $1<p<\infty$, boundedness properties for the discrete set of half-integers multi-indices $\alpha$. These results were extended to $\alpha\in [-1/2,\infty)^n$ by Nowak, see \cite[Theorem 13]{No}, by using the methods contained in \cite{StLP} based on the Littlewood-Paley-Stein theory. Furthermore, in \cite[Therorem 13]{No} it was proved that the norm of Laguerre-Riesz transforms can be controlled by constants that are independent of the dimension $n$. Other dimension free bounds for Riesz transforms in the Laguerre setting had been proved in \cite{GLLNU}, \cite{GIT}  and \cite{MS}. More recently, Wrobel, see \cite{WR1} and \cite{WR2}, extended the free dimension $L^p$-boundedness properties for the Riesz-Laguerre transforms when $\alpha\in (-1,\infty)^n$ and $1<p<\infty$. The weak type $(1,1)$ with respect to $\mu_\alpha$ for Riesz transforms was proved by Sasso when $\alpha\in (0,\infty)^n$, see \cite[Theorem 1.1]{Sa1}. All the above mentioned results concerned the first order Riesz transforms.  In \cite{FSS1}, Forzani, Sasso and Scotto obtained the weak type $(1,1)$ for the second order Riesz-Laguerre transforms. They also find the sharp polynomial weight $w$ making that the Riesz-Laguerre transforms of order greater than two are bounded from $L^1((0,\infty)^n, wd\mu_\alpha)$ into $L^{1,\infty}((0,\infty)^n, d\mu_\alpha)$, when $\alpha\in (0,\infty)^n$. Other operators associated with the Laguerre harmonic analysis were studied in \cite{ FSS2, HE,Sa2,Sa3, Sa4}.

\vspace{1cm}

The aim of this paper is to study the $L^p$-boundedness properties for the variation, oscillation and jump operators defined by the Laguerre Poisson semigroups and the family of {truncated} integrals for Riesz-Laguerre transforms.

Suppose that $(X,\mu)$ is a measure space and that, for every $t>0,$ $S_t$ is a bounded operator from $L^p(X,\mu)$ into itself, with $1\le p<\infty.$

Let $\rho>2$. For every $f\in L^p(X,\mu)$, the variation operator $\V_\rho(\{S_t\}_{t>0})$ is defined by
$$
\V_\rho(\{S_t\}_{t>0})(f)(x)=\sup_{\substack{0<t_k<t_{k-1}<\dots<t_1\\k\in\N}}\left(\sum_{j=1}^{k-1}|S_{t_j}(f)(x)-S_{t_{j+1}}(f)(x)|^\rho \right)^{1/\rho},\:\; x\in X.
$$
Note that, in general, some conditions of $t$-continuity should be imposed in order to ensure the measurability of $\V_\rho(\{S_t\}_{t>0})(f)$ (see the comments after Theorem 1.2 in \cite{CJRW1}).

If $\{t_j\}_{j\in\N}\subset (0,\infty)$ is a decreasing sequence that converges to zero, the oscillation operator of $f\in L^p(X,\mu)$,  $\OO(\{S_t\}_{t>0},\{t_j\}_{j\in\N})(f)$, is given by
$$
\OO(\{S_t\}_{t>0},\{t_j\}_{j\in\N})(f)(x)=\left(\sum_{j=1}^{\infty}\sup_{t_{j+1}\le s_{j+1}<s_j\le t_{j}}|S_{s_j}(f)(x)-S_{s_{j+1}}(f)(x)|^2 \right)^{1/2},\:\; x\in X.
$$
For every $k\in\Z$ and $f\in L^p(X,\mu)$, we define
$$
V_k(\{S_t\}_{t>0})(f)(x)=\sup_{\substack{2^{-k}<t_l<t_{l-1}<\dots<t_{1}<2^{-k+1}\\l\in\N}}\left(\sum_{j=1}^{l-1}|S_{t_j}(f)(x)-S_{t_{j+1}}(f)(x)|^2\right)^{1/2},\:\; x\in X.
$$
For every $f\in L^p(X,\mu)$, the short variation operator $\mathcal{S}_V(\{S_t\}_{t>0})$ is defined by
$$
\mathcal{S}_V(\{S_t\}_{t>0})(f)(x)=\left( \sum_{k=-\infty}^\infty (V_k(\{S_t\}_{t>0})(f)(x))^2\right)^{1/2},\quad x\in X.
$$
Let $\lambda>0.$ For $f\in L^p(X,\mu)$, the $\lambda$-jump operator $\Lambda(\{S_t\}_{t>0},\lambda)$ is defined by
\begin{align*}
    \Lambda(\{S_t\}_{t>0},\lambda)(f)(x)=\sup\{n\in\N: \:\exists \: s_1<t_1\le s_2<t_2\le\dots\le s_n<t_n, \text{ such that }\\ |S_{t_i}(f)(x)-S_{s_i}(f)(x)|>\lambda, \: i=1,\dots,n\}, \: \; x\in X.
\end{align*}
The analysis of these variation, oscillation and jump operators inform us about convergence properties for $\{S_t\}_{t>0}.$

One of the motivations of the variational inequalities is to improve the well-known Doob's maximal inequality, see \red{\cite[Chapter XI]{Doob}}. After the pioneer work of L\'epingle in \cite{Le},  Bourgain obtained the variational estimates related to the Birkhoff ergodic averages, see \cite{Bou}. This work opens  a new research subject in harmonic analysis and ergodic theory. Bourgain's results were extended by Jones, Kaufmann, Rosenblatt and Wierdl in \cite{JKRW}. Variational inequalities for truncations of singular integrals were initiated by Campbell, Jones, Reinhold and Wierdl, \cite{CJRW1}, who studied them for the Hilbert transform. Later, higher dimensional singular integrals were considered in \cite{CJRW2}. $L^p$-boundedness properties for variation operators associated with the classical Poisson semigroups were established in \cite{JW}, see also \cite{CJRW2}. When $p>1$, the results can be extended to strongly continuous semigroups of positive contractions in $L^p$, see \cite{LeX}. In \cite{JR}, it was proved that the variation operators for Stein symmetric diffusion semigroups are of weak type $(1,1).$ For further studies for variation operators related to singular integrals and semigroups of operators, we refer to the reader to \cite{BFHR,CDHL,CHL,CMMTV,GT,HMMT,JSW,MTX1,MTX2} and the references therein.

Let $\beta\ge 0$. We choose $m\in\N$ such that $m-1\le \beta<m.$ Suppose that $g\in C^m(0,\infty).$ We define the Weyl derivative, $\mathcal{D}^\beta g$, as follows
$$
\mathcal{D}^\beta g(t)=\frac{e^{i\pi(m-\beta)}}{\Gamma(m-\beta)}\int_0^\infty g^{(m)}(t+s) s^{m-\beta-1}ds,
$$
provided that $t>0$ and the integral exists. Note that if $\beta\in\N$ and $\lim_{t\to\infty}g^{(\beta)}(t)=0$, then $\mathcal{D}^\beta g(t)=g^{(\beta)}(t)$, $t>0.$

In our first result we establish $L^p$-boundedness properties for $\{ t^\beta\mathcal{D}^\beta P_t^\alpha\}_{t>0}$, with $\beta\ge 0.$

\begin{theorem}\label{teo1.1}
Let $\beta\ge 0$, $\rho>2$ and $\alpha\in (\red{0},\infty)^n.$ Assume that $\{t_j\}_{j\in\N}$ is a decreasing sequence in $(0,\infty)$ that converges to $0$. Then, the operators
\begin{multicols}{2}
\begin{itemize}
\item[$i)$] $\V_\rho(\{ t^\beta\mathcal{D}^\beta P_t^\alpha\}_{t>0})$\\
\item[$ii)$] $\OO(\{t^\beta\mathcal{D}^\beta P_t^\alpha\}_{t>0},\{t_j\}_{j\in\N})$,
\end{itemize}
\begin{itemize}
\item[$iii)$] ${\lambda}\Lambda(\{t^\beta\mathcal{D}^\beta P_t^\alpha\}_{t>0},\lambda)^{1/\rho}$ with $\lambda>0$,\\
\item[$iv)$] $\mathcal{S}_V(\{t^\beta\mathcal{D}^\beta P_t^\alpha\}_{t>0}),$ 
\end{itemize}
\end{multicols}   are bounded from $L^p((0,\infty)^n,\mu_\alpha)$ into itself, for every $1<p<\infty$, and from \newline$L^1((0,\infty)^n,\mu_\alpha)$ into $L^{1,\infty}((0,\infty)^n,\mu_\alpha)$. Furthermore, the $L^p$-boundedness properties of ${\lambda}\Lambda(\{t^\beta\mathcal{D}^\beta P_t^\alpha\}_{t>0},\lambda)^{1/\rho}$ are uniform in $\lambda>0.$
\end{theorem}
By the way in the proof  of Theorem \ref{teo1.1} $iv)$ 
we shall establish $L^p$-boundedness properties for the fractional Littlewood-Paley functions associated with the Poisson semigroup $\{P_t^\alpha\}_{t>0}$. For every $\beta>0$, we define the Littlewood-Paley function $g_\alpha^\beta$ by
$$
g_\alpha^\beta (f)(x)=\left( \int_0^\infty |t^\beta\mathcal{D}_t^\beta P_t^\alpha (f)(x)|^2\frac{dt}{t}\right)^{1/2}, \quad x\in (0,\infty)^n.
$$
\begin{theorem}\label{teo1.2}
Let $\alpha\in (\red{0},{\infty})^n$ and $\beta>0.$ The operator $g_\alpha^\beta$ is bounded from $L^p((0,\infty)^n,\mu_\alpha)$ into itself, for every $1<p<\infty$, and from $L^1((0,\infty), \mu_\alpha)$ into  $L^{1,\infty}((0,\infty), \mu_\alpha)$.
\end{theorem}

We now consider, for every $\alpha>-1$, the Laguerre type operator $\tilde{\Delta}_\alpha$ defined by 
$$
\tilde{\Delta}_\alpha=\frac{1}{2}\frac{d^2}{dx^2}+\left( \frac{2\alpha+1}{2x}-x\right)\frac{d}{dx}+\alpha+1, \quad \text{on } (0,\infty).
$$
We have that
$$
\tilde{\Delta}_\alpha L_k^\alpha(x^2)=(2k+\alpha+1)L_k^\alpha(x^2), \:\; x\in (0,\infty) \text{ and } k\in\N.
$$

Let $\alpha=(\alpha_1,\dots,\alpha_n)\in (-1,{\infty})^n$, \red{$n\in \N$}. We define $\tilde{\Delta}_\alpha=\sum_{i=1}^n\tilde{\Delta}_{\alpha_i}$ on $(0,\infty)^n.$ For every $k=(k_1,\dots,k_n)\in\N^n$, we consider 
$$
\mathcal{L}_k^\alpha(x)=\prod_{i=1}^nL_{k_i}^{\alpha_i}(x_i^2)\sqrt{\frac{\Gamma(k_i+1){\Gamma(\alpha_i+1)}}{\Gamma(k_i+\alpha_i+1)}}, \:\; x=(x_1,\dots,x_n)\in (0,\infty)^n.
$$
We have that
$$
\tilde{\Delta}_\alpha\mathcal{L}_k^\alpha=(2\sum_{i=1}^nk_i+\sum_{i=1}^n\alpha_i+n)\mathcal{L}_k^\alpha,\:\:\;k=(k_1,\dots,k_n)\in\N^n.
$$
The sequence $\{\mathcal{L}_k^\alpha \}_{k\in\N}$ is an orthonormal basis in $L^2((0,\infty)^n, \nu_\alpha)$,  where $d\nu_\alpha(x)={2^n}\prod_{j=1}^n\frac{x_j^{2{\alpha_j}+1}e^{-x_j^2}}{{\Gamma(\alpha_j+1)}}dx_j.$

We define the operator $\Delta_\alpha$ by
$$
\Delta_\alpha f=\sum_{k\in\N^n}\lambda_k^\alpha d_k^\alpha(f)\mathcal{L}_k^\alpha, \:\;f\in\mathcal{D}(\Delta_\alpha),
$$
where $\mathcal{D}(\Delta_\alpha)=\{f\in L^2((0,\infty)^n,\nu_\alpha): \:\; \sum_{k\in\N^n}|\lambda_k^\alpha d_k^\alpha(f)|^2<\infty\}.$ 
Here, for every $k=(k_1,\dots,k_n)\in\N^n$ and $f\in L^2((0,\infty)^n,\nu_\alpha),$
$$
d_k^\alpha(f)=\int_{(0,\infty)^n}f(y)\mathcal{L}_k^\alpha (y)d\nu_\alpha(y)
$$
and $\lambda_k^\alpha=2\sum_{i=1}^n k_i+\sum_{i=1}^n\alpha_i +n.$ Thus, $\Delta_\alpha$ is positive and symmetric. The operator $-\Delta_\alpha$  generates a semigroup $\{\mathcal{W}_t^\alpha\}_{t>0}$ in $L^2((0,\infty)^n, \nu_\alpha),$ being, for every $t>0, $
$$
\mathcal{W}_t^\alpha(f)=\sum_{k\in\N^n}e^{-t\lambda_k^\alpha}d_k^\alpha(f)\mathcal{L}_k^\alpha, \:\: f\in L^2((0,\infty)^n,\nu_\alpha).
$$
For every $x=(x_1,\dots,x_n)\in (0,\infty)^n$, we shall denote $x^2:=(x_1^2,\dots,x_n^2).$ Thus, for every $f\in L^2((0,\infty)^n, \nu_\alpha)$ and $t>0$, we have that
\begin{equation}\label{F0}
    \mathcal{W}_t^\alpha(f)(x)=e^{-t\left( \sum_{i=1}^n\alpha_i+n\right)}\int_{(0,\infty)^n} {W}_{2t}^\alpha(x^2,y^2)f(y)d\nu_\alpha(y)\red{=:}\int_{(0,\infty)^n} \mathcal{W}_{t}^\alpha(x,y)f(y)d\nu_\alpha(y).
\end{equation}
By defining, for  $f\in L^p((0,\infty)^n, \nu_\alpha)$, $1\le p\le\infty,$ $\mathcal{W}_t^\alpha(f)$ as in \eqref{F0}, we deduce that $\{\mathcal{W}_t^\alpha\}_{t>0}$ is a symmetric diffusion semigroup in the Stein's sense in $((0,\infty)^n,\nu_\alpha).$

Let $i=1,\dots,n$. We define the Riesz transform $\mathcal{R}_{i,\alpha}$ associated with the operator $\Delta_\alpha$ by 
$$
\mathcal{R}_{i,\alpha}(f)=\sum_{k\in\N}\frac{d_k^\alpha(f)}{\sqrt{\lambda_k^\alpha}}\partial_{x_i}\mathcal{L}_k^\alpha, \quad f\in L^2((0,\infty)^n, \nu_\alpha).
$$
According to \cite[{Section} 3]{NS1}, $\mathcal{R}_{i,\alpha}$ is bounded from $L^2((0,\infty)^n,\nu_\alpha)$ into itself. By arguing as in the proof of \cite[Theorem {1.1}]{BFRS} (see also \cite{CNS}) we can see that, for every $f\in C^\infty_c((0,\infty)^n)$, the space of compactly supported functions in $(0,\infty)^n$,
$$
\mathcal{R}_{i,\alpha}(f)(x)=\lim_{\epsilon\to 0^+}\int_{\substack{|x-y|>\epsilon\\ y\in (0,\infty)^n}}\mathcal{R}_\alpha^i(x,y)f(y)d\nu_\alpha(y), \quad a.e. \:\; x\in (0,\infty)^n,
$$
where 
$$
\mathcal{R}_\alpha^i(x,y)=\frac{1}{\sqrt{\pi}}\int_0^\infty \partial_{x_i}\mathcal{W}_t^\alpha(x,y)\frac{dt}{\sqrt{t}},\quad x,y\in (0,\infty)^n, \:\;x\neq y.
$$
$L^p$-boundedness properties of $\mathcal{R}_{i,\alpha}$ can be proved as in \cite{Sa1}.

We define, for every $\epsilon>0$, the $\epsilon$-truncated Riesz transform by
$$
\mathcal{R}_{i,\alpha;\epsilon}(f)(x)=\int_{|x-y|>\epsilon}\mathcal{R}_\alpha^i(x,y)f(y)d\nu_\alpha(y), \:\; x\in (0,\infty)^n.
$$
In the following result we establish variational $L^p$ inequalities for the family $\{\mathcal{R}_{i,\alpha;\epsilon}\}_{\epsilon>0}$ of truncations with $i=1,\dots,n.$

\begin{theorem}\label{teo1.3}
Let $\alpha\in (\red{0},\infty)^n$, $\rho>2$ and $i=1,\dots,n.$ Assume that $\{\epsilon_j\}_{j\in\N}$ is a decreasing sequence in $(0,\infty)$ that converges to zero. Then, then operators
\begin{multicols}{2}
\begin{itemize}
 \item [i)] $\V_\rho(\{\mathcal{R}_{i,\alpha;\epsilon}\}_{\epsilon>0})$ 
\item[ii)] {$\OO(\{\mathcal{R}_{i,\alpha;\epsilon}\}_{\epsilon>0},\{\epsilon_j\}_{j\in\N})$ }
\item [iii)]$\lambda\Lambda(\{\mathcal{R}_{i,\alpha;\epsilon} \}_{\epsilon>0},\lambda)^{1/\rho}$, with $\lambda>0,$
\end{itemize}
\end{multicols}
 are bounded from $L^p((0,\infty)^n,{\nu}_\alpha)$ into itself, for every $1<p<\infty$, and from \newline $L^1((0,\infty)^n,{\nu}_\alpha)$ into $L^{1,\infty}((0,\infty)^n,{\nu}_\alpha)$.  Furthermore,  the $L^p$ boundedness properties of $\lambda\Lambda(\{\mathcal{R}_{i,\alpha;\epsilon} \}_{\epsilon>0},\lambda)^{1/\rho}$ are uniform in $\lambda>0.$
\end{theorem}

In the next sections we present the proofs for Theorems \ref{teo1.1}, \ref{teo1.2} and \ref{teo1.3}. Throughout this paper, $c$ and $C$ always represent positive constants that can change in each occurrence.

\section{Proof of Theorem \ref{teo1.1}}

\subsection{Proof of Theorem \ref{teo1.1} i)}\label{subs2.1}
\textcolor{white}{w}
\newline

We define $x^2=(x_1^2,\dots,x_n^2),$ when $x=(x_1,\dots,x_n)\in (0,\infty)^n.$ We consider the map $\Psi:(0,\infty)^n\to (0,\infty)^n$ defined by $\Psi(x)=x^2$, $x\in (0,\infty)^n$, and the Borel measure $\nu_\alpha$ on $(0,\infty)^n$, given by
$$\nu_\alpha(A)=\mu_\alpha(\Psi^{-1}(A)), \text{ for every Borel measurable set }A \text{ in } (0,\infty)^n.
$$
Thus, $\nu_\alpha$ is the probability measure defined by $d\nu_\alpha(x)=2^n\prod_{i=1}^n\frac{x_i^{2\alpha_i+1}e^{-x_i^2}}{\Gamma(\alpha_i+1)}dx_i$ on $(0,\infty)^n.$ It is not hard to see that the operator $S_\Psi$ defined by $S_\Psi(f)(x)=f(\Psi(x))$, $x\in (0,\infty)^n$, is an isometry from $L^q((0,\infty)^n,\mu_\alpha)$ onto $L^q((0,\infty)^n,\nu_\alpha)$ and from $L^{q,\infty}((0,\infty)^n,\mu_\alpha)$ onto $L^{q,\infty}((0,\infty)^n,\nu_\alpha)$, for every $1\le q\le \infty$. Then, an operator $\mathcal{A}$ is bounded from  $L^q((0,\infty)^n,\mu_\alpha)$ into itself (respectively, into $L^{q,\infty}((0,\infty)^n,\mu_\alpha)$) if, and only if, the operator $\hat{\mathcal{A}}$, defined by $\hat{\mathcal{A}}=S_\Psi\mathcal{A}S_\Psi^{-1}$ is bounded from $L^q((0,\infty)^n,\nu_\alpha)$  into itself (respectively, into $L^{q,\infty}((0,\infty)^n,\nu_\alpha)$).

We have that, for every $t>0,$
\begin{align*}
    {\hat{W}}_t^\alpha(f)(x)=S_\Psi[{W}_t^\alpha(S_{\Psi}^{-1}(f))](x)=\int_{(0,\infty)^n}  {\hat{W}}_t^\alpha(x,y)f(y)d\nu_\alpha(y), \quad x\in (0,\infty)^n,
\end{align*}
where
$$
\hat{W}_t^\alpha(x,y)={W_t^\alpha(x^2,y^2)},\quad x,y\in (0,\infty)^n,
$$
and then
\begin{align*}
    \hat{P}_t^\alpha(f)(x)=S_\Psi[P_t^\alpha(S_{\Psi}^{-1}(f))](x)=\int_{(0,\infty)^n}  \hat{P}_t^\alpha(x,y)f(y)d\nu_\alpha(y), \quad x\in (0,\infty)^n,
\end{align*}
being
$$
 \hat{P}_t^\alpha(x,y)={\frac{t}{2\sqrt{\pi}}\int_0^\infty \frac{e^{-\frac{t^2}{4u}}}{u^{3/2}}W_u^\alpha(x^2,y^2)du}, \quad x,y\in (0,\infty)^n \text{ and } t>0.
$$
It follows that
$$
\mathcal{V}_\rho(\{t^\beta\mathcal{D}_t^\beta \hat{P}_t^\alpha\}_{t>0})(f)=S_{\Psi}(\V_\rho\{t^\beta\mathcal{D}_t^\beta {P}_t^\alpha\}_{t>0})(S_{\Psi}^{-1}(f)).
$$
\subsubsection{$\mathcal{V}_\rho(\{t^\beta\mathcal{D}_t^\beta \hat{P}_t^\alpha\}_{t>0})$ is bounded from $L^p((0,\infty)^n,\nu_\alpha)$ into itself,  $1<p<\infty$}
\textcolor{white}{v}\newline
\vspace{0.1cm}

Since $\{W_t^\alpha\}_{t>0}$ is a Stein's symmetric diffusion semigroup on $((0,\infty)^n,\mu_\alpha)$, then  $\{\hat{W}_t^\alpha\}_{t>0}$ is a Stein's symmetric diffusion semigroup on $((0,\infty)^n,\nu_\alpha)$. Then, according to \cite[Corollary 6.1]{LeX} (see also \cite[Theorem 3.3]{JR}), the operator $\mathcal{V}_\rho(\{t^m\partial_t^m\hat{W}_t^\alpha \}_{t>0})$ is bounded from $L^p((0,\infty)^n,\nu_\alpha)$ into itself, for every $1<p<\infty$ and $m\in\N$. Moreover,  $\mathcal{V}_\rho(\{t^m\partial_t^m\hat{P}_t^\alpha \}_{t>0})$ is also bounded from $L^p((0,\infty)^n,\nu_\alpha)$ into itself, for every $1<p<\infty$ and $m\in\N$.

In the following we shall prove that, for every $L^p((0,\infty)^n,\nu_\alpha)$ $1\le p<\infty$, 
$$
\mathcal{V}_\rho(\{t^\beta\mathcal{D}_t^\beta \hat{P}_t^\alpha\}_{t>0})(f)(x)\le C\mathcal{V}_\rho(\{\hat{W}_t^\alpha \}_{t>0})(f)(x),\quad x\in (0,\infty)^n.
$$
Indeed, let $f\in L^p((0,\infty)^n,\nu_\alpha)$, with $1\le p<\infty$ \blue{and $\beta>0$}. We choose $m\in\N$ such that $m-1\le \beta<m$. We have that
\begin{align*}
    \mathcal{D}_t^\beta \hat{P}_t^\alpha(f)(x)&=\frac{e^{i\pi(m-\beta)}}{\Gamma(m-\beta)}\int_0^\infty \partial_t^m\hat{P}_{t+s}^\alpha(f)(x) s^{m-\beta-1}ds\\
    &=\frac{e^{i\pi(m-\beta)}}{\Gamma(m-\beta)}\int_0^\infty\int_0^\infty \frac{\partial_t^m\left( (t+s)e^{-\frac{(t+s)^2}{4u}}\right) }{2\sqrt{\pi}u^{3/2}}\hat{W}_{u}^\alpha(f)(x)du \;s^{m-\beta-1}ds,
\end{align*}
 for almost every $x\in(0,\infty)^n, \text{ and } t>0.$

 \blue{Observe that}
 \blue{\begin{align*}
 	\partial_t^m \left(te^{-\frac{t^2}{4u}}\right)&=-2u\partial_t^{m+1} e^{-\frac{t^2}{4u}}=-\frac{1}{2^m u^{\frac{m-1}{2}}}\partial_v^{m+1}e^{-v^2}|_{v=\frac{t}{2\sqrt{u}}}\\
 	&=\frac{(-1)^me^{-\frac{t^2}{4u}}}{2^m u^{\frac{m-1}{2}}}H_{m+1}\left(\frac{t}{2\sqrt{u}}\right) \quad t,u\in (0,\infty),
 	\end{align*}
 	where, for every $\ell\in\N$, $H_\ell$ denotes the $\ell-th$ Hermite polynomial, see \cite[p. 106, (5.5.3)]{Sz}}.

 \blue{Then, for $t>0$ and $x\in(0,\infty)^n$,
 	\begin{align*}
 	\partial_t^m\hat{P}_{t}^\alpha(f)(x)&=\frac{1}{2\sqrt{\pi}}\int_0^\infty\frac{\partial_t^m \left( te^{-\frac{t^2}{4u}}\right) }{u^{3/2}}\hat{W}_{u}^\alpha(f)(x)du\\
 	&=\frac{(-1)^m}{2^{m+1}\sqrt{\pi}}\int_0^\infty \frac{e^{-\frac{t^2}{4u}}}{ u^{\frac{m+2}{2}}}H_{m+1}\left(\frac{t}{2\sqrt{u}}\right)\hat{W}_{u}^\alpha(f)(x)du\\
 	&=\frac{(-1)^m}{t^{m}2\sqrt{\pi}}\int_0^\infty e^{-v}v^{m/2-1} H_{m+1}\left(\sqrt{v}\right)\hat{W}_{\frac{t^2}{4v}}^\alpha(f)(x)dv.
 	\end{align*}}
 Thus, we get that
 \blue{\begin{align}\label{F2}
 	t^\beta &\mathcal{D}_t^\beta \hat{P}_t^\alpha(f)(x)\nonumber\\&=\frac{(-1)^me^{i\pi(m-\beta)}t^\beta}{\Gamma(m-\beta)2\sqrt{\pi}}\int_0^\infty s^{m-\beta-1}\int_0^\infty \frac{e^{-v}v^{m/2-1}}{(t+s)^m} H_{m+1}\left(\sqrt{v}\right)\hat{W}_{\frac{(t+s)^2}{4v}}^\alpha(f)(x)dv\;ds \nonumber\\
 	&=\frac{(-1)^me^{i\pi(m-\beta)}}{\Gamma(m-\beta)2\sqrt{\pi}}\int_0^\infty z^{m-\beta-1}\int_0^\infty \frac{e^{-v}v^{m/2-1}}{(1+z)^m} H_{m+1}\left(\sqrt{v}\right)\hat{W}_{\frac{t^2(1+z)^2}{4v}}^\alpha(f)(x)dv\;dz,
 	\\&\hspace{10cm}  x\in(0,\infty)^n \text{ and } t>0.\nonumber 
 	\end{align}
 }

 \blue{Observe that the differentiations under the integral signs  above are always justified for almost every $ x\in(0,\infty)^n \text{ and } t>0.$  Indeed, we have that}
 \blue{\begin{align*}
 \int_0^\infty \frac{e^{-\frac{t^2}{4u}}}{ u^{\frac{m+2}{2}}}\left|H_{m+1}\left(\frac{t}{2\sqrt{u}}\right)\right|&|\hat{W}_{u}^\alpha(f)(x)|du\\&
 \le C\int_0^\infty \frac{e^{-\frac{t^2}{8u}}}{u^{\frac{m+2}{2}}}|\hat{W}_{u}^\alpha(f)(x)|du\\
 &\le C\sup_{u>0}|\hat{W}_{u}^\alpha(f)(x)|\frac{1}{t^m}, \quad x\in (0,\infty)^n \text{ and } t>0. 
 \end{align*}}


Since $\{ \hat{W}_t^\alpha\}_{u>0}$ is a Stein's symmetric diffusion semigroup, by Stein's maximal theorem, see \cite[Chapter 3, Section 3]{StLP}, the maximal operator
$$
\hat{W}_*^\alpha(f)=\sup_{t>0}|\hat{W}_t^\alpha(f)|
$$
is bounded from $L^q((0,\infty)^n,\nu_\alpha)$ into itself, for every $1<q\le \infty.$ Furthermore, $\hat{W}_*^\alpha$ is bounded from $L^1((0,\infty)^n,\nu_\alpha)$ into $L^{1,\infty}((0,\infty)^n,\nu_\alpha)$, see {\cite{Mu2}} for the case $n=1$ and \cite{Di} for $n\ge 2.$ Then, $\hat{W}_*^\alpha(f)(x)<\infty$ for almost all $x\in (0,\infty)^n.$

We also have that
\blue{\begin{align}\label{A1}
      \int_0^\infty \int_0^\infty \frac{e^{-\frac{(t+s)^2}{4u}}\left|H_{m+1}\left(\frac{t+s}{2\sqrt{u}}\right)\right| }{u^{\frac{m+2}{2}}}|\hat{W}_{u}^\alpha(f)(x)|du \;s^{m-\beta-1}ds&\le C\hat{W}_{*}^\alpha(f)(x)\int_0^\infty\frac{ s^{m-\beta-1}} {(t+s)^{m}}ds \nonumber\\
      &\le C \hat{W}_*^\alpha (f)(x)t^{-\beta}, \\&\qquad \qquad x\in (0,\infty)^n \text{ and } t>0.\nonumber
\end{align}}

By using Minkowski's inequality \blue{in \eqref{F2}} we get that
\begin{align*}
    \mathcal{V}_\rho&(\{t^\beta\mathcal{D}_t^\beta \hat{P}_t^\alpha\}_{t>0})(f)(x)\\
   & \red{\le C \int_0^\infty\frac{z^{m-\beta-1}}{(1+z)^m}dz\int_0^\infty e^{-v}v^{m/2-1}H_{m+1}\left(\sqrt{v}\right)dv\;\mathcal{V}_\rho(\{\hat{W}_t^\alpha\}_{t>0})(f)(x)}\\
    &\le C \mathcal{V}_\rho(\{\hat{W}_t^\alpha\}_{t>0})(f)(x), \quad \text{ a.e. } x\in (0,\infty)^n.
\end{align*}
\blue{In a similar way obtain that 
$$
 \mathcal{V}_\rho(\{ \hat{P}_t^\alpha\}_{t>0})(f)(x)\le C\; \mathcal{V}_\rho(\{\hat{W}_t^\alpha\}_{t>0})(f)(x), \quad \text{ a.e. } x\in (0,\infty)^n.
$$}
Since $\mathcal{V}_\rho(\{\hat{W}_t^\alpha\}_{t>0})$ is bounded from $L^q((0,\infty)^n,\nu_\alpha)$ into itself for every $1<q<\infty$, see \cite[Corollary 6.1]{LeX}, $  \mathcal{V}_\rho(\{t^\beta\mathcal{D}_t^\beta \hat{P}_t^\alpha\}_{t>0})$ is bounded from $L^q((0,\infty)^n,\nu_\alpha)$ into itself, for every $1<q<\infty$.
\begin{remark} 
We define the maximal operator $\hat{P}_{*,\beta}^\alpha$ by 
$$
\hat{P}_{*,\beta}^\alpha(f)=\sup_{t>0}|t^\beta \mathcal{D}_t^\beta \hat{P}_t^\alpha(f)|.
$$
$L^p$-boundedness properties of    $\;\hat{W}_*^\alpha$ (see \cite{Di,Mu2}, and \cite{StLP}) and \eqref{A1} imply that $\hat{P}_{*,\beta}^\alpha$  is bounded from $L^q((0,\infty)^n,\nu_\alpha)$ into itself, for every $1<q<\infty$ and from $L^1((0,\infty)^n,\nu_\alpha)$ into $L^{1,\infty}((0,\infty)^n,\nu_\alpha)$. These results can be seen as extensions in our Laguerre setting of \cite[Corollary 4.2]{LeX2} and \cite[Section 4]{LSj}.
\end{remark}

\subsubsection{$\mathcal{V}_\rho(\{t^\beta\mathcal{D}_t^\beta \hat{P}_t^\alpha\}_{t>0})$ is bounded from $L^1((0,\infty)^n,\nu_\alpha)$ into $L^{1,\infty}((0,\infty)^n,\nu_\alpha)$.}

\textcolor{white}{v}\newline
\vspace{0.1cm}

 In this moment we do not know how to see that the operator  $\mathcal{V}_\rho(\{\hat{W}_t^\alpha\}_{t>0})$ is bounded from $L^1((0,\infty)^n,\nu_\alpha)$ into $L^{1,\infty}((0,\infty)^n,\nu_\alpha)$, so we will obtain our result by using a procedure developed by Sasso in \cite{Sa2} and \cite{Sa1}.

Observe that 
\begin{align*}
   { \hat{W}_t^\alpha(x,y)}
    &=(1-e^{-t})^{-n}\exp\left(-\frac{e^{-t}}{1-e^{-t}}\sum_{i=1}^n(x_i^2+y_i^2) \right)\prod_{i=1}^n\Gamma(\alpha_i+1)(e^{-t/2}x_iy_i)^{-\alpha_i} \\&\cdot I_{\alpha_i}\left(\frac{2{e^{-t/2}x_i y_i}}{1-{e^{-t}}}\right),\quad
     x=(x_1,\dots,x_n), \: y=(y_1,\dots,y_n)\in (0,\infty)^n \text{ and } t>0.
\end{align*}
\red{A}ccording to \cite[(5.10.22)]{Leb},
$$
I_\nu(z)=\frac{\left(\frac{z}{2}\right)^\nu}{\sqrt{\pi}\Gamma(\nu+1/2)}\int_{-1}^1(1-s^2)^{\nu-1/2}e^{-sz} ds,\quad z\in (0,\infty), \:\:\nu>-1/2,
$$
we can write, for every $x,y\in (0,\infty)^n$ and $t>0,$
\begin{align*}
      { \hat{W}_t^\alpha(x,y)}
     &=(1-e^{-t})^{-n-\sum_{i=1}^n\alpha_i}\int_{(-1,1)^n}\exp\left(-\frac{q_{{-}}(e^{-t/2}x,y,s)}{1-e^{-t}}\right)e^{|y|^2}\Pi_\alpha (s)ds,
   \end{align*}
     where $\Pi_\alpha(s)=\prod_{i=1}^n\frac{\Gamma(\alpha_i+1)}{\Gamma(\alpha_i+1/2)\sqrt{\pi}}(1-s_i^2)^{\alpha_i-1/2}$, $s=(s_1,\dots,s_n)\in (-1,1)^n$ and $q_{\pm}(x,y,s)=\sum_{i=1}^n (x_i^2+y_i^2\pm 2x_iy_is_i)$, 
$x=(x_1,\dots,x_n), \: y=(y_1,\dots,y_n)\in (0,\infty)^n \text{ and } s=(s_1,\dots,s_n)\in (-1,1)^n.$

 Let $E_\rho$ denote the function space that consists of those complex functions $g$ defined on $(0,\infty)$ such that
$$
\|g\|_{E_\rho}=\sup_{\substack{0<t_k<\dots<t_1\\k\in\N}}\left(\sum_{j=1}^{k-1}|g(t_j)-g(t_{j+1})|^\rho\right)^{1/\rho}<\infty.
$$
It is clear that $g$ is constant if, and only if, $\|g\|_{E_\rho}=0$. By identifying the functions that differ \red{by} a constant, $(E_\rho, \| \cdot\|_{E_\rho})$ is a Banach space.

We have that
$$
\mathcal{V}_\rho(\{t^\beta\mathcal{D}_t^\beta \hat{P}_{t}^\alpha(f)\}_{t>0})(x)=\|t^\beta\mathcal{D}_t^\beta \hat{P}_{t}^\alpha(f)(x)\|_{E_\rho}, \quad x\in (0,\infty)^n.
$$
   In order to prove that the $\rho$-variation operator $\mathcal{V}_\rho(\{t^\beta\mathcal{D}_t^\beta \hat{P}_t^\alpha\}_{t>0})$ is bounded from $L^1((0,\infty)^n,\nu_\alpha)$ into $L^{1,\infty}((0,\infty)^n,\nu_\alpha)$ we shall use a procedure developed by Sasso, see \cite{Sa2} and \cite{Sa1}.
 We are going to explain the steps of our proof. The set $(0,\infty)^n\times (0,\infty)^n\times {(-1,1)}^n$ is divided in two parts. For every $\tau>0$, the local region $N_\tau$ is defined by
 $$
 N_\tau=\left\{(x,y,s)\in  (0,\infty)^n\times (0,\infty)^n\times {(-1,1)}^n: \: q_{{-}}(x,y,s)^{1/2}\le \frac{C_0\tau}{1+|x|+|y|}\right\},
 $$
   and the global region $G_\tau=((0,\infty)^n\times (0,\infty)^n\times {(-1,1)}^n)\setminus N_\tau.$ Here,  $C_0=9(n+\sum_{i=1}^n\alpha_i)$.
   
   Let $t>0.$ We decompose $\hat{P}_t^\alpha$ as follows. Assume that $\varphi$ is a smooth function in $(0,\infty)^n\times (0,\infty)^n\times (-1,1)^n$ such that $0\le \varphi\le 1$,
   $$
    \varphi(x,y,s)=\left\{\begin{array}{cc}
       1,  & (x,y,s)\in N_1,  \\
       0,  & (x,y,s)\in N_2^c,
    \end{array} \right.
   $$
   and
   \begin{equation}\label{2.0}
      |\nabla_x\varphi(x,y,s)|+ |\nabla_y\varphi(x,y,s)|\le \frac{C}{q_{{-}}(x,y,s)^{1/2}}, \quad (x,y,s)\in (0,\infty)^n\times (0,\infty)^n\times (-1,1)^n.
   \end{equation}
   We define the local part of $\hat{P}_t^\alpha$ by
   $$
   \hat{P}_{t,loc}^\alpha(f)(x)=\int_{(0,\infty)^n}\hat{P}_{t,loc}^\alpha(x,y)f(y)d\nu_\alpha(y),\quad x\in (0,\infty)^n,
   $$
   where
   $$
   \hat{P}_{t,loc}^\alpha(x,y)=\frac{t}{2\sqrt{\pi}}\int_0^\infty \frac{e^{-\frac{t^2}{4u}}}{u^{3/2}}{\hat{W}_{u,loc}^\alpha(x,y)}du,\quad x,y\in (0,\infty)^n,
   $$ and
   $$
  {\hat{W}_{u,loc}^\alpha(x,y)}=(1-e^{-u})^{-n-\sum_{i=1}^n\alpha_i}\int_{(-1,1)^n}\exp\left(-\frac{q_{{-}}(e^{-u/2}x,y,s)}{1-e^{-u}}\right)e^{|y|^2}\varphi(x,y,s)\Pi_\alpha (s)ds, 
   $$
   with $x,y\in (0,\infty)^n \text{ and } u>0.$
   
   On the other hand, the global part of $ \hat{P}_{t}^\alpha$ is defined by
   $$
   \hat{P}_{t,glob}^\alpha(f)(x)=\int_{(0,\infty)^n}\hat{P}_{t,glob}^\alpha(x,y)f(y)d\nu_\alpha(y),\quad x\in (0,\infty)^n,
   $$
   being
   $$
   \hat{P}_{t,glob}^\alpha(x,y)=\frac{t}{2\sqrt{\pi}}\int_0^\infty \frac{e^{-\frac{t^2}{4u}}}{u^{3/2}}{\hat{W}_{u,glob}^\alpha(x,y)}du,\quad x,y\in (0,\infty)^n
   $$ and
   $$
  {\hat{W}_{u,glob}^\alpha(x,y)}=(1-e^{-u})^{-n-\sum_{i=1}^n\alpha_i}\int_{(-1,1)^n}\exp\left(-\frac{q_{{-}}(e^{-u/2}x,y,s)}{1-e^{-u}}\right)e^{|y|^2}(1-\varphi(x,y,s))\Pi_\alpha (s)ds, 
   $$
   with $x,y\in (0,\infty)^n \text{ and } u>0.$

   It is clear that
   $$
   \hat{P}_t^\alpha= \hat{P}_{t,loc}^\alpha+\hat{P}_{t,glob}^\alpha, \:\: t>0.
   $$
   We have that
   $$
   \mathcal{V}_\rho(\{t^\beta\mathcal{D}_t^\beta \hat{P}_t^\alpha\}_{t>0})(f)\le \mathcal{V}_\rho(\{t^\beta\mathcal{D}_t^\beta \hat{P}_{t,loc}^\alpha\}_{t>0})(f)+\mathcal{V}_\rho(\{t^\beta\mathcal{D}_t^\beta \hat{P}_{t,glob}^\alpha\}_{t>0})(f).
   $$

   We shall now establish that the operators   $\mathcal{V}_\rho(\{t^\beta\mathcal{D}_t^\beta \hat{P}_{t,loc}^\alpha\}_{t>0})$ and $\mathcal{V}_\rho(\{t^\beta\mathcal{D}_t^\beta \hat{P}_{t,glob}^\alpha\}_{t>0})$  are bounded from $L^1((0,\infty)^n,\nu_\alpha)$ into $L^{1,\infty}((0,\infty)^n,\nu_\alpha)$. For that aim we have to prove the following Lemmas \ref{Claim1}--\ref{Claim3}.
   
   We define the measure $dm_\alpha(y)=e^{|y|^2}d\nu_\alpha(y),$ $y\in (0,\infty)^n.$

   \begin{lemma}\label{Claim1}
 There exists a positive function $K_\alpha$ defined on $(0,\infty)^n\times (0,\infty)^n\times (-1,1)^n$ such that
   \begin{small}
   \begin{align*}
    \mathcal{V}_\rho(\{t^\beta\mathcal{D}_t^\beta \hat{P}_{t,glob}^\alpha\}_{t>0}&)(f)(x)\\&\le \int_{(0,\infty)^n}\int_{(-1,1)^n}K_\alpha(x,y,s)(1-\varphi(x,y,s))\Pi_\alpha(s)ds|f(y)|dm_\alpha(y),\\ &\qquad\qquad\qquad\qquad\qquad\qquad\qquad\qquad\qquad\qquad\qquad\qquad x\in (0,\infty)^n.
   \end{align*}
   \end{small}
 Moreover,  the operator $\mathbb{K}_\alpha$, defined by
   \begin{align*}
   \mathbb{K}_\alpha(f)(x)=\int_{(0,\infty)^n}\int_{(-1,1)^n}K_\alpha(x,y,s)(1-\varphi(x,y,s))\Pi_\alpha(s)ds|f(y)|&dm_\alpha(y),\\&\:\: x\in (0,\infty)^n,
   \end{align*}
   is bounded from $L^q((0,\infty)^n,\nu_\alpha)$ into itself, $1<q<\infty$, and from $L^1((0,\infty)^n,\nu_\alpha)$ into $L^{1,\infty}((0,\infty)^n,\nu_\alpha)$. 
   \end{lemma}
   Once we prove Lemma \ref{Claim1}, the same $L^p$ boundedness properties will hold for \newline $ \mathcal{V}_\rho(\{t^\beta\mathcal{D}_t^\beta \hat{P}_{t,glob}^\alpha\}_{t>0}).$
   \begin{proof}
   Suppose that $g:(0,\infty)\to\C$ is differentiable on $(0,\infty)$. Then,
   \begin{align}\label{normE}
   \|g\|_{E_\rho}&=\sup_{\substack{0<t_k<\dots<t_1\\ k\in\N}}\left(\sum_{j=1}^{k-1}|g(t_j)-g(t_{j+1})|^\rho \right)^{1/\rho}=\sup_{\substack{0<t_k<\dots<t_1 \nonumber\\ k\in\N}}\left(\sum_{j=1}^{k-1}\left|\int_{t_{j+1}}^{t_j}g'(t)dt\right|^\rho \right)^{1/\rho}\\
   &\le \sup_{\substack{0<t_k<\dots<t_1\\ k\in\N}}\sum_{j=1}^{k}\int_{t_{j+1}}^{t_j}|g'(t)|dt\le \int_0^\infty |g'(t)|dt.
   \end{align}
   Assume that $m\in\N$ such that $m-1\le \beta<m.$ We have that
   \begin{align*}
       \partial_t^m\hat{P}_t^\alpha (f)(x)&=\partial_t^m\left(\frac{t}{2\sqrt{\pi}}\int_0^\infty \frac{e^{-\frac{t^2}{4u}}}{u^{3/2}}\hat{W}_{u}^\alpha(f)(x)du\right) =\partial_t^m\left(\frac{1}{\sqrt{\pi}}\int_0^\infty \frac{e^{-v}}{v^{1/2}}\hat{W}_{\frac{t^2}{4v}}^\alpha(f)(x)dv\right)\\
       & =\partial_t^{m-1}\left(\frac{1}{\sqrt{\pi}}\int_0^\infty \frac{e^{-v}}{v^{1/2}}\frac{t}{2v}\partial_w\hat{W}_w^\alpha(f)(x)|_{w=\frac{t^2}{4v}}dv\right)\\
        &       =\partial_t^{m-1}\left(\frac{1}{\sqrt{\pi}}\int_0^\infty \frac{e^{-\frac{t^2}{4w}}}{w^{1/2}}\partial_w\hat{W}_w^\alpha(f)(x)dw\right) 
       ,\quad x\in (0,\infty)^n, \: t>0. 
   \end{align*}
   We get
   \begin{align*}
       t^\beta  \mathcal{D}_t^\beta &\hat{P}_{t,glob}^\alpha(f)(x)=\frac{e^{i\pi(m-\beta)}}{\Gamma(m-\beta)}t^\beta\int_0^\infty \partial_t^m\hat{P}_{t+s,glob}^\alpha(f)(x) s^{m-\beta-1}ds\\
    &=\frac{e^{i\pi(m-\beta)}}{\sqrt{\pi}\Gamma(m-\beta)}t^\beta\int_0^\infty\int_0^\infty \frac{\partial_t^{m-1}\left( e^{-\frac{(t+s)^2}{4u}}\right) }{u^{1/2}}\partial_u\hat{W}_{u,glob}^\alpha(f)(x)du \;s^{m-\beta-1}ds, \\&\qquad\quad \hspace{8cm}x\in (0,\infty)^n, \: t>0.
   \end{align*}
   By using Minkowski's inequality and \eqref{normE}, we deduce that
    \begin{align*}
        \mathcal{V}_\rho(&\{t^\beta\mathcal{D}_t^\beta \hat{P}_{t,glob}^\alpha\}_{t>0})(f)(x)\\&\le C\int_0^\infty\int_0^\infty\|t^\beta \partial_t^{m-1}e^{-\frac{(t+s)^2}{4u}} \|_{E_\rho}\frac{|\partial_u\hat{W}_{u,glob}^\alpha(f)(x)|}{\sqrt{u}}du\; s^{m-\beta-1}ds\\
        &\le C \int_0^\infty\int_0^\infty\int_0^\infty |\partial_t(t^\beta\partial_t^{m-1}e^{-\frac{(t+s)^2}{4u}})|dt\frac{1}{\sqrt{u}}|\partial_u\hat{W}_{u,glob}^\alpha(f)(x)|du \;s^{m-\beta-1}ds,\:\\ &\qquad  \hspace{10cm}x\in (0,\infty)^n.
    \end{align*}
    According to \cite[Lemma 4]{BCCFR}, we get
    \begin{align*}
        |\partial_t(t^\beta\partial_t^{m-1}e^{-\frac{(t+s)^2}{4u}})|&\le \beta |t^{\beta-1}\partial_t^{m-1}e^{-\frac{(t+s)^2}{4u}}|+ t^\beta|\partial_t^{m}e^{-\frac{(t+s)^2}{4u}}|\\
        &\le C( t^{\beta-1}u^{\frac{1-m}{2}}+t^\beta u^{-m/2}) e^{-\frac{(t+s)^2}{8u}}, \quad t,s,u\in (0,\infty).
    \end{align*}
By using that
$$
\int_0^\infty e^{-\frac{s^2}{8u}}s^{m-\beta-1}ds\le C u^{\frac{m-\beta}{2}}, \quad u\in (0,\infty),
$$
we get
\begin{align*}
     \mathcal{V}_\rho(\{t^\beta\mathcal{D}_t^\beta &\hat{P}_{t,glob}^\alpha\}_{t>0})(f)(x)\\&\le C\int_0^\infty\int_0^\infty|\partial_u\hat{W}_{u,glob}^\alpha(f)(x)|( t^{\beta-1}u^{-\beta/2}+t^\beta u^{-{\frac{1+\beta}{2}}})e^{-\frac{t^2}{8u}}du\;dt \\
     &\le C\int_0^\infty|\partial_u\hat{W}_{u,glob}^\alpha(f)(x)|du,\qquad x\in (0,\infty)^n.
\end{align*}
We recall that
\begin{small}
\begin{align*}
&\hat{W}_{u,glob}^\alpha(f)(x)=\int_{(0,\infty)^n} {\hat{W}_{u,glob}^\alpha(x,y)}f(y)d\nu_\alpha(y)
\\&=\int_{(0,\infty)^n}f(y)(1-e^{-u})^{-n-\sum_{i=1}^n\alpha_i}\int_{(-1,1)^n}\exp\left(-\frac{q_{{-}}(e^{-u/2}x,y,s)}{1-e^{-u}}\right)(1-\varphi(x,y,s))\Pi_\alpha (s)ds\;dm_\alpha(y),     
\end{align*}
\end{small}
where $x\in (0,\infty)^n$.

For every $x,y\in (0,\infty)^n$ and $s\in (-1,1)^n$, we define
$$
g_{x,y,s}(u)=(1-e^{-u})^{-n-\sum_{i=1}^n\alpha_i}\exp\left(-\frac{q_{{-}}(e^{-u/2}x,y,s)}{1-e^{-u}}\right), \quad u\in (0,\infty).
$$
Thus, for every $x\in (0,\infty)^n$, we obtain that 
\begin{align*}
     \mathcal{V}_\rho(\{t^\beta\mathcal{D}_t^\beta &\hat{P}_{t,glob}^\alpha\}_{t>0})(f)(x)\\&\le \int_{(0,\infty)^n}|f(y)|\int_{(-1,1)^n}\int_0^\infty|g_{x,y,s}'(u)|du\;|1-\varphi(x,y,s)|\Pi_\alpha (s)ds\;dm_\alpha(y).
\end{align*}
Let $x=(x_1,\dots,x_n),$ $y=(y_1,\dots,y_n)\in(0,\infty)^n$ and $s=(s_1,\dots,s_n)\in (-1,1)^n.$
We have that 
\begin{align*}
    &g_{x,y,s}'(u)=\big[-(n+\sum_{i=1}^n\alpha_i)(1-e^{-u})^{-n-1-\sum_{i=1}^n\alpha_i}e^{-u}\\
    &\;\:+(1-e^{-u})^{-n-1-\sum_{i=1}^n\alpha_i}(e^{-u}|x|^2{-}e^{-u/2}\sum_{i=1}^n x_iy_is_i)\\
    &\;\:+(1-e^{-u})^{-n-2-\sum_{i=1}^n\alpha_i}e^{-u}(e^{-u}|x|^2+|y|^2-2e^{-u/2}\sum_{i=1}^n x_iy_is_i)\big]\exp\left(-\frac{q_{{-}}(e^{-u/2}x,y,s)}{1-e^{-u}}\right)\\
   &= \frac{\exp\left(-\frac{q_{{-}}(e^{-u/2}x,y,s)}{1-e^{-u}}\right)}{(1-e^{-u})^{n+2+\sum_{i=1}^n\alpha_i}}\big[-(n+\sum_{i=1}^n\alpha_i)e^{-u}(1-e^{-u})+(1-e^{-u})(e^{-u}|x|^2{-}e^{-u/2}\sum_{i=1}^n x_iy_is_i)\\
    &\;\:+e^{-u}(e^{-u}|x|^2+|y|^2-2e^{-u/2}\sum_{i=1}^n x_iy_is_i)\big], \quad u\in (0,\infty).
\end{align*}

Observe that we can write
$$
g_{x,y,s}'(u)=F_{x,y,s}(u)P_{x,y,s}(e^{-u/2}),\quad u\in (0,\infty),
$$
where $F_{x,y,s}$ is a continuous positive function in $(0,\infty)$ and $P_{x,y,s}$ is a polynomial whose degree is at most four. Then, the sign of $g_{x,y,s}'(u)$ changes at most four times and we get that
\begin{equation}\label{g'}
\int_0^\infty|g_{x,y,s}'(u)|du\le C\sup_{u\in (0,\infty)}|g_{x,y,s}(u)|,
\end{equation}
where $C$ does not depend on $(x,y,s).$

We consider the function
$$
h_{x,y,s}(t)=g_{x,y,s}(-\log(1-t)), \quad t\in (0,1).
$$
We can write
\begin{align*}
h_{x,y,s}(t)&=\frac{1}{t^{n+\sum_{i=1}^n\alpha_i}}\exp\left(-\frac{(1-t)|x|^2+|y|^2-2\sum_{i=1}^nx_iy_is_i\sqrt{1-t}}{t}\right)=\frac{e^{-v(t)}}{t^{n+\sum_{i=1}^n\alpha_i}}, \:\; t\in (0,1),
\end{align*}
where $$v(t)=\frac{a}{t}-\frac{\sqrt{1-t}}{t}b-|x|^2, \quad t\in (0,1),$$  and being $a=|x|^2+|y|^2$ and $b=2\sum_{i=1}^n x_iy_is_i$.

We have that
$$
h_{x,y,s}'(t)=\frac{e^{-v(t)}}{t^{n+\sum_{i=1}^n\alpha_i}}\frac{2(a-(n+\sum_{i=1}^n\alpha_i)t)\sqrt{1-t}+b(t-2)}{2t^2\sqrt{1-t}},\quad t\in (0,1).
$$
Suppose that $\sum_{i=1}^n s_ix_iy_i\le 0$. If follows that, when $(x,y,s)\not\in N_1,$ 
\begin{align*}
a&\ge \frac{1}{2}(|x|^2+|y|^2-2\sum_{i=1}^nx_iy_is_i)=\frac{1}{2}q_{-}(x,y,s)\ge \frac{C_0^2}{2(1+|x|+|y|)^2}\\
&\ge\left\{ \begin{array}{cc}
C_0^2/8,     &  |x|+|y|<1,\\
    \frac{C_0^2}{8(|x|+|y|)^2}\ge \frac{C_0^2}{32a}, & |x|+|y|\ge 1. 
\end{array}\right.
\end{align*} 
Then, when $(x,y,s)\not\in N_1,$
$$
a\ge \left\{ \begin{array}{cc}
C_0^2/8,     &  |x|+|y|<1,\\
     \frac{C_0}{4\sqrt{2}}, & |x|+|y|\ge 1. 
\end{array}\right.
$$
Since $C_0>8(n+\sum_{i=1}^n\alpha_i),$ we get that
$$
a-(n+\sum_{i=1}^n\alpha_i)t\ge (n+\sum_{i=1}^n\alpha_i)(1-t)\ge 0, \:\:t\in (0,1),
$$
provided that  $(x,y,s)\not\in N_1.$

We conclude that $h'_{x,y,s}(t)\ge 0$, $t\in (0,1)$, when $(x,y,s)\not\in N_1$. It follows that
$$
\sup_{t\in (0,1)}|h_{x,y,s}(t)|\red{=h_{x,y,s}(1)=} e^{-|y|^2}, \:\;\text{ when }(x,y,s)\not\in N_1.
$$
Assume now that $\sum_{i=1}^n s_ix_iy_i\red{>}0$. If $(x,y,s)\not\in N_1$, we obtain that
\begin{align*}
    \sqrt{a^2-b^2}&=\sqrt{(|x|^2+|y|^2)^2-4(\sum_{i=1}^n s_ix_iy_i)^2}=\sqrt{q_+(x,y,s)q_-(x,y,s)}\\
    &\ge \left\{\begin{array}{cc}
     \frac{ C_0^2}{(1+|x|+|y|)^2}\ge \frac{C_0^2}{4}, &  |x|+|y|< 1,  \\  C_0\frac{\sqrt{|x|^2+|y|^2}}{1+|x|+|y|}\ge \frac{ C_0}{3}, &  |x|+|y|\ge 1
          \end{array}\right.
\end{align*}
Since $\min\{ C_0/3,C_0^2/4\}>2(n+\sum_{i=1}^n\alpha_i)$, we have that $a>\red{2\left(n+\sum_{i=1}^n\alpha_i\right)}$ and the equation
$$
2(a-(n+\sum_{i=1}^n\alpha_i)t)\sqrt{1-t}=b(2-t)
$$
has a unique solution in $(0,1)$, provided that $(x,y,s)\not\in N_1$.
{Indeed, let $(x,y,s)\not\in N_1$. We consider the functions
\blue{$$
g(t)=\frac{b(2-t)}{2(a-(n+\sum_{i=1}^n\alpha_i)t)}=\frac{b}{{2}(n+\sum_{i=1}^n\alpha_i)}+\frac{b(2(n+\sum_{i=1}^n\alpha_i)-a)}{{2}(n+\sum_{i=1}^n\alpha_i)(a-(n+\sum_{i=1}^n\alpha_i)t)}, \, t\in [0,1]
$$ \,\text{ and } $$h(t)=\sqrt{1-t},\quad t\in [0,1].
$$ }
We have that
$$
g'(t)=\frac{b[2(n+\sum_{i=1}^n\alpha_i)-a]}{2(a-(n+\sum_{i=1}^n\alpha_i)t)^2}< 0, \quad t\in [0,1],
$$
and
$$
g''(t)=\frac{b[2(n+\sum_{i=1}^n\alpha_i)-a](n+\sum_{i=1}^n\alpha_i)}{(a-(n+\sum_{i=1}^n\alpha_i)t)^3}< 0, \quad t\in [0,1].
$$
 Furthermore, $g(0)=\frac{b}{a}<1=h(0)$ and $g(1)=\frac{b}{2(a-(n+\sum_{i=1}^n\alpha_i))}>0=h(1)$. } \blue{We also have that $g'(t)\neq h'(t)$, $t\in (0,1)$.  Indeed, for  $t\in (0,1)$, $g'(t)=h'(t)$ if, and only if, $g_1(t)=h_1(t)$, being
 $$
 h_1(t)=b\left(a-2\left(n+\sum_{i=1}^n\alpha_i\right)\right)\sqrt{1-t} \quad \text{ and } g_1(t)=\left(a-\left(n+\sum_{i=1}^n\alpha_i\right)t\right)^2.
 $$
$h_1$ and $g_1$ are strictly decreasing in $(0,1)$ and
\begin{align*}
g_1(1)&=\left(a-\left(n+\sum_{i=1}^n\alpha_i\right)\right)^2=a\left(a-2\left(n+\sum_{i=1}^n\alpha_i\right)\right)+\left(n+\sum_{i=1}^n\alpha_i\right)^2>b\left(a-2\left(n+\sum_{i=1}^n\alpha_i\right)\right)\\
&=h_1(0).
\end{align*}
It follows that $g_1(t)\neq h_1(t)$, $t\in (0,1)$, and we conclude that $g'(t)\neq h'(t)$, $t\in (0,1)$.
Therefore, we deduce that there exists a unique $t_n\in (0,1)$ such that $g(t_n)=h(t_n)$.}
By using the procedure developed in \cite[page {850}]{MPS}, we deduce that
\begin{equation}\label{suph}
\sup_{t\in (0,1)}h_{x,y,s}(t)\le C \left(\frac{q_+(x,y,s)}{q_-(x ,y,s)}\right)^{(n+\sum_{i=1}^n\alpha_i)/2}\exp \left(-\frac{|y|^2-|x|^2}{2}-\frac{\red{\sqrt{q_+(x,y,s)q_-(x,y,s)}}}{2}\right),
\end{equation}
when $(x,y,s)\not\in N_1$. \red{Indeed, let $(x,y,s)\not\in N_1$. Observe that $v'(t)=0$ if, and only if,
$$
\frac{b(2-t)}{2a}=\sqrt{1-t}.
$$
As above, we can see that the last equation has a unique positive solution $t_0=2\frac{\sqrt{a^2-b^2}}{a+\sqrt{a^2-b^2}}$ and at that point $v(t)$ attains its minimum. It is clear that $t_n\le t_0.$ Furthermore, from \cite[page 850]{MPS} we know that $\frac{1}{4}t_0\le t_n$. 
We deduce that  
$$
h_{x,y,s}(t_0)\le h_{x,y,s}(t_n)= \sup_{t\in (0,1)}h_{x,y,s}(t)\le C\frac{e^{-v(t_0)}}{t_0^{n/2}}=C h_{x,y,s}(t_0).
$$
Thus, by using that
$$
t_0\simeq \frac{\sqrt{a^2-b^2}}{a}\simeq \frac{\sqrt{a-b}}{\sqrt{a+b}}=\sqrt{\frac{q_-(x,y,s)}{q_+(x,y,s)}}
$$
and
$$
v(t_0)=\frac{\sqrt{a^2-b^2}}{2}+\frac{a}{2}-|x|^2=\frac{|y|^2-|x|^2}{2}+\frac{\sqrt{q_+(x,y,s)q_-(x,y,s)}}{2},
$$
it is shown that \eqref{suph} holds.}

We have proved that, by taking into account our choice of $C_0$, for every $(x,y,s)\not\in N_1$,
\begin{align*}
\sup_{u\in (0,\infty)}|g_{x,y,s}(u)|&\le C\left\{ \begin{array}{cc}
e^{-|y|^2}, & \sum_{i=1}^n s_ix_iy_i<0, \\
\left(\frac{q_+(x,y,s)}{q_-(x ,y,s)}\right)^{(n+\sum_{i=1}^n\alpha_i)/2}e^{-\frac{|y|^2-|x|^2}{2}-\frac{\red{\sqrt{q_+(x,y,s)q_-(x,y,s)}}}{2}},  & \sum_{i=1}^n s_ix_iy_i\ge  0,
\end{array}
\right.\\
&\red{=:}K_\alpha(x,y,s).
\end{align*}
According to the proof of {\cite[Proposition 3.1]{Sa1}}, the operator $\mathbb{K}_\alpha$ defined by
$$
\mathbb{K}_\alpha(f)(x)=\int_{(0,\infty)^n}\int_{(-1,1)^n}K_\alpha(x,y,s) (1-\varphi(x,y,s))\Pi_\alpha(s)ds f(y)dm_\alpha(y), \:\: x\in (0,\infty)^n,
$$
is bounded from $L^1((0,\infty)^n,\nu_\alpha)$ into $L^{1,\infty}((0,\infty)^n,\nu_\alpha)$.

By \cite[Proposition 3]{Sa2}, it follows that $\mathbb{K}_\alpha$ is bounded from $L^2((0,\infty)^n,\nu_\alpha)$ into itself. The same arguments developed in the proof of  \cite[Proposition 3]{Sa2} allow us to see that $\mathbb{K}_\alpha$ is bounded from $L^q((0,\infty)^n,\nu_\alpha)$ into itself, for every $1<q<\infty.$ In order to prove this last property we can also proceed by using interpolation and duality. 
   \end{proof}

   \vspace{0.3cm}
   In the second lemma we shall study the operator $\mathcal{V}_\rho(\{t^\beta\mathcal{D}_t^\beta \hat{P}_{t,loc}^\alpha\}_{t>0}).$ We shall consider this operator as a vector-valued singular integral operator.

%
   We define, for every $t>0,$
   $$
   Q_t^\alpha(x,y,s)=\frac{t}{2\sqrt{\pi}}\int_0^\infty \frac{e^{-\frac{t^2}{4u}}}{u^{3/2}}(1-e^{-u})^{-n-\sum_{i=1}^n\alpha_i}\exp\left(-\frac{q_{{-}}(e^{-u/2}x,y,s)}{1-e^{-u}}\right)du 
  $$
   with $x,y\in (0,\infty)^n,$ $s\in (-1,1)^n.$ 
   Then,
   \begin{align*}
   t^\beta\mathcal{D}_t^\beta \hat{P}_{t,loc}^\alpha (f)(x)=\int_{(0,\infty)^n}\int_{(-1,1)^n}t^\beta\mathcal{D}_t^\beta Q_t^\alpha(x,y,s)&\varphi(x,y,s)\Pi_\alpha(s)\;ds\;f(y)dm_\alpha(y), \\&\:\;\qquad x\in (0,\infty)^n, \; t>0.
   \end{align*}
   
   \begin{lemma}\label{Claim2}
   Let $\tau>0$. For every $(x,y,s)\in N_\tau$,
   \begin{equation}\label{eq2.1}
   \|t^\beta\mathcal{D}_t^\beta Q_t^\alpha(x,y,s)\|_{E_\rho}\le C q_{{-}}(x,y,s)^{-n-\sum_{i=1}^n\alpha_i}
   \end{equation}
   and
   \begin{equation}\label{eq2.2}
   \|\nabla_xt^\beta\mathcal{D}_t^\beta Q_t^\alpha(x,y,s)\|_{E_\rho}+ \|\nabla_yt^\beta\mathcal{D}_t^\beta Q_t^\alpha(x,y,s)\|_{E_\rho}\le C  q_{{-}}(x,y,s)^{-n-1/2-\sum_{i=1}^n\alpha_i}. 
   \end{equation}
   \end{lemma}
   
\begin{proof}
 We are going to prove \eqref{eq2.1}.
By keeping the same notation as in the proof of Lemma \ref{Claim1}, we obtain that
\begin{align*}
 \|t^\beta\mathcal{D}_t^\beta Q_t^\alpha(x,y,s)\|_{E_\rho}&\le C \int_0^\infty|g'_{x,y,s}(u)|du\\&\le C\sup_{u\in (0,\infty)}|g_{x,y,s}(u)|, \quad x,y\in (0,\infty)^n \text{ and } s\in (-1,1)^n.
\end{align*}
\red{According to \cite[(2.6)]{Sa2},}  $q_{{-}}(e^{-u/2}x,y,s)\ge q_{{-}}(x,y,s)-C(1-e^{-u/2}),$ $(x,y,s)\in N_\tau$ and $u\in (0,\infty)$. \red{I}t follows that
\begin{align*}
|g_{x,y,s}(u)|&\le C (1-e^{-u})^{-n-\sum_{i=1}^n\alpha_i}\exp\left(-\frac{q_{{-}}(x,y,s)}{1-e^{-u}}\right)\\&\le\frac{ C}{q_-(x,y,s)^{n+\sum_{i=1}^n\alpha_i}}, \:\:(x,y,s)\in N_\tau \text{ and }u\in (0,\infty).
\end{align*}
Hence, we obtain that 
$$
\|t^\beta\mathcal{D}_t^\beta Q_t^\alpha(x,y,s)\|_{E_\rho}\le \frac{ C}{q_-(x,y,s)^{n+\sum_{i=1}^n\alpha_i}}, \:\:(x,y,s)\in N_\tau.
$$
In order to prove \eqref{eq2.2}, it is sufficient to see that
$$
\|t^\beta\mathcal{D}_t^\beta \partial_{x_i}Q_t^\alpha(x,y,s)\|_{E_\rho}\le \frac{ C}{q_-(x,y,s)^{n+1/2+\sum_{i=1}^n\alpha_i}}, \:\:(x,y,s)\in N_\tau,
$$
for every $i=1,\dots,n$.

Let $i=1,\dots,n$. We have that
$$
\partial_{x_i}Q_t^\alpha(x,y,s)=\frac{-{t}}{2\sqrt{\pi}}\int_0^\infty\frac{e^{-\frac{t^2}{4u}}}{u^{3/2}}\frac{2e^{-u}x_i-2e^{-u/2}y_is_i}{(1-e^{-u})^{n+1+\sum_{i=1}^n\alpha_i}}\exp\left(-\frac{q_{{-}}(e^{-u/2}x,y,s)}{1-e^{-u}}\right)du,
$$
$x,y\in (0,\infty)^n$, $s\in (-1,1)^n$ and $t>0$.

Since $|e^{-u/2}x_i-y_is_i|\le \sqrt{q_-(e^{-u/2}x,y,s)}$, $x=(x_1,\dots,x_n)$, $y=(y_1,\dots, y_n)\in (0,\infty)^n$, $s=(s_1,\dots,s_n)\in (-1,1)^n$ and $u\in (0,\infty)$, by proceeding as above we obtain that
\begin{align*}
\|t^\beta\mathcal{D}_t^\beta \partial_{x_i}Q_t^\alpha(x,y,s)\|_{E_\rho}&\le C \sup_{u\in (0,\infty)}\left|\frac{e^{-u}x_i -e^{-u/2}y_is_i}{(1-e^{-u})^{n+1+\sum_{i=1}^n\alpha_i}}\exp\left(-\frac{q_{{-}}(e^{-u/2}x,y,s)}{1-e^{-u}}\right)\right|\\
&\le \frac{ C}{q_-(x,y,s)^{n+1/2+\sum_{i=1}^n\alpha_i}}, \:\:(x,y,s)\in N_\tau.
 \end{align*}
	\end{proof}

   Note that the measure $\nu_\alpha$ is not doubling in $((0,\infty)^n,|\cdot|)$, where $|\cdot|$ denotes the Euclidean norm in $(0,\infty)^n$. Hence, $((0,\infty)^n,|\cdot|,\nu_\alpha)$ is not a space of homogeneous type in the Coifman-Weiss sense, see \cite{CoW}. However, $((0,\infty)^n,|\cdot|,m_\alpha)$ is a space of homogeneous type.  
   
   We define
   $$
   H_t^\alpha(x,y)=\int_{(-1,1)^n}t^\beta\mathcal{D}_t^\beta Q_t^\alpha(x,y,s)\varphi(x,y,s)\Pi_\alpha(s) ds,\quad x,y\in(0,\infty)^n, \: t>0.
   $$
   According to \cite[Lemma 3.1]{BCN}, from \eqref{eq2.1}  we deduce that 
   \begin{align}\label{eq2.3}
       \| H_t^\alpha(x,y)\|_{E_\rho}&\le \int_{(-1,1)^n}\|t^\beta\mathcal{D}_t^\beta Q_t^\alpha(x,y,s) \|_{E_\rho} |\varphi(x,y,s)|\Pi_\alpha(s)ds\nonumber\\
       &\le C \int_{(-1,1)^n}q_{{-}}(x,y,s)^{-n-\sum_{i=1}^n\alpha_i} |\varphi(x,y,s)|\Pi_\alpha(s)ds\nonumber\\
       &\le \frac{C}{m_\alpha(B(x,|x-y|))},\quad x,y\in (0,\infty)^n, \: x\neq y,
   \end{align}
   and, from \eqref{2.0} and \eqref{eq2.2},
    \begin{align}\label{eq2.4}
    \|\nabla_xH_t^\alpha(x,y)\|_{E_\rho}+ \|\nabla_yH_t^\alpha(x,y)\|_{E_\rho}&\le C \int_{(-1,1)^n}q_{{-}}(x,y,s)^{-n-1/2-\sum_{i=1}^n\alpha_i} \Pi_\alpha(s)ds\nonumber\\
       &\le \frac{C}{|x-y|m_\alpha(B(x,|x-y|))},\:\:x,y\in (0,\infty)^n, \: x\neq y.
    \end{align}

  Let $N\in\N$. If $g$ is a  complex function defined on $[\frac{1}{N},N]$, we say that $g\in E_{\rho,N}$ when
$$
\|g\|_{E_{\rho,N}}:=\sup_{\substack{\frac1{N}\le t_k<\dots<t_1\le N\\k\in\N}}\left(\sum_{j=1}^{k-1}|g(t_j)-g(t_{j+1})|^\rho\right)^{1/\rho}\red{<\infty}.
$$
  By identifying those functions whose difference is constant in $[\frac{1}{N},N]$,  $(E_{\rho,N}, \|\cdot\|_{E_{\rho,N}})$ is a Banach space. 
  
  Suppose that $F:(0,\infty)^n\to E_{\rho,N}$ is continuous. We shall prove that $F$ is strongly measurable. It is clear that $F$ is weakly continuous and hence weakly measurable. According to Petti's Theorem {(\cite[p. 131]{Yo})}, in order to prove that $F$ is strongly measurable, it is sufficient to see that $F((0,\infty)^n)\subseteq\mathcal{B}$, where $\mathcal{B}$ is a separable Banach space. Since $F$ is continuous, we have that $F((0,\infty)^n)\subseteq \overline{F(\mathbb{Q}_+^n)}^{E_{\rho,N}}$, where $\mathbb{Q_+}:=\mathbb{Q}\cap(0,\infty),$ being $\mathbb{Q}$ the set of rational numbers. We denote by $\text{span}_\mathbb{C}F(\mathbb{Q}_+^n)$ (respectively $\text{span}_{\mathbb{Q}+i\mathbb{Q}}F(\mathbb{Q}_+^n)$) the set of all finite linear combinations of elements in $F(\mathbb{Q}_+^n)$ with coefficients  in $\mathbb{C}$ (respectively, in $\mathbb{Q}+i\mathbb{Q}$). We can write 
  $$
  F((0,\infty)^n)\subset \overline{\text{span}_{\mathbb{Q}+i\mathbb{Q}}F(\mathbb{Q}_+^n)}^{E_{\rho,N}}=\overline{\text{span}_{\mathbb{C}}F(\mathbb{Q}_+^n)}^{E_{\rho,N}}.
  $$
    The set $\text{span}_{\mathbb{Q}+i\mathbb{Q}}F(\mathbb{Q}_+^n)$ is numerable. Hence, $\overline{\text{span}_{\mathbb{C}}F(\mathbb{Q}_+^n)}^{E_{\rho,N}}$ is a separable Banach subspace of $E_{\rho,N}$. We conclude that $F$ is strongly measurable.

    Let $f\in C^\infty_c((0,\infty)^n)$ and $x\in (0,\infty)^n\setminus\supp f$. We consider the function
    $$
    F_x(y)=H^\alpha(x,y)f(y), \quad y\in (0,\infty)^n,
    $$
    where 
    $$
    [H^\alpha(x,y)](t)=H_t^\alpha(x,y), \quad x,y\in (0,\infty)^n, \;\: x\neq y,\:\text{ and } t>0.
    $$
    We are going to see that $F_x$ is continuous from $(0,\infty)^n $ into $E_{\rho,N}.$ Let $y_0\in  (0,\infty)^n.$ If $y_0\not\in\supp f$, then there exists $r_0>0$ such that $F_x(y)=0,$ $y\in B(y_0,r_0)$ and $F_x$ is continuous \red{at} $y_0.$ On the other hand, we have that
    $$
    \|F_x(y)-F_x(y_0)\|_{E_{\rho,N}}\le \int_{1/N}^N|\partial_t[H_t^\alpha(x,y)f(y)-H_t^\alpha(x,y_0)f(y_0)]|dt, \:\: \:y\in (0,\infty)^n.
    $$
    Since the function $\partial_t[H_t^\alpha(x,y)f(y)]$, $(t,y)\in (0,\infty)\times (0,\infty)^n$ is continuous and $[\frac{1}{N},N]$ is compact in $(0,\infty)$, we conclude that
    $$
    \lim_{y\to y_0}F_x(y)=F_x(y_0), \text{ in } E_{\rho,N}.
    $$
    Thus, we prove that $F_x$ is continuous.

    We define
    $$
    \mathcal{H}^\alpha_N(f)(x)=\int_{(0,\infty)^n}H^\alpha(x,y)f(y)dm_\alpha(y), \quad x\not\in \supp f,
    $$
    where the integral is understood in the $E_{\rho,N}$-Bochner sense. According to \eqref{eq2.3}, the integral defining $\mathcal{H}_N^\alpha(f)(x)$ is $\| \cdot\|_{E_{\rho,N}}$-norm convergent, for every $x\not\in\supp f$.
    
    We also define 
    $$
    \mathbb{H}^\alpha (f)(x,t)=\int_{(0,\infty)^n}H_t^\alpha(x,y)f(y)dm_\alpha(y), \:\: x\in (0,\infty)^n \text{ and } t>0.
    $$
    We have that
    \begin{equation}\label{eq2.5}
    \|  \mathcal{H}_N^\alpha (f)(x)- \mathbb{H}^\alpha (f)(x,\cdot)\|_{E_{\rho,N}}=0, \quad x\not\in\supp f.
    \end{equation}
    Indeed, let $a,b\in (\frac{1}{N},N)$. We consider
    $$
    L_{a,b}(g)=g(b)-g(a),\quad g\in E_{\rho,N}.
    $$
    Since $|L_{a,b}(g)|\le \|g\|_{E_{\rho,N}},$ $g\in E_{\rho,N}$, $L_{a,b}$ is in $E_{\rho,N}'$, the dual space of $E_{\rho,N}.$ Then,
    \begin{align*}
        L_{a,b}(\mathcal{H}^\alpha_{\red{N}}(f)(x))&=\int_{(0,\infty)^n}L_{a,b}(H^\alpha(x,y))f(y)dm_\alpha(y)\\
        &=\mathbb{H}^\alpha(f)(x,\red{b})-\mathbb{H}^\alpha(f)(x,\red{a}), \quad\:\:\: x\not\in\supp f.
    \end{align*}
    Thus, \eqref{eq2.5} is proved.

    

Recall that  $  \mathcal{V}_\rho(\{t^\beta\mathcal{D}_t^\beta \hat{P}_t^\alpha\}_{t>0})$ is bounded from $L^2((0,\infty)^n,\nu_\alpha)$ into itself. By Lemma \ref{Claim1},  $  \mathcal{V}_\rho(\{t^\beta\mathcal{D}_t^\beta \hat{P}_{t,glob}^\alpha\}_{t>0})$ is also bounded from $L^2((0,\infty)^n,\nu_\alpha)$ into itself. 
Thus, 
we deduce that $\mathcal{V}_\rho(\{t^\beta\mathcal{D}_t^\beta \hat{P}_{t,loc}^\alpha\}_{t>0})$ is also bounded from $L^2((0,\infty)^n,\nu_\alpha)$ into itself.

This implies that the operator $\mathbb{H}^\alpha$ is bounded from 
$L^2((0,\infty)^n,\nu_\alpha)$ into $L^2_{E_{\rho,N}}((0,\infty)^n,\nu_\alpha)$.  Furthermore, we have that
\begin{equation}\red{\label{supHH}}
\sup_{N\in\N}\|\mathbb{H}^\alpha\|_{L^2((0,\infty)^n,\nu_\alpha)\to L^2_{E_{\rho,N}}((0,\infty)^n,\nu_\alpha)}<\infty.
\end{equation}

\begin{lemma}\label{Claim3}
Let $N\in\N$. The operator $\mathbb{H}^\alpha$ is bounded from 
$L^q((0,\infty)^n,\nu_\alpha)$ into $L^q_{E_{\rho,N}}((0,\infty)^n,\nu_\alpha)$, $1<q<\infty,$ (respectively from $L^1((0,\infty)^n,\nu_\alpha)$ into $L^{1,\infty}_{E_{\rho,N}}((0,\infty)^n,\nu_\alpha)$) if, and only if, $\mathbb{H}^\alpha$ is bounded from 
$L^q((0,\infty)^n,m_\alpha)$ into $L^q_{E_{\rho,N}}((0,\infty)^n,m_\alpha)$, $1<q<\infty,$ (respectively from $L^1((0,\infty)^n,m_\alpha)$ into $L^{1,\infty}_{E_{\rho,N}}((0,\infty)^n,m_\alpha)$).
\end{lemma}
\begin{proof}
Suppose that $\mathbb{H}^\alpha$ is bounded from $L^1(0,\infty)^n,m_\alpha)$ into $L^{1,\infty}_{E_{\rho,N}}(((0,\infty)^n,m_\alpha)$.

According to \cite[Lemma 4]{Sa3}, there exists a sequence of balls $\{B_j:=B(x_j,r_j)\}_{j\in\N}$ in $(0,\infty)^n$, being $x_j\in (0,\infty)^n$ and $r_j\in (0,1)$ for every $j\in\N$, and satisfying the following properties:
\begin{itemize}
	\item[(i)]$(0,\infty)^n=\cup_{j\in\N}B_j;$
	\item[(ii)] For every $\delta>1$, the sequence $\{\delta B_j\}_{j\in\N}$ has bounded overlap;
	\item[(iii)]There exists $C>1$ such that, for every measurable set $E\subset B_j$, $j\in\N,$  
	$$
	\frac{1}{C} e^{-x_j^2}m_\alpha(E)\le \nu_\alpha(E)\le C e^{-x_j^2} m_\alpha (E);
	$$
	\item[(iv)]There exists $\delta_0>1$ such that if $x\in B_j$ and $y\not\in \delta_0B_j$,  $j\in\N,$  then $(x,y,s)\not\in N_2$, $s\in (-1,1)^n$.
	\end{itemize}
We have that, for every $j\in\N$, 
$$
\chi_{B_j}(x)\mathbb{H}^\alpha(f)(x,t)=\chi_{B_j}(x)\mathbb{H}^\alpha(\chi_{\delta_0B_j}f)(x,t), \quad x\in (0,\infty)^n, \:\;t>0.
$$

Then, for every $f\in L^1((0,\infty)^n,\nu_\alpha)$,
\begin{align*}
\nu_\alpha(&\{x\in (0,\infty)^n: \: \; \|\mathbb{H}^\alpha(f)(x,\cdot)\|_{E_{\rho,N}}>\lambda\})\\
&
\le C\sum_{j\in\N}\nu_\alpha(\{ x\in B_j: \: \|\mathbb{H}^\alpha (f) (x,\cdot)\|_{E_{\rho,N}}>\lambda\})\\
&=C\sum_{j\in\N}\nu_\alpha(\{ x\in B_j: \: \|\mathbb{H}^\alpha (\chi_{\delta_0 B_j}f) (x,\cdot)\|_{E_{\rho,N}}>\lambda\})\\
&\le C\sum_{j\in\N}e^{-x_j^2}m_\alpha(\{ x\in B_j: \: \|\mathbb{H}^\alpha (\chi_{\delta_0 B_j}f) (x,\cdot)\|_{E_{\rho,N}}>\lambda\})\\
&\le C \sum_{j\in\N}e^{-x_j^2}\frac{\|\chi_{\delta_0B_j}f\|_{L^1((0,\infty)^n,m_\alpha)}}{\lambda}\le \frac{C}{\lambda}\int_{(0,\infty)^n}|f(x)|d\nu_\alpha(x),\quad \lambda>0.
	\end{align*}
	
	The other properties stated in Lemma \ref{Claim3} can be proved in a similar way.
	\end{proof}

From Lemma \ref{Claim3} \red{ and \eqref{supHH}}, we deduce that $\mathbb{H}^\alpha$ is bounded from 
$L^2((0,\infty)^n,m_\alpha)$ into \newline $L^2_{E_{\rho,N}}((0,\infty)^n,m_\alpha)$, for every $N\in\N$, and $$
\sup_{N\in\N}\|\mathbb{H}^\alpha\|_{L^2((0,\infty)^n,m_\alpha)\to L^2_{E_{\rho,N}}((0,\infty)^n,m_\alpha)}<\infty.
$$

Then, according to Calder\'on-Zygmund theory for vector valued singular integrals, see \cite{GLY} and \cite{RRT}, by taking into account \eqref{eq2.3}, \eqref{eq2.4} and \eqref{eq2.5}, it follows that $\mathbb{H}^\alpha$ is bounded from $L^1((0,\infty)^n,m_\alpha)$ into $L^{1,\infty}_{E_{\rho,N}}((0,\infty)^n,m_\alpha)$), for every $N\in\N$ and
$$
\sup_{N\in\N}\|\mathbb{H}^\alpha\|_{L^1((0,\infty)^n,m_\alpha)\to L^{1,\infty}_{E_{\rho,N}}((0,\infty)^n,m_\alpha)}<\infty.
$$
By using again Lemma \ref{Claim3}  we conclude that $\mathbb{H}^\alpha$ is bounded from $L^1((0,\infty)^n,\nu_\alpha)$ into $L^{1,\infty}_{E_{\rho,N}}((0,\infty)^n,\nu_\alpha)$, for every $N\in\N$ and
$$
\sup_{N\in\N}\|\mathbb{H}^\alpha\|_{L^1((0,\infty)^n,\nu_\alpha)\to L^{1,\infty}_{E_{\rho,N}}((0,\infty)^n,\nu_\alpha)}<\infty.
$$ 
By taking $N\to\infty$ and by using monotone convergence theorem we deduce that the operator $\mathcal{V}_\rho(\{t^\beta\mathcal{D}_t^\beta \hat{P}_{t,loc}^\alpha\}_{t>0})$ is bounded from $L^1((0,\infty)^n,\nu_\alpha)$ into $L^{1,\infty}((0,\infty)^n,\nu_\alpha)$). Thus, by Lemma \ref{Claim1} we deduce that $\mathcal{V}_\rho(\{t^\beta\mathcal{D}_t^\beta \hat{P}_{t}^\alpha\}_{t>0})$ is bounded from $L^1((0,\infty)^n,\nu_\alpha)$ into $L^{1,\infty}((0,\infty)^n,\nu_\alpha)$). 

\edproof
\subsection{Proof of Theorem \ref{teo1.1} ii)}

\textcolor{white}{}
\vspace{0.5cm}

    In order to prove Theorem \ref{teo1.1} for the oscillation operator $\OO(\{t^\beta \mathcal{D}_t^\beta P_t^\alpha\}_{t>0}, \{t_j\}_{j\in\N}))$, we can proceed as in the subsection 2.1 for the variation operator $\V_\rho(\{t^\beta \mathcal{D}_t^\beta P_t^\alpha\}_{t>0})$. We only comment some details.
    
    We define the space $F(\{t_j\}_{j\in\N})$ that consists of all the complex functions $g$ defined in $(0,\infty)$ such that
    $$
    \|g\|_{F(\{t_j\}_{j\in\N})}:=\left(\sum_{j=1}^\infty \sup_{t_{j+1}\le s_{j+1}<s_j\le t_j}|g(s_{j+1})-g(s_j)|^2 \right)^{1/2}<\infty.
    $$
    Note that $\|g\|_{F(\{t_j\}_{j\in\N})}=0$ if, and only if, $g$ is constant in $(0,t_1]$. By taking the corresponding quotient space $(F(\{t_j\}_{j\in\N}),\|\cdot\|_{F(\{t_j\}_{j\in\N})})$ can be seen as a Banach space.
    
    If $g\in F(\{t_j\}_{j\in\N})$ is a differentiable function in $(0,\infty)$, we have that
    \begin{align*}
    \|g\|_{F(\{t_j\}_{j\in\N})}&=\left( 
  \sum_{j=1}^\infty \sup_{t_{j+1}\le s_{j+1}<s_j\le t_j}\left|\int_{s_{j+1}}^{s_j}g'(t)dt\right|^2\right)^{1/2}\\
    &\le  \sum_{j=1}^\infty \int_{t_{j+1}}^{t_j}|g'(t)|dt\le \int_0^{t_1}|g'(t)|dt\le \int_0^\infty|g'(t)|dt.
    \end{align*}
    
    Since $\{W_t^\alpha\}_{t>0}$ is a symmetric diffusion semigroup in the Stein's sense, according to \cite[Theorem 3.3]{JR}, the oscillation operator $\OO(\{ W_t^\alpha\}_{t>0},\{t_j\}_{j\in\N})$ is bounded from $L^q((0,\infty)^n,\mu_\alpha)$ into itself and
    $$
    \sup_{\{s_j\}\downarrow 0}\|\OO(\{ W_t^\alpha\}_{t>0},\{s_j\}_{j\in\N})\|_{L^q((0,\infty)^n,\mu_\alpha)\rightarrow L^q((0,\infty)^n,\mu_\alpha)}<\infty,
    $$
    for every $1<q<\infty$.

    The rest of the proof follows as in subsection 2.1.
  
    \edproof

    \subsection{Proof of Theorem \ref{teo1.1} iii)}
    
    \textcolor{white}{}
    \vspace{0.5cm}

    According to \cite[page 6712]{JSW}, we have that
    $$
    \lambda\Lambda(\{t^\beta \mathcal{D}_t^\beta P_t^\alpha\}_{t>0},\lambda)^{1/\rho}(f)\le C\; \V_\rho (\{t^\beta \mathcal{D}_t^\beta P_t^\alpha\}_{t>0})(f), \quad \lambda>0.
    $$
    Here, $C>0$ does not depend on $\lambda>0$. Thus, the $L^p$-boundedness properties for the jump operators can be deduced from the corresponding $L^p$-boundedness properties for the variation operators.

    \edproof
    
    \subsection{Proof of Theorem \ref{teo1.1} iv)}
    
    \textcolor{white}{}
    \vspace{0.5cm}

    Let $k\in\Z$. We can write
    \begin{scriptsize}
    \begin{align*}
    &V_k(\{t^\beta \mathcal{D}_t^\beta P_t^\alpha\}_{t>0})(f)(x)\\&=\sup_{\substack{2^{-k}<t_l<t_{l-1}<\dots<t_1<2^{-k+1}\\ l\in\N}}\left(\sum_{j=1}^{l-1}\red{\Bigg|}t^\beta \mathcal{D}_t^\beta P_t^\alpha (f)(x)\big|_{t=t_j}-t^\beta \mathcal{D}_t^\beta P_t^\alpha (f)(x)\big|_{t=t_{j+1}} \red{\Bigg|}^2 \right)^{1/2}\\
    &\le\sup_{\substack{2^{-k}<t_l<t_{l-1}<\dots<t_1<2^{-k+1}\\ l\in\N}}\left(\sum_{j=1}^{l-1}{\left|\int_{t_{j+1}}^{t_j}\partial_t(t^\beta \mathcal{D}_t^\beta P_t^\alpha (f)\red{)}(x) dt\right|^2} \right)^{1/2}\le \int_{2^{-k}}^{2^{-k+1}}|\partial_t(t^\beta \mathcal{D}_t^\beta P_t^\alpha (f)\red{)}(x)|dt\\
     &\le \beta  \int_{2^{-k}}^{2^{-k+1}}t^{\beta-1} |\mathcal{D}_t^\beta P_t^\alpha (f)(x)|{dt}+\int_{2^{-k}}^{2^{-k+1}}t^{\beta}| \mathcal{D}_t^{\beta+1} P_t^\alpha (f)(x)|{dt}\\
    &\le   \left[\left(\beta \int_{2^{-k}}^{2^{-k+1}}|t^\beta \mathcal{D}_t^\beta P_t^\alpha (f)(x)|^2\frac{dt}{t}\right)^{1/2}+\left(\int_{2^{-k}}^{2^{-k+1}}|t^{\beta+1} \mathcal{D}_t^{\beta+1} P_t^\alpha (f)(x)|^2\frac{dt}{t}\right)^{1/2}\right] \left( \int_{2^{-k}}^{2^{-k+1}}\frac{dt}{t}\right)^{1/2}\\
    &\le \sqrt{\log 2}\left[\beta  \left(\int_{2^{-k}}^{2^{-k+1}}|t^\beta \mathcal{D}_t^\beta P_t^\alpha (f)(x)|^2\frac{dt}{t}\right)^{1/2} + \left(\int_{2^{-k}}^{2^{-k+1}}|t^{\beta+1} \mathcal{D}_t^{\beta+1} P_t^\alpha (f)(x)|^2\frac{dt}{t}\right)^{1/2}\right],
    \end{align*}
    \end{scriptsize}
  for every  $x\in (0,\infty)^n$.
    
    Then,
    \begin{align*}
    &\mathcal{S}_V(\{t^\beta \mathcal{D}_t^\beta P_t^\alpha\}_{t>0})(f)(x)=\left(\sum_{k\in\Z}(V_k(\{t^\beta \mathcal{D}_t^\beta P_t^\alpha\}_{t>0})(f)(x))^2\right)^{1/2}\\
    &\le C \left[ \left(\int_{0}^\infty|t^\beta \mathcal{D}_t^\beta P_t^\alpha (f)(x)|^2\frac{dt}{t}\right)^{1/2}+\left(\int_{0}^\infty|t^{\beta+1} \mathcal{D}_t^{\beta+1} P_t^\alpha
    (f)(x)|^2\frac{dt}{t}\right)^{1/2}\right], \: x\in (0,\infty)^n.
    \end{align*}
   Recall that, for every $\gamma>0$,  we define the Littlewood-Paley $g_\alpha^\gamma$-function as follows
   $$
   g_\alpha^\gamma(f)(x)=\left(\int_{0}^\infty|t^\gamma \mathcal{D}_t^\gamma P_t^\alpha (f)(x)|^2\frac{dt}{t}\right)^{1/2}, \quad x\in (0,\infty)^n.
   $$
   Therefore, the proof of Theorem \ref{teo1.1} iv) will be completed once we prove Theorem \ref{teo1.2}.

   \subsubsection{ Proof of Theorem 1.2.}
      
   \textcolor{white}{}
   \vspace{0.5cm}
   
   According to \cite[Proposition 3.1]{TZ}, if $0<\gamma_1\le \gamma_2$, there exists $C>0$ such that
   $$
   g_\alpha^{\gamma_1}(f)\le C g_\alpha^{\gamma_2}(f).
   $$
Therefore, it is sufficient to establish the $L^p$-boundedness properties for the Littlewood-Paley $g_\alpha^m$-function, for every $m\in\N$, $m\ge 1$.
   
   Let $m\in\N$, $m\ge 1$. Since $\{P_t^\alpha\}_{t>0}$ is a Stein symmetric diffusion semigroup on $((0,\infty)^n,\mu_\alpha)$, from \cite[Chapter 4, section 6, Corollary 1]{StLP} we deduce that $g_\alpha^m$ is bounded from  $L^p((0,\infty)^n,\mu_\alpha)$ into itself, for every $1<p<\infty$.

   We are going to see that $g_\alpha^m$ is bounded from $L^1((0,\infty)^n,\mu_\alpha)$ into $L^{1,\infty}((0,\infty)^n,\mu_\alpha)$. We employ a procedure similar to the one used in the proof of Theorem \ref{teo1.1} i). As in subsection \ref{subs2.1}, we consider the function $\Psi:(0,\infty)^n\to (0,\infty)^n$ defined by $\Psi(x)=x^2$, $x\in (0,\infty)^n$, and the measure $\nu_\alpha$ on $(0,\infty)^n$ given by
   $$
   \nu_\alpha(A)=\mu_\alpha(\Psi^{-1}(A)),
   $$
   for every Borel measurable set $A$ in $(0,\infty)^n$. We define
   $$
   \hat{g}_\alpha^m(f)(x)=\left( \int_0^\infty |t^m\partial_t^m \hat{P}_t^\alpha (f)(x)|^2\frac{dt}{t}\right)^{1/2}, \quad x\in (0,\infty)^n.
   $$
   Our objective is to see that $ \hat{g}_\alpha^m$ is bounded from $L^1((0,\infty)^n,\nu_\alpha)$ into $L^{1,\infty}((0,\infty)^n,\nu_\alpha)$.

   We consider the local and global operators associated with $\hat{g}_\alpha^m$ defined by
   $$
   \hat{g}_{\alpha,loc}^m(f)(x)=\left( \int_0^\infty |t^m\partial_t^m \hat{P}_{t,loc}^\alpha (f)(x)|^2\frac{dt}{t}\right)^{1/2}, \quad x\in (0,\infty)^n,
   $$
   and 
   $$
   \hat{g}_{\alpha,glob}^m(f)(x)=\left( \int_0^\infty |t^m\partial_t^m \hat{P}_{t,glob}^\alpha (f)(x)|^2\frac{dt}{t}\right)^{1/2}, \quad x\in (0,\infty)^n.
   $$
   
  It is clear that
  $$
   \hat{g}_{\alpha}^m(f)(x)\le  \hat{g}_{\alpha,loc}^m(f)(x)+ \hat{g}_{\alpha,glob}^m(f)(x), \quad x\in (0,\infty)^n.
  $$
   We will obtain our result when we prove the following Lemmas \ref{Claim2.1}-\ref{Claim2.3} (see subsection \ref{subs2.1}).
     \begin{lemma}\label{Claim2.1}
   	There exists a positive function $K_\alpha$ defined on $(0,\infty)^n\times (0,\infty)^n\times (-1,1)^n$ such that
   	\begin{small}
   		\begin{align*}
   	\hat{g}_{\alpha,glob}^m(f)(x)\le \int_{(0,\infty)^n}\int_{(-1,1)^n}K_\alpha(x,y,s)(1-\varphi(x,y,s))\Pi_\alpha(s)ds|f(y)|dm_\alpha(y),\\ \quad x\in (0,\infty)^n,
   		\end{align*}
   	\end{small}
 and satisfying that the operator $\mathbb{K}_\alpha$, defined by
   	$$
   	\mathbb{K}_\alpha(f)(x)=\int_{(0,\infty)^n}\int_{(-1,1)^n}K_\alpha(x,y,s)(1-\varphi(x,y,s))\Pi_\alpha(s)dsf(y)dm_\alpha(y),\:\: x\in (0,\infty)^n
   	$$
   	is bounded from $L^q((0,\infty)^n,\nu_\alpha)$ into itself, $1<q<\infty$, and from $L^1((0,\infty)^n,\nu_\alpha)$ into $L^{1,\infty}((0,\infty)^n,\nu_\alpha)$. 
   \end{lemma}
   
   \begin{proof}
   	We can write 
   	$$
   	\partial_t^m\hat{P}_t^\alpha(f)(x)=\frac{1}{\sqrt{\pi}}\int_0^\infty \frac{\partial_t^{m-1}(e^{-\frac{t^2}{4u}})}{u^{1/2}}\partial_u\hat{W}_u^\alpha(f)(x)du 
   	,\quad x\in (0,\infty)^n, \: t>0.
   	$$
   	Minkowski's integral inequality and \cite[Lemma {4}]{BCCFR} lead to
   	\begin{align*}
   	 	\hat{g}_{\alpha,glob}^m(f)(x)&\le C \int_0^\infty \left( \int_0^\infty|t^m\partial_t^{m-1}(e^{-\frac{t^2}{4u}})|^2\frac{dt}{t}\right)^{1/2}\frac{1}{\sqrt{u}}| \partial_u\hat{W}_{u,glob}^\alpha(f)(x)|du
   	 	\\
   	 	&\le C \int_0^\infty \left( \int_0^\infty\frac{e^{-\frac{t^2}{8u}}t^{2m-1}}{u^{m-1}}{dt}\right)^{1/2}\frac{1}{\sqrt{u}}| \partial_u\hat{W}_{u,glob}^\alpha(f)(x)|du\\
   	 	&\le  C \int_0^\infty| \partial_u\hat{W}_{u,glob}^\alpha(f)(x)|du ,\quad x\in (0,\infty)^n. 
   	\end{align*}
   	We can finish the proof as in the proof of Lemma \ref{Claim1} in subsection \ref{subs2.1}.
   \end{proof}

     \begin{lemma}\label{Claim2.2}
   	Let $\tau>0$. For every $(x,y,s)\in N_\tau,$
   	\begin{equation*}
   	\|t^m\partial_t^m Q_t^\alpha(x,y,s)\|_{L^2((0,\infty),\frac{dt}{t})}\le C q_{{-}}(x,y,s)^{-n-\sum_{i=1}^n\alpha_i}
   	\end{equation*}
   	and
   	\begin{align*}
   	\|\nabla_xt^m\partial_t^m Q_t^\alpha(x,y,s)\|_{L^2((0,\infty),\frac{dt}{t})}+ \|\nabla_yt^m\partial_t^m Q_t^\alpha(x,y,s)\|_{L^2((0,\infty),\frac{dt}{t})}\\\le C  q_{{-}}(x,y,s)^{-n-1/2-\sum_{i=1}^n\alpha_i}.
   	\end{align*}
   \end{lemma}

\begin{proof}
	Let $\tau>0$. We recall that 
	 $$
	Q_t^\alpha(x,y,s)=\frac{t}{2\sqrt{\pi}}\int_0^\infty \frac{e^{-\frac{t^2}{4u}}}{u^{3/2}}(1-e^{-u})^{-n-\sum_{i=1}^n\alpha_i}\exp\left(-\frac{q_{{-}}(e^{-u/2}x,y,s)}{1-e^{-u}}\right)du, 
	$$
	with $x,y\in (0,\infty)^n$, $t\in (0,\infty)$ and $s\in (-1,1)^n.$

	Since $q_{{-}}(e^{-u/2}x,y,s)\ge q_{{-}}(x,y,s)-C(1-e^{-u/2}),$ $(x,y,s)\in N_\tau$ and $u\in (0,\infty)$ \red{(see \cite[(2.6)]{Sa2}), by using \eqref{g'}} we get that
	\begin{align*}
	 	\|&t^m\partial_t^m Q_t^\alpha(x,y,s)\|_{L^2((0,\infty),\frac{dt}{t})}\\&\le C\int_0^\infty \left( \int_0^\infty\frac{e^{-\frac{t^2}{8u}}t^{2m-1}}{u^{{m}}}{dt}\right)^{1/2}\left|\partial_u\left[(1-e^{-u})^{-n-\sum_{i=1}^n\alpha_i}\exp\left(-\frac{q_{{-}}(e^{-u/2}x,y,s)}{1-e^{-u}}\right)\right]\right|du\\
	 	&\le C \int_0^\infty\left|\partial_u\left[(1-e^{-u})^{-n-\sum_{i=1}^n\alpha_i}\exp\left(-\frac{q_{{-}}(e^{-u/2}x,y,s)}{1-e^{-u}}\right)\right]\right|du\\
	 	&\le C\sup_{u\in (0,\infty)} (1-e^{-u})^{-n-\sum_{i=1}^n\alpha_i}\exp\left(-\frac{q_{{-}}(x,y,s)}{1-e^{-u}}\right)\\&\le\frac{ C}{q_-(x,y,s)^{n+\sum_{i=1}^n\alpha_i}}, \:\:(x,y,s)\in N_\tau.
	\end{align*}

In a similar way we can prove the inequalities involving $\nabla_{x,y}$.
\end{proof}

\begin{lemma}\label{Claim2.3}
The operator $\red{	\tilde{\mathbb{H}}}^\alpha$, defined  by
	$$
\red{	\tilde{\mathbb{H}}}^\alpha (f)(x,t)=\int_{(0,\infty)^n}\red{	\tilde{\mathbb{H}}}_t^\alpha(x,y)f(y)dm_\alpha(y), \:\: x\in (0,\infty)^n \text{ and } t>0,
	$$
	where
	$$
	\red{	\tilde{\mathbb{H}}}_t^\alpha(x,y)=\int_{(-1,1)^n}t^m\partial_t^m Q_t^\alpha(x,y,s)\varphi(x,y,s)\Pi_\alpha(s) ds,\quad x,y\in(0,\infty)^n, \: t>0,
	$$
	is bounded from 
	$L^q((0,\infty)^n,\nu_\alpha)$ into $L^q_{L^2((0,\infty),\frac{dt}{t})}((0,\infty)^n,\nu_\alpha)$, $1<q<\infty,$ (respectively from $L^1((0,\infty)^n,\nu_\alpha)$ into $L^{1,\infty}_{L^2((0,\infty),\frac{dt}{t})}((0,\infty)^n,\nu_\alpha)$) if, and only if, $\red{	\tilde{\mathbb{H}}}^\alpha$ is bounded from 
	$L^q((0,\infty)^n,m_\alpha)$ into $L^q_{L^2((0,\infty),\frac{dt}{t})}((0,\infty)^n,m_\alpha)$, $1<q<\infty,$ (respectively from $L^1((0,\infty)^n,m_\alpha)$ into $L^{1,\infty}_{L^2((0,\infty),\frac{dt}{t})}((0,\infty)^n,m_\alpha)$).
\end{lemma}

   \begin{proof}
   	We can proceed as in the proof of Lemma \ref{Claim3} in the subsection \ref{subs2.1}.
   	\end{proof}
 
By combining Lemmas \ref{Claim2.1}-\ref{Claim2.3} and applying vector-valued Calder\'on-Zygmund theory for the local operator, the proof of Theorem \ref{teo1.2} is finished.
 \edproof

    \section{Proof of Theorem \ref{teo1.3}}\label{sec3}
    
 Let $i=1,\dots,n$. We consider the truncated integral $\mathcal{R}_{i,\alpha;\epsilon}$ defined by
 $$
 \mathcal{R}_{i,\alpha;\epsilon}(f)(x)=\int_{|x-y|>\epsilon} \mathcal{R}_\alpha^{i}(x,y)f(y)d{\nu}_\alpha(y),\quad x\in (0,\infty)^n, \text{    where } \epsilon>0.   
 $$

 We define the local part $\mathcal{R}_{i,\alpha;\epsilon,loc}$ of $\mathcal{R}_{i,\alpha;\epsilon}$ by
   $$
  \mathcal{R}_{i,\alpha;\epsilon,loc}(f)(x)=\int_{|x-y|>\epsilon} \mathcal{R}_{\alpha,loc}^{i}(x,y)f(y)d{\nu}_\alpha(y),\quad x\in (0,\infty)^n, \:\:\epsilon>0,
  $$  
    where 
    $$
    \mathcal{R}_{\alpha,loc}^{i}(x,y)=\frac{1}{\sqrt{\pi}}\int_0^\infty \partial_{x_i}W_{2t,loc}^\alpha (x^2,y^2)e^{-t({\sum_{i=1}^n\alpha_i+n})}\frac{dt}{\sqrt{t}}, \quad x,y\in (0,\infty)^n, \: x\neq y.
    $$
    The global part $  \mathcal{R}_{i,\alpha;\epsilon,glob}$ is defined by
    $$
      \mathcal{R}_{i,\alpha;\epsilon,glob}(f)=  \mathcal{R}_{i,\alpha;\epsilon}(f)-  \mathcal{R}_{i,\alpha;\epsilon,loc}(f), \quad \epsilon>0.
    $$
    We have that 
    $$
    \V_\rho(\{  \mathcal{R}_{i,\alpha;\epsilon}\}_{\epsilon>0})(f)\le     \V_\rho(\{  \mathcal{R}_{i,\alpha;\epsilon,loc}\}_{\epsilon>0})(f)+    \V_\rho(\{  \mathcal{R}_{i,\alpha;\epsilon,glob}\}_{\epsilon>0})(f).
    $$
    \subsection{$L^p$-boundedness properties of $\V_\rho(\{  \mathcal{R}_{i,\alpha;\epsilon,glob}\}_{\epsilon>0})$}
    
    \textcolor{white}{v}\vspace{0.2cm}

  Let $0<\epsilon_k<\epsilon_{k-1}<\dots<\epsilon_1$, with $k\in\N$. We get
  \begin{align*}
  \big(\sum_{j=1}^{k-1}|\mathcal{R}_{i,\alpha;\epsilon_j,glob}(f)(x)&-\mathcal{R}_{i,\alpha;\epsilon_{j+1},glob}(f)(x)|^\rho\big)^{1/\rho}\\
  &= \left(\sum_{j=1}^{k-1}\left|\int_{\epsilon_{j+1}<|x-y|<\epsilon_{{j}}}\mathcal{R}^i_{\alpha,glob}(x,y)f(y)d\nu_\alpha(y)\right|^\rho\right)^{1/\rho}\\
  &\le \int_{\epsilon_{k}<|x-y|<\epsilon_1}|\mathcal{R}^i_{\alpha,glob}(x,y)||f(y)|d\nu_\alpha(y),\quad x\in (0,\infty)^n.
  \end{align*}
    Here, $\mathcal{R}^i_{\alpha,glob}(x,y)=\mathcal{R}^i_{\alpha}(x,y)-\mathcal{R}^i_{\alpha,loc}(x,y)$, $x,y\in (0,\infty)^n$.

Then,
    $$
     \V_\rho(\{  \mathcal{R}_{i,\alpha;\epsilon,glob}\}_{\epsilon>0})(f)(x)\le \int_{{(0,\infty)^n}}|\mathcal{R}^i_{\alpha,glob}(x,y)||f(y)|d\nu_\alpha(y),\quad x\in (0,\infty)^n.
    $$
    According to the proof of \cite[Proposition 3.1]{Sa1}, $ \V_\rho(\{  \mathcal{R}_{i,\alpha;\epsilon,glob}\}_{\epsilon>0})$ is bounded from $L^1((0,\infty)^n,\nu_\alpha)$ into $L^{1,\infty}((0,\infty)^n,\nu_\alpha)$ and by \cite[Remark 3.2]{Sa1} (see also \cite[Proposition 3]{Sa2}),  $ \V_\rho(\{  \mathcal{R}_{i,\alpha;\epsilon,glob}\}_{\epsilon>0})$ is bounded from $L^2((0,\infty)^n,\nu_\alpha)$ into itself. By interpolation we deduce that  $ \V_\rho(\{  \mathcal{R}_{i,\alpha;\epsilon,glob}\}_{\epsilon>0})$ is bounded from $L^q((0,\infty)^n,\nu_\alpha)$ into itself, for every $1<q<2$.
    
    Actually, by defining 
    \begin{align*}
    K_\alpha(x,y,s)=\left\{ \begin{array}{cc}
  1, & \sum_{i=1}^n s_ix_iy_i\le 0, \\
    \left(\frac{q_+(x,y,s)}{q_-(x ,y,s)}\right)^{(n+\sum_{i=1}^n\alpha_i)/2}e^{\frac{|y|^2+|x|^2}{2}-\frac{\sqrt{q_+(x,y,s)q_-(x,y,s)}}{2}},  & \sum_{i=1}^n s_ix_iy_i>  0,
    \end{array}
    \right.\\
    \end{align*}
    and using \cite[page 402]{Sa1}, we have that
    $$
    \red{|}\mathcal{R}_{\alpha,glob}^i(x,y)\red{|}\le C\int_{(-1,1)^n}K_\alpha(x,y,s)\chi_{N_1^c}(x,y,s)\Pi_\alpha(s)ds, \:\;x,y\in (0,\infty)^n. 
    $$
    The operator $\mathbb{K}_\alpha$, defined by
     $$
  \mathbb{K}_\alpha(f)(x)=\int_{(0,\infty)^n}f(y)\int_{(-1,1)^n}K_\alpha(x,y,s)\chi_{N_1^c}(x,y,s)\Pi_\alpha(s)dsd\nu_\alpha(y), \:\;x\in (0,\infty)^n,
    $$
    is selfadjoint in $L^2((0,\infty)^n,\nu_\alpha)$ and it is bounded from $L^2((0,\infty)^n,\nu_\alpha)$ into itself and from   $L^1((0,\infty)^n,\nu_\alpha)$ into  $L^{1,\infty}((0,\infty)^n,\nu_\alpha)$. Then, by using interpolation and duality we conclude that  $\mathbb{K}_\alpha$ is bounded from $L^q((0,\infty)^n,\nu_\alpha)$ into itself, for every $1<q<\infty$, and from $L^1((0,\infty)^n,\nu_\alpha)$ into  $L^{1,\infty}((0,\infty)^n,\nu_\alpha)$.

    \subsection{$L^p$-boundedness properties of $ \V_\rho(\{  \mathcal{R}_{i,\alpha;\epsilon,loc}\}_{\epsilon>0})$}
    
    \textcolor{white}{v}\vspace{0.2cm}

   \red{Firstly, we shall prove  that the variation operator $ \V_\rho(\{  \mathcal{R}_{i,\alpha;\epsilon}\}_{\epsilon>0})$ is bounded from 
   	$L^2((0,\infty)^n,\nu_\alpha)$ into itself. Once this property is proved,} since  $ \V_\rho(\{  \mathcal{R}_{i,\alpha;\epsilon,glob}\}_{\epsilon>0})$ is bounded from $L^2((0,\infty)^n,\nu_\alpha)$ into itself, we will deduce that  $ \V_\rho(\{  \mathcal{R}_{i,\alpha;\epsilon,loc}\}_{\epsilon>0})$ is also bounded from $L^2((0,\infty)^n,\nu_\alpha)$ into itself. 
    	Thus, by proceeding as in the proof of Lemma \ref{Claim3}, we get that  $ \V_\rho(\{  \mathcal{R}_{i,\alpha;\epsilon,loc}\}_{\epsilon>0})$ is bounded from $L^2((0,\infty)^n,m_\alpha)$ into itself (we recall that $((0,\infty)^n, |\cdot|, m_\alpha)$ is a space of homogeneous type).


  \vspace{0.2 cm}
    In \cite{BCT}, the authors studied the variation and oscillation operators for the Riesz transfom $\mathbb{R}_\lambda$ associated with the Laguerre type operator $\mathbb{L}_\lambda$ given by
    $$
    \mathbb{L}_\lambda =\frac{1}{2}{D_\lambda}^*D_\lambda+\lambda +1, \:\:\lambda>-1,
    $$
    in $(0,\infty)$. Here, $D_\lambda f(x)=x^{\lambda+1/2}\frac{d}{dx}(x^{-\lambda-1/2}f(x))+xf(x)$ and $D_\lambda^*$ represents the formal adjoint of $D_\lambda$ in $L^2((0,\infty),dx)$.

    We now recall the definition of  $\mathbb{R}_\lambda$. For every $k\in\N$, we define
    $$
    \varphi_k^\lambda(x)=\left(\frac{\Gamma(k+1)}{\Gamma(k+1+\lambda)} \right)^{1/2}e^{-x^2/2}x^\lambda L_k^\lambda(x^2)(2x)^{1/2}, \:\:x\in (0,\infty).
    $$
    $\{\varphi_k^\lambda\}_{k=0}^\infty$ is an orthonormal basis in $L^2((0,\infty),dx)$. We have that
    $$
    \mathbb{L}_\lambda\varphi_k^\lambda=(2k+\lambda+1)\varphi_k^\lambda, \quad k\in\N.
    $$
    We define $\mathfrak{{L}}_\lambda$ as follows
      $$
    \mathfrak{L}_\lambda f=\sum_{k=0}^\infty(2k+\lambda+1)e_k^\lambda(f)\varphi_k^\lambda, \quad f\in \mathcal{D}(    \mathfrak{L}_\lambda),
    $$
    where the domain of $\mathfrak{L}_\lambda$, $\mathcal{D}(\mathfrak{L}_\lambda)$,     is defined by
    $$
    \mathcal{D}(\mathfrak{L}_\lambda)=\left\{f\in L^2(0,\infty): \:\sum_{k=0}^\infty|(2k+\lambda+1)e_k^\lambda(f)|^2<\infty \right\}.
    $$
    Here, for every $f\in L^2((0,\infty), dx)$,
    $$
    e_k^\lambda(f)=\int_0^\infty \varphi_k^\lambda (x)f(x)dx, \quad k\in\N.
    $$
  For every $t>0$, we consider
  $$
  \mathbb{W}_t^\lambda(f)=\sum_{k=0}^\infty e^{-(2k+\lambda+1)t}e_k^\lambda(f)\varphi_k^\lambda,\quad f\in L^2((0,\infty),dx).
  $$  
    
  $\{  \mathbb{W}_t^\lambda\}_{t>0}$ is the semigroup of operators generated by $-\mathfrak{L}_\lambda$ in $L^2((0,\infty),dx)$. We can write, for every $t>0$,
  \begin{equation}\label{eq3.1}
    \mathbb{W}_t^\lambda(f)(x)=\int_0^\infty  \mathbb{W}_t^\lambda(x,y)f(y)dy, \quad f\in L^2((0,\infty),dx),
  \end{equation}
    where 
    $$
    \mathbb{W}_t^\lambda(x,y)=2(xy)^{1/2}\frac{e^{-t}}{1-e^{-2t}}I_{\alpha}\left(\frac{2x ye^{-t}}{1-e^{-2t}}\right)\exp\left(-\frac{1}{2}(x^2+y^2)\frac{1+e^{-2t}}{1-e^{-2t}} \right),\quad x,y\in (0,\infty).
    $$

    The integral in \eqref{eq3.1} also defines a bounded operator in $L^p((0,\infty),dx)$, for every $1\le p\le \infty$. Furthermore, $\{ \mathbb{W}_t^\lambda\}_{t>0}$ can be seen as a semigroup of operators in $L^p((0,\infty),dx)$, for every $1\le p\le \infty$.
    
    In the $\mathfrak{L}_\lambda$-setting, the Riesz transform $\mathbb{R}_\lambda$ is defined by
    $$
    \mathbb{R}_\lambda(f)(x)=\lim_{\epsilon\to 0}\int_{|x-y|>\epsilon} \mathbb{R}_\lambda(x,y)f(y)dy, \quad \text{ for almost all } x\in(0,\infty)
    $$
    and every $f\in L^p((0,\infty),dx)$, $1\le p<\infty$, being
    $$
    \mathbb{R}_\lambda(x,y)=\frac{1}{\sqrt{\pi}}\int_0^\infty D_\lambda \mathbb{W}_t^\lambda(x,y)\frac{dt}{\sqrt{t}}, \quad  x,y\in (0,\infty), \:\:x\neq y.
    $$

    $L^p$-boundedness properties of $\mathbb{R}_\lambda$  were studied in \cite{AMST} and  \cite{NS2}.
    
    For every $\epsilon>0$, we denote
    $$
     \mathbb{R}_{\lambda,\epsilon}(f)(x)=\int_{|x-y|>\epsilon} \mathbb{R}_\lambda(x,y)f(y) dy, \quad x\in (0,\infty).
    $$
    In \cite[Theorem 1.4]{BCT}, $L^p$-boundedness properties of the variation operator $\V_\rho(\{\mathbb{R}_{\lambda,\epsilon}\}_{\epsilon>0})$ were established. In particular, it was proved that $\V_\rho(\{\mathbb{R}_{\lambda,\epsilon}\}_{\epsilon>0})$ is bounded from $L^2((0,\infty),dx)$ into itself. By using the procedure developed in \cite{BCC}, this property can be extended to higher dimensions.

    Let $i=1,\dots,n$. By keeping in mind our usual notation and denoting
    $D_{\alpha_i} f(x)=x_i^{\alpha_i+1/2}\frac{d}{dx_i}(x_i^{-\alpha_i-1/2}f(x))+x_if(x)$, $x=(x_1,\dots,x_n)\in (0,\infty)^n$, 
    we define, for every $\epsilon>0,$
    $$
      \mathbb{R}_{i,\alpha;\epsilon}(f)(x)=\int_{|x-y|>\epsilon} \mathbb{R}_\alpha^i(x,y)f(y) dy, \quad x\in (0,\infty)^n.
    $$
    where
    $$
    \mathbb{R}_\alpha^i(x,y)=\frac{1}{\sqrt{\pi}}\int_0^\infty D_{\alpha_i} \mathbb{W}_t^\alpha(x,y)\frac{dt}{\sqrt{t}},\quad x,y\in (0,\infty)^n, \:\: x\neq y,
    $$
    and
    $$
    \mathbb{W}_t^\alpha(x,y)=\prod_{i=1}^n\mathbb{W}_t^{\alpha_i}(x_i,y_i), \quad x=(x_1,\dots,x_n), \:\: y=(y_1,\dots,y_n)\in (0,\infty)^n \text{ and }t>0.
        $$
    We have that the variation operator $\V_\rho(\{\mathbb{R}_{i,\alpha;\epsilon}\}_{\epsilon>0})$ is bounded from $L^2((0,\infty)^n,dx)$ into itself.
    
    \red{Since
    \begin{align*}
    \mathbb{W}_t^\alpha(x,y)={2^n}e^{-\frac{|x|^2+|y|^2}{2}}{\prod_{j=1}^n\frac{(x_jy_j)^{\alpha_j+1/2}}{\Gamma(\alpha_j+1)}}\blue{\mathcal{W}_t^{\alpha}(x,y)}, \quad x,y\in (0,\infty)^n, \text{ and } t>0,
    \end{align*}
    \blue{and  $$\mathcal{W}_t^{\alpha}(x,y)=e^{-t\left(\sum_{i=1}^n \alpha_i+n\right)}\prod_{i=1}^nW_{2t}^{\alpha_i}(x_i^2,y_i^2)=:\prod_{i=1}^n \mathcal{W}_t^{\alpha_i}(x_i,y_i),$$}
we get that
\begin{align*}
x_i^{\alpha_i+1/2}\partial_{x_i}\left(x_i^{-\alpha_i-1/2}\mathbb{W}_t^\alpha(x,y)\right)&=2^n{\prod_{\substack{j=1\\j\neq i}}^n\frac{(x_jy_j)^{\alpha_j+1/2}}{\Gamma(\alpha_j+1)}}\mathcal{W}_t^{\alpha_j}(x_j,y_j)e^{-\frac{|x|^2+|y|^2}{2}}\\
\cdot&\blue{\frac{(x_iy_i)^{\alpha_i+1/2}}{\Gamma(\alpha_i+1)}}\left(-x_i\mathcal{W}_t^{\alpha_i}(x_i,y_i)+\partial_{x_i}\mathcal{W}_t^{\alpha_i}(x_i,y_i)\right), \\&\qquad \qquad \qquad\qquad \qquad\qquad\:\, x,y\in (0,\infty)^n, \; t>0.
\end{align*}
It follows that}    
    $$
    D_{\alpha_i}\mathbb{W}_t^\alpha(x,y)={2^n}e^{-\frac{|x|^2+|y|^2}{2}}{\prod_{j=1}^n\frac{(x_jy_j)^{\alpha_j+1/2}}{\Gamma(\alpha_j+1)}}\partial_{x_i}\mathcal{W}_t^\alpha(x,y), \quad x,y\in (0,\infty)^n, \text{ and } t>0.
    $$
    Then,
    $$
    \mathbb{R}_\alpha^i(x,y)=2^ne^{-\frac{|x|^2+|y|^2}{2}}{\prod_{j=1}^n\frac{(x_jy_j)^{\alpha_j+1/2}}{\Gamma(\alpha_j+1)}}\mathcal{R}_\alpha^i(x,y),\quad x,y\in (0,\infty)^n, \:\: x\neq y,
    $$
    and, for every $\epsilon>0$, 
    \begin{align*}
    \mathcal{R}_{i,\alpha;\epsilon}(f)(x)&=\int_{|x-y|>\epsilon} e^{\frac{|x|^2+|y|^2}{2}}{\prod_{j=1}^n(x_jy_j)^{-(\alpha_j+1/2)}}\mathbb{R}_\alpha^i(x,y)f(y) e^{-|y|^2}{\prod_{j=1}^n y_j^{2\alpha_j+1}}dy\\
    &=e^{|x|^2/2}{\prod_{j=1}^n x_j^{-(\alpha_j+1/2)}}\mathbb{R}_{i,\alpha;\epsilon}(e^{-|y|^2/2}{\prod_{j=1}^n y_j^{\alpha_j+1/2}}f)(x), \quad x\in (0,\infty)^n.
    \end{align*}
    
    Thus, for every $x\in (0,\infty)^n$, it follows that
    \begin{align*}
    \V_\rho(\{\mathcal{R}_{i,\alpha;\epsilon}\}_{\epsilon>0})(f)(x)=e^{|x|^2/2}{\prod_{j=1}^nx_j^{-(\alpha_j+1/2)}} \V_\rho(\{\mathbb{R}_{i,\alpha;\epsilon}\}_{\epsilon>0})(e^{-|y|^2/2}{\prod_{j=1}^n y_j^{\alpha_j+1/2}}f)(x),
    \end{align*}
    and if $f\in {L^2((0,\infty)^n, \nu_\alpha)}$ we get that 
    \begin{align*}
    \|  \V_\rho(\{\mathcal{R}_{i,\alpha;\epsilon}\}_{\epsilon>0})(f)&\|_{L^2((0,\infty)^n, \nu_\alpha)}=\left\| \V_\rho(\{\mathbb{R}_{i,\alpha;\epsilon}\}_{\epsilon>0})\left(e^{-\frac{|y|^2}{2}}{2^{\red{n/2}}}\prod_{j=1}^n \frac{y_j^{\alpha_j+1/2}}{\red{\sqrt{\Gamma(\alpha_j+1)}}}f\right)\right\|_{L^2((0,\infty)^n, dx)}\\
    &\le C \left\|e^{-|y|^2{/2}}{2^{n/2}\prod_{j=1}^n \frac{y_j^{\alpha_j+1/2}}{\sqrt{\Gamma(\alpha_j+1)}}}f \right\|_{L^2((0,\infty)^n, dx)}=C \|f\|_{L^2((0,\infty)^n, \nu_\alpha)}.
    \end{align*}
    Therefore, we have just proved that $ \V_\rho(\{\mathcal{R}_{i,\alpha;\epsilon}\}_{\epsilon>0})$ is bounded from $L^2((0,\infty)^n, \nu_\alpha)$ into itself.

     Observe that
    
    $$
    \mathcal{R}^i_{\alpha,loc}(x,y)=\int_{(-1,1)^n} \mathcal{R}^i_{\alpha,loc}(x,y,s)\Pi_\alpha(s)ds, \quad x,y\in (0,\infty)^n, \:\: x\neq y,
    $$
    {where}
    \begin{align*}
    \mathcal{R}^i_{\alpha,loc}(x,y&,s)=\frac{{-2}}{\sqrt{\pi}}\int_0^\infty (1-e^{-2t})^{-n-\sum_{i=1}^n\alpha_i{-1}}(e^{-2t}x_i-e^{-t}y_i{s_i})e^{-t({n+\sum_{i=1}^n\alpha_i})}\\
    &\:\:\cdot\exp\left(-\frac{q_{{-}}(e^{-t}x,y,s)}{1-e^{-2t}}\right)e^{|y|^2}\varphi(x,y,s)\frac{dt}{\sqrt{t}},\:\: x,y\in (0,\infty)^n, \: x\neq y, \:\:s\in (-1,1)^n.
    \end{align*}
    According to \cite[Lemma 3.3]{Sa1}, we have that
    \begin{equation}\label{eq3.2}
    |\mathcal{R}_{\alpha,loc}^i(x,y,s)|\le C q_-(x,y,s)^{-n-\sum_{i=1}^n\alpha_i}
    \end{equation}
    and
    \begin{equation}\label{eq3.3}
    |\nabla_x\mathcal{R}_{\alpha,loc}^i(x,y,s)|+ |\nabla_y\mathcal{R}_{\alpha,loc}^i(x,y,s)|\le C q_-(x,y,s)^{-n-1/2-\sum_{i=1}^n\alpha_i},
    \end{equation}
    for every $(x,y,s)\in N_\tau$, with $\tau>0$.
    
    Then, by using \cite[Lemma 3.1]{BCN} we deduce that
    \begin{equation}\label{eq3.4}
    |\mathcal{R}_{\alpha,loc}^i(x,y)|\le \frac{C}{m_\alpha(B(x,|x-y|))},\quad x,y\in (0,\infty)^n, \: x\neq y,
    \end{equation}
    and 
    \begin{equation}\label{eq3.5}
    |\nabla_x\mathcal{R}_{\alpha,loc}^i(x,y)|+ |\nabla_y\mathcal{R}_{\alpha,loc}^i(x,y)|\le  \frac{C}{|x-y|m_\alpha(B(x,|x-y|))},\quad x,y\in (0,\infty)^n, \: x\neq y.
    \end{equation}
    	
    
    
   \red{Inequalities \eqref{eq3.4} and \eqref{eq3.5} will be very  useful in the sequel. }
        
     
      We are going to see that $ \V_\rho(\{\mathcal{R}_{i,\alpha;\epsilon,loc}\}_{\epsilon>0})$ is bounded \red{from $L^1((0,\infty)^n, m_\alpha)$ into $L^{1,\infty}((0,\infty)^n, m_\alpha)$. Once this property is established, by using interpolation we conclude that $\V_\rho(\{\mathcal{R}_{i,\alpha;\epsilon,loc}\}_{\epsilon>0})$ is bounded from $L^q((0,\infty)^n, m_\alpha)$ into itself, for every $1<q\le 2$.}
    
    
    Suppose that $f\red{\in L^2((0,\infty)^n,m_\alpha)}$ is compactly supported. \red{Then, $f\in L^1((0,\infty)^n,m_\alpha)$}. By {rescaling}, in order to prove our result it is sufficient to see that
    $$
    m_\alpha(\{x\in (0,\infty)^n:\: \V_\rho(\{\mathcal{R}_{i,\alpha;\epsilon,loc}\}_{\epsilon>0}){(f)}(x)>2\})\le C \|f\|_{L^1((0,\infty)^n, m_\alpha),}
    $$
        where $C>0$ does not depend on $f$.
   According to Calder\'on-Zygmund decomposition of $f$ with {height} $2$ (see  \cite[Th\'eor\`eme 2.2, p.73-74]{CoW}), there exists a family of cubes $\{I_j\}_{j\in\N}$ and $C,M>0$ such that

   \begin{itemize}
 	\item[(i)]  	$\displaystyle|f(x)|\le C,$ $x\in \Omega^c$, where $\displaystyle\Omega=\cup_{j=1}^\infty I_j;$
     	\item[(ii)]  $\displaystyle2<\frac{1}{m_\alpha(I_j)}\int_{I_j}|f(x)|dx\le C, \quad j\in\N;$
     	 	\item[(ii)]  $\displaystyle\sum_{j=1}^\infty m_\alpha (I_j)\le C\int_{(0,\infty)^n}|f(x)|dm_\alpha(x);$
     	 	 	\item[(iv)]  $\displaystyle\sum_{j=1}^\infty \chi_{I_j}(x)\le M,$ $x\in (0,\infty)^n$.
    \end{itemize}
    There exist also integrable functions $g$ and $b_j$, $j\in\N$, such that
       \begin{itemize}
    	\item[(v)]  	$f=g+b$, where $b=\sum_{j=1}^\infty b_j$;
    	\item[(vi)]  $|g(x)|\le C$, for almost all $x\in (0,\infty)^n$;
    	\item[(vii)]  $\|g\|_{L^1((0,\infty)^n,m_\alpha)}\le C \|f\|_{L^1((0,\infty)^{\red{n}},m_\alpha)};$
    	\item[(viii)]$\supp b_j\subseteq I_j$, $\displaystyle\int_{(0,\infty)^n}b_j(x)dm_\alpha(x)=0$ and $\displaystyle\frac{1}{m_\alpha(I_j)}\int_{I_j}|b_j(x)|dm_\alpha(x)\le C$, $j\in\N$.
    	\item[(ix)]  $\displaystyle\int_{I_i}|b_i(x)|dm_\alpha(x)\le C \int_{I_i}|f(x)|dm_\alpha(x)$, $i\in\N$.
    \end{itemize}
\red{Note that, 
	since  $f\in L^2((0,\infty)^n,m_\alpha)$ and
	 $g\in L^1((0,\infty)^n,m_\alpha)\cap L^\infty((0,\infty)^n,m_\alpha),$ $g,b\in L^2((0,\infty)^n,m_\alpha)$.}
    We have that
    \begin{align*}
    m_\alpha(\{x\in (0,\infty)^n:\: \V_\rho&(\{\mathcal{R}_{i,\alpha;\epsilon,loc}\}_{\epsilon>0}){(f)}({x})>2\})\\
    &\le   m_\alpha(\{x\in (0,\infty)^n:\: \V_\rho(\{\mathcal{R}_{i,\alpha;\epsilon,loc}\}_{\epsilon>0}){(g)}({x})>1\}) \\&+  m_\alpha(\{x\in (0,\infty)^n:\: \V_\rho(\{\mathcal{R}_{i,\alpha;\epsilon,loc}\}_{\epsilon>0}){(b)}({x})>1\}).
    \end{align*}
    Since  $ \V_\rho(\{\mathcal{R}_{i,\alpha;\epsilon,loc}\}_{\epsilon>0})$ is bounded from $L^2((0,\infty)^n, m_\alpha)$ into itself, we get that
    \begin{align*}
     m_\alpha(\{x&\in (0,\infty)^n:\: \V_\rho(\{\mathcal{R}_{i,\alpha;\epsilon,loc}\}_{\epsilon>0}){(g)}({x})>1\})\\&\le \int_{(0,\infty)^n} (\V_\rho(\{\mathcal{R}_{i,\alpha;\epsilon,loc}\}_{\epsilon>0}){(g)}({x}))^2dm_\alpha(x)\le C \int_{(0,\infty)^n}|g(x)|^2dm_\alpha(x)
     \\
     &\le  C \int_{(0,\infty)^n}|g(x)| dm_\alpha(x)\le  C \int_{(0,\infty)^n}|f(x)| dm_\alpha(x).
    \end{align*}
    On the other hand, we have that
       \begin{align*}
    m_\alpha(\{x\in (0,\infty)^n:\: \V_\rho(&\{\mathcal{R}_{i,\alpha;\epsilon,loc}\}_{\epsilon>0}){(b)}({x})>1\})\\
     &\le m_\alpha(\{x\in \tilde{\Omega}^c:\: \V_\rho(\{\mathcal{R}_{i,\alpha;\epsilon,loc}\}_{\epsilon>0}){(b)}({x})>1\})\\&\quad+ m_\alpha(\{x\in \tilde{\Omega}:\: \V_\rho(\{\mathcal{R}_{i,\alpha;\epsilon,loc}\}_{\epsilon>0}){(b)}({x})>1\}).
    \end{align*}
    Here, $\tilde{\Omega}=\displaystyle\bigcup_{j=1}^\infty\tilde{I}_j.$ If $I=\prod_{i=1}^n (a_i,b_i)\subset (0,\infty)^n$ and ${d}_i=b_i-a_i$, $i=1,\dots,n,$ we define $\tilde{I}=  \prod_{i=1}^n\left(\left(\blue{\frac{a_i+b_i}{2}-\left(\frac{3}{2}\sqrt{n}+1\right){d}_i,\frac{a_i+b_i}{2}+\left(\frac{3}{2}\sqrt{n}+1\right){d}_i}\right)\cap (0,\infty)\right)$.
    
    Since $m_\alpha$ is doubling, we have that
    \begin{align*}
    m_\alpha(\{x\in \tilde{\Omega}:\: \V_\rho(\{\mathcal{R}_{i,\alpha;\epsilon,loc}\}_{\epsilon>0}){(b)}({x})>1\})&\le m_\alpha (\tilde{\Omega})\le C\sum_{j=1}^n m_\alpha(I_j)\\&\le C \int_{(0,\infty)^n}|f(x)|dm_\alpha(x).
    \end{align*}   
    In order to prove that 
    $$
    m_\alpha(\{x\in \tilde{\Omega}^c:\: \V_\rho(\{\mathcal{R}_{i,\alpha;\epsilon,loc}\}_{\epsilon>0}){(b)}({x})>1\})\le C \int_{(0,\infty)^n}|f(x)|dm_\alpha(x),
    $$
    we shall follow the five steps in \cite[p. {1425--1430}]{MTX2}. We now prove the properties that we will need in our setting  {so that the procedure runs.}
    
    \red{	Since  $b\in L^2((0,\infty)^n,m_\alpha)$, $\V_\rho(\{\mathcal{R}_{i,\alpha;\epsilon,loc}\}_{\epsilon>0}){(b)}\in L^2((0,\infty)^n,m_\alpha)$ and hence \newline $\V_\rho(\{\mathcal{R}_{i,\alpha;\epsilon,loc}\}_{\epsilon>0}){(b)}(x)<\infty$, $x\in\blue{(0,\infty)^{n}}\setminus \mathbb{V}$, where $m_\alpha (\mathbb{V})=0.$  A continuity argument allows us to see that
      $$
    \V_\rho(\{\mathcal{R}_{i,\alpha;\epsilon,loc}\}_{\epsilon>0}){(b)}({x})= \sup_{\substack{0<t_k\dots<t_1\\t_j\in\mathbb{Q}_+,\; j=1,\dots,k,\; k\in\N}}\left(\sum_{j=1}^{k-1}|\mathcal{R}_{i,\alpha;t_j,loc}{(b)}({x})-\mathcal{R}_{i,\alpha;t_{j+1},loc}{(b)}({x})|^\rho \right)^{1/\rho},
    $$  
where $\mathbb{Q}_+=\mathbb{Q}\cap (0,\infty)$.}
    
    \red{Let $x\in \tilde{\Omega}^c\setminus \mathbb{V}$. We can find $0<t_k<t_{k-1}<\dots<t_1$, for certain $k\in\N$, being $t_j\in\mathbb{Q}_+,\; j=1,\dots,k,$ and that     
    	\begin{equation}\label{F1}
\V_\rho(\{\mathcal{R}_{i,\alpha;\epsilon,loc}\}_{\epsilon>0}){(b)}({x})\le 2\left(\sum_{j=1}^{k-1}|\mathcal{R}_{i,\alpha;t_j,loc}{(b)}({x})-\mathcal{R}_{i,\alpha;t_{j+1},loc}{(b)}({x})|^\rho \right)^{1/\rho}.
\end{equation}
It is clear that the family $\{t_j\}_{j=1}^k$ is depending on $x$. Furthermore, the set $\{t_j\}_{j=1}^{\blue{k}}$ satisfying \eqref{F1} does not need to be unique. It is necessary to choose $\{t_j\}_{j=1}^{\blue{k}}$ in  a precise way in order to get a measurable function in the right hand side of \eqref{F1}.}


\blue{ 
We denote by $P_{finite}(\mathbb{Q}_+)$ the set of all finite subsets of $\mathbb{Q}_+$. 
For every $x\in \tilde{\Omega}^c\setminus \mathbb{V}$, we define the set $\mathcal{S}(x)\subset P_{finite}(\mathbb{Q}_+) $ as follows: the set $A=\{t_1,t_2,\dots, t_k\}\subset \mathbb{Q}_+$, being $t_{j+1}<t_j$, $j=1,\dots,k-1$, is in $\mathcal{S}(x)$ when \eqref{F1} holds. We define $$A(x)=\min\{j\in\N: \, Z_j\in\mathcal{S}(x)\}, \quad x\in \tilde{\Omega}^c\setminus \mathbb{V}.$$  
If $j\in \N$ and $Z_j=\{t_1,\dots,t_k\}$ with $t_{l+1}<t_l$, $l=1,\dots,k-1$, we write
$$
V_j(x)=\left(\sum_{l=1}^{k-1}|\mathcal{R}_{i,\alpha;t_l,loc}{(b)}({x})-\mathcal{R}_{i,\alpha;t_{l+1},loc}{(b)}({x})|^\rho \right)^{1/\rho}, \quad x\in(0,\infty)^{n},
$$
and
\begin{align*}
Y_j=&\Big\{x\in \tilde{\Omega}^c\setminus \mathbb{V}:\,\V_\rho(\{\mathcal{R}_{i,\alpha;\epsilon,loc}\}_{\epsilon>0}){(b)}({x})\le 2 V_j(x) \text{ and }\\
&\qquad\V_\rho(\{\mathcal{R}_{i,\alpha;\epsilon,loc}\}_{\epsilon>0}){(b)}({x})> 2 V_l(x), \:{l<j} \Big\}.
\end{align*} 
The functions $V_j$, $j\in \N$, and $\V_\rho(\{\mathcal{R}_{i,\alpha;\epsilon,loc}\}_{\epsilon>0}){(b)}$, as well as the sets $Y_j$, $j\in \N$, are measurable. Then, the function
$$
\mathbb{F}=\sum_{j=0}^\infty \chi_{Y_j}V_j
$$
is measurable. Note that $\mathbb{F}(x)=V_{A(x)}(x),$ $x\in\tilde{\Omega}^c\setminus \mathbb{V}$.  }

 If $J$ is an interval in $(0,\infty)$ we define
$R_J=\{y\in\R^n: \; |y|\in J\}$ and the following two sets
$$
\mathcal{I}_1(J)\blue{(x)}=\{\ell\in\N: \: I_\ell\subset (x+R_J)\},\blue{\quad x\in (0,\infty)^n,}
$$
and
$$
\mathcal{I}_2(J)\blue{(x)}=\{\ell\in\N: \: I_\ell\cap (x+\partial R_J)\neq \emptyset\},\blue{\quad x\in (0,\infty)^n}.
$$

\blue{In the sequel we will assume that $x\in\tilde{\Omega}^c\setminus \mathbb{V}$. To simplify notations,  we will not write the dependence on $x$ of $\{t_j\}_{j=1}^k$, $\mathcal{I}_1(J)$ and $\mathcal{I}_2(J)$, where $J$ is an interval in $(0,\infty)$.}

\blue{We denote by} $J_j=(t_{j+1},t_j]$, $j=1,\dots,k-1$. \blue{It follows that} $\mathcal{R}_{i,\alpha;t_j,loc}{(b_\ell)}({x})-\mathcal{R}_{i,\alpha;t_{j+1},loc}{(b_\ell)}({x})\neq 0$ only if $\ell\in \mathcal{I}_1(J_j)\cup \mathcal{I}_2(J_j)$.

{
  We have that
\begin{align*}
&\left(\sum_{j=1}^{k-1}\Big|\mathcal{R}_{i,\alpha;t_j,loc}{(b)}({x})-\mathcal{R}_{i,\alpha;t_{j+1},loc}{(b)}({x})\Big|^\rho \right)^{1/\rho}\\
&\le \left(\sum_{j=1}^{k-1}\Big|\sum_{\ell\in \mathcal{I}_1(J_j)}(\mathcal{R}_{i,\alpha;t_j,loc}{(b_\ell)}({x})-\mathcal{R}_{i,\alpha;t_{j+1},loc}{(b_\ell)}({x}))\Big|^\rho \right)^{1/\rho}\\
 &+\left(\sum_{j=1}^{k-1}\Big|\sum_{\ell\in \mathcal{I}_2(J_j)}(\mathcal{R}_{i,\alpha;t_j,loc}{(b_\ell)}({x})-\mathcal{R}_{i,\alpha;t_{j+1},loc}{(b_\ell)}({x}))\Big|^\rho \right)^{1/\rho}\\
 &\le \sum_{j=1}^{k-1}\Big|\sum_{\ell\in \mathcal{I}_1(J_j)}(\mathcal{R}_{i,\alpha;t_j,loc}{(b_\ell)}({x})-\mathcal{R}_{i,\alpha;t_{j+1},loc}{(b_\ell)}({x}))\Big|\\
 &+\left(\sum_{j=1}^{k-1}\Big|\sum_{\ell\in \mathcal{I}_2(J_j)}(\mathcal{R}_{i,\alpha;t_j,loc}{(b_\ell)}({x})-\mathcal{R}_{i,\alpha;t_{j+1},loc}{(b_\ell)}({x}))\Big|^2\right)^{1/2}.
\end{align*}
Our objective is to see that
\begin{align}\label{B1}
 m_\alpha&\Bigg(\Bigg\{x\in \tilde{\Omega}^c\blue{\setminus \mathbb{V}}:\: \sum_{j=1}^{k-1}\Big|\sum_{\ell\in \mathcal{I}_1(J_j)}(\mathcal{R}_{i,\alpha;t_j,loc}{(b_\ell)}({x})-\mathcal{R}_{i,\alpha;t_{j+1},loc}{(b_\ell)}({x}))\Big|>1/2\Bigg\}\Bigg)\nonumber\\&\le C \int_{(0,\infty)^n}|f(x)|dm_\alpha(x)
\end{align}
and 
\begin{align}\label{B2}
m_\alpha&\Bigg(\Bigg\{x\in \tilde{\Omega}^c\blue{\setminus \mathbb{V}}:\: \sum_{j=1}^{k-1}\Big|\sum_{\ell\in \mathcal{I}_2(J_j)}(\mathcal{R}_{i,\alpha;t_j,loc}{(b_\ell)}({x})-\mathcal{R}_{i,\alpha;t_{j+1},loc}{(b_\ell)}({x}))\Big|^2>1/4\Bigg\}\Bigg)\nonumber\\&\le C \int_{(0,\infty)^n}|f(x)|dm_\alpha(x).
\end{align} }
 \blue{	In the sequel we shall denote by  $c_{{\ell}}$ the center of $I_{{\ell}}$  and  by $l_{{\ell}}$ the length of the side of $I_{\ell}$, for every ${{\ell}}\in\N$.  }
    \begin{enumerate}
   
    	\item[(a)] This first step corresponds to step 2 in \cite{MTX2}. Since each $b_{\red{\ell}}$ has zero mean, according to \eqref{eq3.5} we get  
    	\begin{align*}
    { \left|\int_{I_{\red{\ell}}}\mathcal{R}_{{\alpha,loc}}^i(x,y)b_{\red{\ell}}(y)dm_\alpha(y)\right|}&\le	\int_{I_{\red{\ell}}}|\mathcal{R}_{{\alpha,loc}}^i(x,y)-\mathcal{R}_{{\alpha,loc}}^i(x,c_{\red{\ell}})||b_{\red{\ell}}(y)|dm_\alpha(y)\\&\le \frac{C\;\red{l_{\red{\ell}}}\int_{I_{\red{\ell}}}|b_{\red{\ell}}(y)|dm_\alpha(y)}{|x-c_{\red{\ell}}|m_\alpha(B(x,|x-c_{\red{\ell}}|))},\,\, x\not\in \tilde{\Omega} \text{ and } {\red{\ell}}\in\N.
    	\end{align*}
    \end{enumerate}

    	Since $m_\alpha$ is doubling, we obtain that
    	\begin{align*}
   	m_\alpha\Bigg(\Bigg\{&x\in \tilde{\Omega}^c\setminus \mathbb{V}:\: \sum_{j=1}^{k-1}\Big|\sum_{\ell\in \mathcal{I}_1(J_j)}(\mathcal{R}_{i,\alpha;t_j,loc}{(b_\ell)}({x})-\mathcal{R}_{i,\alpha;t_{j+1},loc}{(b_\ell)}({x}))\Big|>1/2\Bigg\}\Bigg)
    \\&\le C 	\int_{{\tilde{\Omega}^c}}\sum_{\red{\ell}=1}^\infty\int_{I_{\red{\ell}}}|\mathcal{R}_{{\alpha,loc}}^i(x,y)-\mathcal{R}_{{\alpha,loc}}^i(x,c_{\red{\ell}})||b_{\red{\ell}}(y)|dm_\alpha(y){dm_\alpha(x)}\\
    	&\le C\sum_{{\red{\ell}}=1}^\infty \red{l_{\red{\ell}}}\int_{I_{\red{\ell}}}|b_{\red{\ell}}(y)|dm_\alpha(y)\int_{|x-c_{\red{\ell}}|>\red{l_{\red{\ell}}}}\frac{dm_\alpha(x)}{|x-c_{\red{\ell}}|m_\alpha(B(x,|x-c_{\red{\ell}}|))}\\
    	&\le C\sum_{\red{\ell}=1}^\infty \red{l_{\red{\ell}}}\int_{I_{\red{\ell}}}|b_{\red{\ell}}(y)|dm_\alpha(y)\sum_{k=0}^\infty\int_{2^k\red{l_{\red{\ell}}}<|x-c_{\red{\ell}}|<2^{k+1}\red{l_{\red{\ell}}}}\frac{dm_\alpha(x)}{|x-c_{\red{\ell}}|m_\alpha(B(x,|x-c_{\red{\ell}}|))}\\
    	&\le C\sum_{{\red{\ell}}=1}^\infty \red{l_{\red{\ell}}}\int_{I_{\red{\ell}}}|b_{\red{\ell}}(y)|dm_\alpha(y)\sum_{k=0}^\infty\frac{1}{2^k\red{l_{\red{\ell}}}}\frac{m_\alpha(B(x,2^{k+1}\red{l_{\red{\ell}}}))}{m_\alpha(B(x,2^k\red{l_{\red{\ell}}}))}\\
    		&\le C\sum_{{\red{\ell}}=1}^\infty\int_{I_{\red{\ell}}}|b_{\red{\ell}}(y)|dm_\alpha(y)\le  C \int_{(0,\infty)^n}|f(y)|dm_\alpha(y),
    	\end{align*}
    	\red{and we can prove \eqref{B1}.}
    	
    	\red{In order to see that \eqref{B2} holds, we pass through long and short variations. We separate the family $\{J_j\}_{j=1}^{k-1}$ in two parts. We say that $J_j$ is of $I$- type when it does not contain any power of $2$ and $J_j$ is of $II$- type when it contains powers of $2$. If $J_j$ is of $I$- type, then there exists \blue{$l\in\Z$} for which $J_j\subseteq (\blue{2^l,2^{l+1}})$. On the other hand, if $J_j$ is of $II$- type, we can write $J_j=(t_{j+1},2^{m_j}]\cup (2^{m_j},2^{n_j}]\cup (2^{n_j}, t_j]$ where $m_j=\min\{\blue{l: \: 2^l}\in J_j\}$ and  $n_j=\max\{\blue{l: \: 2^l}\in J_j\}$. We define $S$ as the set of intervals $J$ in the $I$- type and those $(t_{j+1},2^{m_j}]$ and $(2^{n_j}, t_j]$ arising from intervals $J_j$ of $II$-type, such that the associated annuli translated by $x$, $x+R_J$, intersect some $I_\ell$. Also, we define $\mathcal{L}$ as the set of intervals  $(2^{m_j},2^{n_j}]$ arising from intervals $J_j$ of $II$-type, such that the  associated annuli translated by $x$, $x+R_{[2^{n_j},2^{m_j})}$, intersect some $I_\ell$.}

    	\red{
    	We have that
     \begin{align*}
     &\Bigg(\sum_{j=1}^{k-1}\left|\sum_{\ell\in \mathcal{I}_2(J_j)}(\mathcal{R}_{i,\alpha;t_j,loc}{(b_\ell)}({x})-\mathcal{R}_{i,\alpha;t_{j+1},loc}{(b_\ell)}({x}))\right|^2\Bigg)^{1/2}\\
     &\le \sqrt{3}\Bigg(\sum_{\substack{J\in\mathcal{L}\\ J \text{ related to } J_j, j=1,\dots,k-1}}\left|\sum_{\ell\in \mathcal{I}_2(J)}(\mathcal{R}_{i,\alpha;t_j,loc}{(b_\ell)}({x})-\mathcal{R}_{i,\alpha;t_{j+1},loc}{(b_\ell)}({x}))\right|^2\Bigg)^{1/2}\\
     &+\sqrt{3}\Bigg(\sum_{\substack{J\in S\\ J \text{ related to } J_j, j=1,\dots,k-1}}\Bigg|\sum_{\ell\in \mathcal{I}_2(J)}(\mathcal{R}_{i,\alpha;t_j,loc}{(b_\ell)}({x})-\mathcal{R}_{i,\alpha;t_{j+1},loc}{(b_\ell)}({x}))\Bigg|^2\Bigg)^{1/2}.
     \end{align*}
     The first sum on the right side is known as the long variation and the second one as the short variation.
}

\red{In order to prove \eqref{B2}, it is sufficient to see that
	\begin{small}
	\begin{align}\label{B3}
	m_\alpha&\Bigg(\Bigg\{x\in \tilde{\Omega}^c\blue{\setminus \mathbb{V}}:\: \sum_{\substack{J\in\mathcal{L}\\ J \text{ related to } J_j, j=1,\dots,k-1}}\Bigg|\sum_{\ell\in \mathcal{I}_2(J)}(\mathcal{R}_{i,\alpha;t_j,loc}{(b_\ell)}({x})-\mathcal{R}_{i,\alpha;t_{j+1},loc}{(b_\ell)}({x}))\Bigg|^2>\frac{1}{48}\Bigg\}\Bigg)\nonumber\\&\le C \int_{(0,\infty)^n}|f(x)|dm_\alpha(x)
	\end{align}
	\end{small}
and
\begin{small}
	\begin{align}\label{B4}
m_\alpha&\Bigg(\Bigg\{x\in \tilde{\Omega}^c\blue{\setminus \mathbb{V}}:\: \sum_{\substack{J\in S\\ J \text{ related to } J_j, j=1,\dots,k-1}}\Bigg|\sum_{\ell\in \mathcal{I}_2(J)}(\mathcal{R}_{i,\alpha;t_j,loc}{(b_\ell)}({x})-\mathcal{R}_{i,\alpha;t_{j+1},loc}{(b_\ell)}({x}))\Bigg|^2>\frac{1}{48}\Bigg\}\Bigg)\nonumber\\&\le C \int_{(0,\infty)^n}|f(x)|dm_\alpha(x).
\end{align}
\end{small}}
    
    	\begin{enumerate}
    	    	\item[(b)] \red{ Now we prove \eqref{B3}. Let $x\in{\tilde{\Omega}^c}\blue{\setminus \mathbb{V}}$ and $k\in\Z$. We denote $D_k=(2^k,2^{k+1}]$ and 	$\mathcal{I}_{2,k}=\mathcal{I}_2(D_k)$. }
    	    \end{enumerate}	
     For every $j\in\Z$,  let $\red{\mathbb{I}}_j=\{ \red{\ell}\in\N: \: l_{\red{\ell}}\in (2^j-2^{j-1},2^j+2^{j-1}]\}$  and define
    	$$
    	\triangle_j=\sum_{\red{\ell}\in \red{\mathbb{I}}_j}\int_{I_{\red{\ell}}}|f(y)|dm_\alpha(y).    	$$
    	We also consider 
    	$$
    	h_{k,j}(x)=\sum_{\ell\in \mathcal{I}_{2,k}\cap \red{\mathbb{I}}_j}|\mathcal{R}_{i,\alpha;2^k,loc}(b_\ell)(x)-\mathcal{R}_{i,\alpha;2^{k+1},loc}(b_\ell)(x)|.
    	$$
    	\blue{Suppose that 	$\ell\in \mathcal{I}_{2,k}\cap \red{\mathbb{I}}_j$. We have that
    	$$
  |x-c_\ell|\le 2^{k+1}+\frac{\sqrt{n} }{2}(2^j+2^{j-1})=2^{k+1}+\frac{3 }{2}\sqrt{n}2^{j-1}. 
$$
Since $x\not\in \tilde{I}_\ell$, it follows that $|x-c_\ell|\ge \left(\frac{3 }{2}\sqrt{n}+1\right)l_\ell\ge\left(\frac{3 }{2}\sqrt{n}+1\right)(2^j-2^{j-1})=\left(\frac{3 }{2}\sqrt{n}+1\right)2^{j-1}$.
We deduce that
$$
2^{k+1}+\frac{3 }{2}\sqrt{n}2^{j-1}\ge \left(\frac{3 }{2}\sqrt{n}+1\right)2^{j-1}.
$$
Hence, $j\le k+2$. }

 { In order to apply \cite[Lemma 2.2]{MTX2},} we are going to see that
    	$$
    	\int_{\tilde{\Omega}^c}|h_{k,j}(x)|^2\red{dm_\alpha(x)}\le C2^{j-k}\triangle_j,\quad j\in\Z, \:\ \blue{j\le k+2}.
    	$$
Let $j\in\Z.$    	From \eqref{eq3.4} \red{and the third property in (viii)}, we get that
    	\begin{align*}
    	|h_{k,j}(x)|^2&=\left(\sum_{\ell\in \mathcal{I}_{2,k}\cap \red{\mathbb{I}}_j}|\mathcal{R}_{i,\alpha;2^k,loc}(b_\ell)(x)-\mathcal{R}_{i,\alpha;2^{k+1},loc}(b_\ell)(x)| \right)^2\\
    	&\le C\left(\sum_{\ell\in \mathcal{I}_{2,k}\cap \red{\mathbb{I}}_j}\int_{2^k<|x-y|<2^{k+1}}\frac{|b_\ell(y)|}{m_\alpha(B(x,|x-y|))}dm_\alpha(y) \right)^2\\
    	&\le C \left(\sum_{\ell\in \mathcal{I}_{2,k}\cap \red{\mathbb{I}}_j}\frac{1}{m_\alpha(B(x,2^k))}\int_{I_\ell}{|b_\ell(y)|}dm_\alpha(y) \right)^2\\
        	&\le \frac{C}{m_\alpha(B(x,2^k))^2}\sum_{\ell\in \mathcal{I}_{2,k}\cap \red{\mathbb{I}}_j}{m_\alpha(I_\ell)}\sum_{\ell\in \mathcal{I}_{2,k}\cap \red{\mathbb{I}}_j}\int_{I_\ell}{|b_\ell(y)|}dm_\alpha(y), \quad x\in\tilde{\Omega}^c.
    	\end{align*}
    	
    \blue{If	$\ell\in \mathcal{I}_{2,k}\cap \red{\mathbb{I}}_j$ and $2^{k+1}>2\sqrt{n}(2^j+2^{j-1})$, we can write
    \begin{align*}
    \sum_{\ell\in \mathcal{I}_{2,k}\cap \red{\mathbb{I}}_j}{m_\alpha(I_\ell)}&\le C [m_\alpha(B(x,2^{k+1}+\blue{\sqrt{n}(2^j+2^{j-1})})\setminus B(x,2^{k+1}-\blue{\sqrt{n}(2^j+2^{j-1})}))\\&+m_\alpha(B(x,2^{k}+\blue{\sqrt{n}(2^j+2^{j-1})})\setminus B(x,2^{k}-\blue{\sqrt{n}(2^j+2^{j-1}))})],
    \end{align*}
where}
    		\begin{align*}
   &m_\alpha(B(x,2^{k+1}+\blue{\sqrt{n}(2^j+2^{j-1})})\setminus B(x,2^{k+1}-\blue{\sqrt{n}(2^j+2^{j-1})}))\\
   &=C\int_{B(x,2^{k+1}+\blue{\sqrt{n}(2^j+2^{j-1})})\setminus B(x,2^{k+1}-\blue{\sqrt{n}(2^j+2^{j-1})})}\prod_{i=1}^n \blue{z_i}^{2\alpha_i+1}d\blue{z_i}\\
   &\le C \prod_{i=1}^n(x_i+2^{\blue{k+1}}+\blue{\sqrt{n}(2^j+2^{j-1})})^{2\alpha_i+1}((2^{k+1}+\blue{\sqrt{n}(2^j+2^{j-1})})^n- (2^{k+1}-\blue{\sqrt{n}(2^j+2^{j-1}))}^n)\\
   &\le C\; 2^{k(n-1)}2^j \prod_{i=1}^n(x_i+2^{k+2})^{2\alpha_i+1}\le  C\; 2^{k(n-1)}2^j \prod_{i=1}^n(x_i+2^{k})^{2\alpha_i+1}.
    	\end{align*}
    	\blue{If $\ell\in \mathcal{I}_{2,k}\cap \red{\mathbb{I}}_j$ and $2^{k+1}\le 2\sqrt{n}(2^j+2^{j-1})$ being $j\le k+2$,
    		We have that
    		\begin{align*}
    		\sum_{\ell\in \mathcal{I}_{2,k}\cap \red{\mathbb{I}}_j}{m_\alpha(I_\ell)}&\le C m_\alpha(B(x,2^{k+1}+\blue{\sqrt{n}(2^j+2^{j-1})})),
    		\end{align*}
    		where
    	\begin{align*}
    	   m_\alpha(B(x,2^{k+1}+&\sqrt{n}(2^j+2^{j-1})))\\
    	&\le C \prod_{i=1}^n(x_i+2^{k+1}+\sqrt{n}(2^j+2^{j-1}))^{2\alpha_i+1}(2^{k+1}+\sqrt{n}(2^j+2^{j-1}))^n\\
    	&\le C\; 2^{k(n-1)}2^j \prod_{i=1}^n(x_i+2^{k})^{2\alpha+1}.
    	\end{align*}}
    	Therefore, 
    	$$
    		\sum_{\ell\in \mathcal{I}_{2,k}\cap \red{\mathbb{I}}_j}{m_\alpha(I_\ell)}\le C\;   2^{k(n-1)}2^j \prod_{i=1}^n(x_i+2^{k})^{2\alpha_i+1}.
    	$$

        On the other hand, since $|x-c_\ell|\sim 2^k$, when $x\in\tilde{\Omega}^c$ and $\ell\in \mathcal{I}_{2,k}\cap\red{\mathbb{I}}_j$, we obtain that
        \begin{align*}
        m_\alpha(B(x,2^{k}))&\sim  m_\alpha(B(x,|x-c_\ell|))\sim|x-c_\ell|^n\prod_{i=1}^n(x_i+|x-c_\ell|)^{2\alpha_i+1}\\
        &\sim |x-c_\ell|^n\prod_{i=1}^n(x_i+2^k)^{2\alpha_i+1}, \quad x\in\tilde{\Omega}^c\:\; \text{and } \ell\in \mathcal{I}_{2,k}\cap\red{\mathbb{I}}_j.
        \end{align*}
        Since $|x-c_\ell|\ge 2^{k-1}$ when $x\in\tilde{\Omega}^c$ and $\ell\in \mathcal{I}_{2,k}\cap\red{\mathbb{I}}_j$, we can write 
       
        \begin{align}\label{F}
        \int_{\tilde{\Omega}^c}&\red{|h_{k,j}(x)|^2dm_\alpha(x)}\nonumber\\&\red{\le C     \int_{\tilde{\Omega}^c}\frac{2^{k(n-1)}2^j\prod_{i=1}^n(x_i+2^k)^{2\alpha_i+1}}{m_\alpha(B(x,2^k))^2}	\sum_{\ell\in \mathcal{I}_{2,k}\cap \red{\mathbb{I}}_j}\int_{I_\ell}|b_\ell(y)|dm_\alpha(y)dm_\alpha(x)}  \nonumber   \\      &\le C \sum_{\ell\in \mathcal{I}_{2,k}\cap\red{\mathbb{I}}_j}\int_{I_\ell}|f(y)|dm_\alpha(y)\int_{|x-c_\ell|>2^{k-1}}\frac{2^{k(n-1)}2^j}{|x-c_\ell|^{2n}}dx\\
        &\le C   \sum_{\ell\in \mathcal{I}_{2,k}\cap\red{\mathbb{I}}_j}\int_{I_\ell}|f(y)|dm_\alpha(y)2^{j-k}\le C\triangle_j 2^{j-k}.\nonumber
        \end{align}
              \blue{Notice that we have used (ix) to get \eqref{F}.}
              
                 {Thus, by using \cite[Lemma 2.2]{MTX2}, step 4 in the proof of \cite[Theorem 1.1]{MTX2}, adapted to our setting, is completed \red{ and \eqref{B3} is proved}.} 
            \begin{enumerate}
            
            \item[(c)] In order to prove step 5 in the proof of \cite[Theorem 1.1]{MTX2} for our case, we can proceed as in (b).
            \red{Thus, \eqref{B4} is proved.}
             \end{enumerate}
            
            
          { To prove the $L^q$-boundedness of $ \V_\rho(\{\mathcal{R}_{i,\alpha;\epsilon,loc}\}_{\epsilon>0})$ for the remaining values of $q$,} we are going to see that $\V_\rho(\{\mathcal{R}_{i,\alpha;\epsilon,loc}\}_{\epsilon>0})$ is bounded from $L^\infty((0,\infty)^n)$ into \newline \begin{small}$BMO((0,\infty)^n,m_\alpha)$\end{small} {and then we shall use an interpolation argument.}
       
            We recall that the space $BMO((0,\infty)^n,m_\alpha)$ consists of all those functions $f\in L^1_{loc}((0,\infty)^n,m_\alpha)$ such that
            $$
            \|f\|_{BMO((0,\infty)^n,m_\alpha)}:=\sup_{Q \text{ cube in }(0,\infty)^n}\inf_{c\in\R}\frac{1}{m_\alpha(Q)}\int_Q|f(y)-c|dm_\alpha(y)<\infty.
            $$
            In order to prove our objective we adapt some ideas used by Mas in \cite{Mas}. We have to find $C>0$ such that, for every cube $Q$ in $(0,\infty)^n$ and $f\in L^\infty((0,\infty)^n)$, there exists $c\in\R$ such that
            $$
            \int_{Q}|\V_\rho(\{\mathcal{R}_{i,\alpha;\epsilon,loc}\}_{\epsilon>0})(f)(y)-c|dm_\alpha(y)\le C\|f\|_{L^\infty((0,\infty)^n)}m_\alpha(Q).
            $$
            Let $f\in L^\infty((0,\infty)^n)$ and $Q$ be a cube in $(0,\infty)^n$. We decompose $f$ as follows
            $$
            f=f_1+f_2,
            $$
           where $f_1=f\chi_{3Q}$.
           
           Since $\V_\rho(\{\mathcal{R}_{i,\alpha;\epsilon,loc}\}_{\epsilon>0})$ is bounded from $L^2((0,\infty)^n,{ m_\alpha})$ into itself, we have that
           \begin{align*}
             \int_{Q}|&\V_\rho(\{\mathcal{R}_{i,\alpha;\epsilon,loc}\}_{\epsilon>0})(f_1)(y)|dm_\alpha(y)\le (m_\alpha(Q))^{1/2}\|\V_\rho(\{\mathcal{R}_{i,\alpha;\epsilon,loc}\}_{\epsilon>0})(f_1)\|_{L^2((0,\infty)^n,{ m_\alpha})}\\
             &\le C(m_\alpha(Q))^{1/2}\|f_1\|_{L^2((0,\infty)^n,{ m_\alpha})}\le C(m_\alpha(Q))^{1/2}(m_\alpha(3Q))^{1/2}\|f\|_{L^{{\infty}}((0,\infty)^n)}\\
             &\le C \|f\|_{L^\infty((0,\infty)^n)}m_\alpha(Q).
           \end{align*}
           Since $\V_\rho(\{\mathcal{R}_{i,\alpha;\epsilon,loc}\}_{\epsilon>0})$ is sublinear, for every $c\in\R$ we get that
           \begin{align*}
           |\V_\rho(\{\mathcal{R}_{i,\alpha;\epsilon,loc}\}_{\epsilon>0})(f)-c|\le \V_\rho(\{\mathcal{R}_{i,\alpha;\epsilon,loc}\}_{\epsilon>0})(f_1)+ |\V_\rho(\{\mathcal{R}_{i,\alpha;\epsilon,loc}\}_{\epsilon>0})(f_2)-c|.
           \end{align*}
           We denote by $z_Q$ and $\red{l_Q}$ the center and the side length of $Q$, respectively. We consider $c=\V_\rho(\{\mathcal{R}_{i,\alpha;\epsilon,loc}\}_{\epsilon>0})(f_2)(z_Q)$. We can {assure} that $c<\infty$. Indeed, first note  that if $(x,y,s)\in N_\tau,$ for some $\tau>0$, it follows that
           \begin{align*}
           |x-y|^2=\sum_{i=1}^n(x_i^2+y_i^2-2x_iy_i)\le \sum_{i=1}^n(x_i^2+y_i^2-2x_iy_is_i)=q_-(x,y,s)\le \left(\frac{2C_0}{1+|x|+|y|}\right)^2.
                     \end{align*}
                     According to \eqref{eq3.4}, we obtain that
                                         
                                          \begin{align*}
                     \V_\rho(\{\mathcal{R}_{i,\alpha;\epsilon,loc}\}_{\epsilon>0})(f_2)(x)&\le C \int_{B\left(x,\frac{2C_0}{1+|x|}\right)\cap[(0,\infty)^n\setminus(3Q)]}\frac{|f(y)|}{m_\alpha(B(x,|x-y|))}dm_\alpha(y)\\
                     &\le C\frac{m_\alpha(B(x,2C_0))}{m_\alpha(B(x,\red{l_Q}))}\|f\|_{L^\infty((0,\infty)^n)}<\infty, \quad x\in Q.
                     \end{align*}
                     
                    On the other hand, we can write
                    \begin{align*}
                    |\V_\rho(\{\mathcal{R}_{i,\alpha;\epsilon,loc}\}_{\epsilon>0})(f_2)(x)-c|&\le \sup_{0<\epsilon_k<\dots<\epsilon_1}\Big(\sum_{j=1}^{k-1}\Big|\int_{\epsilon_{j+1}<|x-y|<\epsilon_j} \mathcal{R}^i_{\alpha,loc}(x,y)f_2(y)dm_\alpha(y)\\&-\int_{\epsilon_{j+1}<|z_Q-y|<\epsilon_j} \mathcal{R}^i_{\alpha,loc}(z_Q,y) f_2(y)dm_\alpha(y)\Big|^\rho\Big)^{1/ \rho}, \quad x\in Q.
                    \end{align*}
                    Let $x\in Q$. We choose $k\in\N$ and $0<\epsilon_k<\dots<\epsilon_1$ such that
                    \begin{small}
                    \begin{align*}
                    &|\V_\rho(\{\mathcal{R}_{i,\alpha;\epsilon,loc}\}_{\epsilon>0})(f_2)(x)-c|^\rho\\
                    &\le 2  \sum_{j=1}^{k-1}\Big|\int_{\epsilon_{j+1}<|x-y|<\epsilon_j} \mathcal{R}^i_{\alpha,loc}(x,y)f_2(y)dm_\alpha(y)-\int_{\epsilon_{j+1}<|z_Q-y|<\epsilon_j} \mathcal{R}^i_{\alpha,loc}(z_Q,y) f_2(y)dm_\alpha(y)\Big|^\rho.
                    \end{align*}
                    For every $j=1,\dots,k-1$, we have that
                    \begin{align*}
                    \Big|\int_{\epsilon_{j+1}<|x-y|<\epsilon_j} \mathcal{R}^i_{\alpha,loc}(x,y)f_2(y)dm_\alpha(y)-\int_{\epsilon_{j+1}<|z_Q-y|<\epsilon_j} &\mathcal{R}^i_{\alpha,loc}(z_Q,y) f_2(y)dm_\alpha(y)\Big|\\
                    &\quad\le \|f\|_{L^\infty((0,\infty)^n)}(\alpha_j+\beta_j), 
                    \end{align*}
                    \end{small}
                    	where
                    	$$
                    	\alpha_j=\int_{(0,\infty)^n\setminus (3Q)}\chi_{\{\epsilon_{j+1}<|z|<\epsilon_j\}}(x-y)|\mathcal{R}^i_{\alpha,loc}(x,y)- \mathcal{R}^i_{\alpha,loc}(z_Q,y)| dm_\alpha(y)
                    	$$
                    	  and
                    	$$
                    	\beta_j=\int_{(0,\infty)^n\setminus (3Q)}|\chi_{\{\epsilon_{j+1}<|z|<\epsilon_j\}}(x-y)-\chi_{\{\epsilon_{j+1}<|z|<\epsilon_j\}}(z_Q-y)|| \mathcal{R}^i_{\alpha,loc}(z_Q,y)| dm_\alpha(y).
                    	$$

Then, 
$$
|\V_\rho(\{\mathcal{R}_{i,\alpha;\epsilon,loc}\}_{\epsilon>0})(f_2)(x)-c|\le C \|f\|_{L^\infty((0,\infty)^n)}\left(\sum_{j=1}^{k-1}(\alpha_j+\beta_j)^\rho\right)^{1/\rho}.
$$
By using \eqref{eq3.5}, we get that
\begin{align*}
\left(\sum_{j=1}^{k-1}\alpha_j^\rho\right)^{1/\rho}&\le \sum_{j=1}^{k-1}\alpha_j\le C \int_{(0,\infty)^n\setminus (3Q)}\frac{|x-z_Q|}{|z_Q-y|m_\alpha(B(z_Q,|z_Q-y|))}dm_\alpha(y)\\
&\le C \red{l_Q}\sum_{\ell=1}^\infty\int_{3^{\red{\ell}+1}Q\setminus 3^\ell Q}\frac{1}{|z_Q-y|m_\alpha(B(z_Q,|z_Q-y|))}dm_\alpha(y)\\
	&\le C\sum_{\ell=1}^\infty\frac{m_\alpha(3^{\ell+1}Q)}{m_\alpha(B(z_Q,3^\ell\red{l_Q}))\red{3}^\ell}\le C,
\end{align*}
where $C$ does not depend on $(x,Q)$.

We now estimate $\left(\sum_{j=1}^{k-1}\beta_j^\rho\right)^{1/\rho}.$ We have that $|x-z_Q|\le\frac{\sqrt{n}}{2}\red{l_Q}$. We choose $M\in\N$ such that $\sqrt{n}<2^{M-1}$. We can assume that there exists $j_0\in\N$, $1\le j_0\le k$, $\epsilon_{j_0}=2^{M+2}\red{l_Q}$. If it is necessary, we can consider $\{\epsilon_j \}_{j=1}^k\cup \{2^{M+2}\red{l_Q}\}$ instead of $\{\epsilon_j \}_{j=1}^k$.

We define, as in \cite{Mas},
$$
J_0=\{j\in\{1,\dots,k\}: \:j\ge j_0\}, 
$$
and for every $m>M+2$,
\begin{align*}
J_{m,1}&=\{j\in\{1,\dots,{k-1}\}: \:2^{m-1}\red{l_Q}\le \epsilon_{j+1}<\epsilon_j\le 2^m \red{l_Q}, \text{ and } \epsilon_j-\epsilon_{j+1}\ge 2^M\red{l_Q}\},  \\
J_{m,2}&=\{j\in\{1,\dots,{k-1}\}: \:2^{m-1}\red{l_Q}\le \epsilon_{j+1}<\epsilon_j\le 2^m \red{l_Q}, \text{ and } \epsilon_j-\epsilon_{j+1}< 2^M\red{l_Q}\},\\
J_{m,3}&=\{j\in\{1,\dots,{k-1}\}: \:2^{m-1}\red{l_Q}\le \epsilon_{j+1}\le 2^m \red{l_Q}<\epsilon_j \}.
\end{align*}
It is clear that $\{1,\dots,k\}=J_0\cup (\cup_{m>M+2}(J_{m,1}\cup J_{m,2}\cup J_{m,3})$.

From \eqref{eq3.4} we deduce that
\begin{align*}
&\left(\sum_{j\in J_0}\beta_j^\rho\right)^{1/\rho}\le C \sum_{j\in J_0}\int_{(0,\infty)^n\setminus(3Q)}\frac{\chi_{\{\epsilon_{j+1}\le |{z}|\le \epsilon_j\}}(x-y)+\chi_{\{\epsilon_{j+1}\le |{z}|\le \epsilon_j\}}(z_Q-y)}{m_\alpha(B(z_Q,|z_Q-y|))}dm_\alpha(y)\\
&\quad\le C \left[\int_{\substack{|x-y|<2^{M+2}\red{l_Q}\\(0,\infty)^n\setminus(3Q)}}\frac{dm_\alpha(y)}{m_\alpha(B(z_Q,|z_Q-y|))}+ \int_{\substack{|z_Q-y|<2^{M+2}\red{l_Q}\\(0,\infty)^n\setminus(3Q)}}\frac{dm_\alpha(y)}{m_\alpha(B(z_Q,|z_Q-y|))}\right]\\&\\
&\quad\le C \frac{m_\alpha(B(z_Q,(2^{M+2}+\sqrt{n})\red{l_Q}))+m_\alpha(B(z_Q,2^{M+2}\red{l_Q}))}{m_\alpha(B(z_Q,\red{l_Q}))}\le C,
\end{align*}
being $C$ independent of $(x,Q)$.

Let $m>M+2$. We define $A(a,r,R)=\{y\in (0,\infty)^n: \: r\le |y-a|\le R\}$ with $a\in (0,\infty)^n$, $0<r<R$. Assume that $j\in J_{{m,1}}$. We have that
\begin{align*}
\supp(&\chi_{\{\epsilon_{j+1}\le |z|\le \epsilon_j\}}(x-\cdot)-\chi_{\{\epsilon_{j+1}\le |z|\le \epsilon_j\}}(z_Q-\cdot))\\
&\subseteq (A(x,\epsilon_{j+1},\epsilon_j)\setminus A(z_Q,\epsilon_{j+1},\epsilon_j))\cup (A(z_Q,\epsilon_{j+1},\epsilon_j)\setminus A(x,\epsilon_{j+1},\epsilon_j))\\
&=:A_j(x,z_Q).
\end{align*}
Since $x\in Q$ and $j\in J_{{m,1}}$, $A_j(\red{x},z_Q)\subset A_1\cup A_2$, where 
$$
A_1=A(x,\epsilon_{j+1}-2^M\red{l_Q}, \epsilon_{j+1}+2^M\red{l_Q})
$$
and 
$$
A_2=A(x,\epsilon_{j}-2^M\red{l_Q}, \epsilon_{j}+2^M\red{l_Q}).
$$
By passing from balls to cubes and proceeding as above, we deduce that 
\begin{align*}
m_\alpha&(\{ y\in (0,\infty)^n: \:\chi_{\{\epsilon_{j+1}\le |{z}|\le \epsilon_j\}}(x-y)-\chi_{\{\epsilon_{j+1}\le |{z}|\le \epsilon_j\}}(z_Q-y)\neq 0\})\le m_\alpha(A_1\cup A_2) \\
&\le C2^M\red{l_Q}(\epsilon_j+2^M\red{l_Q})^{n-1}\left[\prod_{i=1}^n(x_i+\epsilon_j+2^M\red{l_Q})^{2\alpha_i+1}+\prod_{i=1}^n(z_{Q,i}+\epsilon_j+2^M\red{l_Q})^{2\alpha_i+1}\right]\\
&\le C2^M\red{l_Q}(2^m\red{l_Q}+2^M\red{l_Q})^{n-1}\prod_{i=1}^n(z_{Q,i}+\epsilon_j+2^M\red{l_Q})^{2\alpha_i+1}\\
&\le C(\red{l_Q})^n2^{m(n-1)}\prod_{i=1}^n(z_{Q,i}+\epsilon_j+2^M\red{l_Q})^{2\alpha_i+1},
\end{align*}
where $z_{Q,i}$ denotes the $i$-coordinate of $z_{Q}.$
By \eqref{eq3.4}, we get that
\begin{align*}
|&\mathcal{R}^i_{\alpha,loc}(z_Q,y)|\le \frac{C}{m_\alpha(B(z_Q,|z_Q-y|))}\\
&\quad\le \frac{C}{|z_Q-y|^n\prod_{i=1}^n|z_{Q,i}+|z_Q-y||^{2\alpha_i+1}}\\
&\quad\le\frac{C}{(2^{m-1}\red{l_Q})^n\prod_{i=1}^n(z_{Q,i}+\epsilon_{j+1}-\red{l_Q})^{2\alpha_i+1}}\\
 &\quad\le\frac{C}{(2^{m-1}\red{l_Q})^n\prod_{i=1}^n(z_{Q,i}
	{+\epsilon_{j}+2^M\red{l_Q}})^{2\alpha_i+1}},\:\: y\in ((0,\infty)^n\setminus(3Q)^c)\cap A_j(x,z_Q).
\end{align*}
In the last inequality we have used that
\begin{align*}
    z_{Q,i}+\epsilon_{j+1}-\red{l_Q}&=z_{Q,i}+\frac{\epsilon_{j}}{4}+\frac{\epsilon_{j+1}}{2}+\frac{\epsilon_{j+1}}{2}-\frac{\epsilon_{j}}{4}-\red{l_Q}\ge z_{Q,i}+\frac{\epsilon_{j}}{4}+\frac{\epsilon_{j+1}}{2}-\red{l_Q}\\
    &\ge z_{Q,i}+\frac{\epsilon_{j}}{4}+2^{m-2}\red{l_Q}-\red{l_Q}\ge\frac{1}{4}(z_{Q,i}+{\epsilon_{j}}+2^{M}\red{l_Q}).
\end{align*}
Then,
\begin{small}
\begin{align*}
\sum_{j\in J_{m,1}}\beta_j^\rho &\le C \sum_{j\in J_{m,1}} \left(\frac{m_\alpha(\{ y\in (0,\infty)^n: \:\chi_{\{\epsilon_{j+1}\le |{z}|\le \epsilon_j\}}(x-y)-\chi_{\{\epsilon_{j+1}\le |{z}|\le \epsilon_j\}}(z_Q-y)\neq 0\})}{(2^{m-1}\red{l_Q})^n\prod_{i=1}^n(z_{Q,i}+\epsilon_{j}+2^M\red{l_Q})^{2\alpha_i+1}} \right)^\rho\\
&\le C \sum_{j\in J_{m,1}} \frac{1}{2^{m\rho}}.
\end{align*}
\end{small}
Since the set $J_{m,1}$ has at most $2^{m-1-M}$ elements, we conclude that
$$
\sum_{j\in J_{m,1}}\beta_j^\rho \le C 2^{m(1-\rho)}\le C2^{-m}.
$$
Suppose that $j\in J_{m,2}$.  As above, we have that
\begin{align*}
&m_\alpha(\{ y\in (0,\infty)^n: \:\chi_{\{\epsilon_{j+1}\le |{z}|\le \epsilon_j\}}(x-y)-\chi_{\{\epsilon_{j+1}\le |{z}|\le \epsilon_j\}}(z_Q-y)\neq 0\})\\&\\
&\le m_\alpha(\{ y\in (0,\infty)^n: \:\chi_{\{\epsilon_{j+1}\le |{z}|\le \epsilon_j\}}(x-y)=1\}\\&\qquad+m_\alpha(\{ y\in (0,\infty)^n: \:\chi_{\{\epsilon_{j+1}\le |{z}|\le \epsilon_j\}}(z_Q-y)=1\})\\
&\le C(\epsilon_j-\epsilon_{j+1})\epsilon_j^{n-1}\left(\prod_{i=1}^n(x_i+\epsilon_j)^{2\alpha_i+1}+\prod_{i=1}^n(z_{Q,i}+\epsilon_j)^{2\alpha_i+1} \right)\\
&\le C(\epsilon_j-\epsilon_{j+1})\epsilon_j^{n-1}\left(\prod_{i=1}^n(z_{Q,i}+\red{l_Q}+\epsilon_j)^{2\alpha_i+1}+\prod_{i=1}^n(z_{Q,i}+\epsilon_j)^{2\alpha_i+1} \right)
\end{align*}
and 
\begin{align*}
|&\mathcal{R}^i_{\alpha,loc}(z_Q,y)|\le \frac{C}{m_\alpha(B(z_Q,|z_Q-y|))}\le \frac{C}{|z_Q-y|^n\prod_{i=1}^n(z_{Q,i}+|z_Q-y|)^{2\alpha_i+1}}\\
&\le\frac{C}{(2^{m-1}\red{l_Q})^n\prod_{i=1}^n(z_{Q,i}+2^{m-1}\red{l_Q})^{2\alpha_i+1}}, \quad y\in A_j(x, z_Q)\cap ((0,\infty)^n\setminus(3Q)). 
\end{align*}
We get that
\begin{align*}
\sum_{j\in J_{m,2}}&\left(\frac{m_\alpha(\{ y\in (0,\infty)^n: \:\chi_{\{\epsilon_{j+1}\le |{z}|\le \epsilon_j\}}(x-y)-\chi_{\{\epsilon_{j+1}\le |{z}|\le \epsilon_j\}}(z_Q-y)\neq 0\})}{(2^{m-1}\red{l_Q})^n\prod_{i=1}^n(z_{Q,i}+2^{m-1}\red{l_Q})^{2\alpha_i+1}} \right)^\rho\\
&\le C \sum_{j\in J_{m,2}}\frac{(\epsilon_j-\epsilon_{j+1})^\rho\epsilon_j^{(n-1)\rho}}{(2^m\red{l_Q})^{n\rho}}\le  C \sum_{j\in J_{m,2}}\frac{(\epsilon_j-\epsilon_{j+1})(2^M\red{l_Q})^{\rho-1}(2^m\red{l_Q})^{(n-1)\rho}}{(2^m\red{l_Q})^{n\rho}}\\
&\le  C \sum_{j\in J_{m,2}}\frac{\epsilon_j-\epsilon_{j+1}}{\red{l_Q}}2^{-m\rho}\le C 2^{\rho(1-m)}\le C 2^{-m}.
\end{align*}
Suppose that $j\in  J_{m,3}$. It is clear that $J_{m,3}$ contains at most one element. We distinguish two cases. Assume first that $\epsilon_j-\epsilon_{j+1}<2^M\red{l_Q}$. We proceed as in the case $j\in  J_{m,2}$. We can write
\begin{align*}
m_\alpha&(\{ y\in (0,\infty)^n: \:\chi_{\{\epsilon_{j+1}\le |{z}|\le \epsilon_j\}}(x-y)-\chi_{\{\epsilon_{j+1}\le |{z}|\le \epsilon_j\}}(z_Q-y)\neq 0\})\\
&\le C (\epsilon_j-\epsilon_{j+1})\epsilon_j^{n-1}\left(\prod_{i=1}^n(z_{Q,i}+\red{l_Q}+\epsilon_j)^{2\alpha_i+1}+\prod_{i=1}^n(z_{Q,i}+\epsilon_j)^{2\alpha_i+1}\right)\\
&\le C 2^M\red{l_Q}(2^m\red{l_Q}+2^M\red{l_Q})^{n-1}\left(\prod_{i=1}^n(z_{Q,i}+\red{l_Q}+\epsilon_j)^{2\alpha_i+1}+\prod_{i=1}^n(z_{Q,i}+\epsilon_j)^{2\alpha_i+1}\right),
\end{align*}
and
$$
\beta_j\le C \frac{2^M\red{l_Q}(2^m\red{l_Q}+2^M\red{l_Q})^{n-1}}{(2^{m-1}\red{l_Q})^n}\le C 2^{-m}.
$$
Secondly, suppose that $\epsilon_{j}-\epsilon_{j+1}>2^M\red{l_Q}$. We now argue as in the case $j\in J_{m,1}$.

By keeping the notation above, we have that
$$
\supp(\chi_{\{\epsilon_{j+1}\le |z|\le \epsilon_j\}}(x-\cdot)-\chi_{\{\epsilon_{j+1}\le |z|\le \epsilon_j\}}(z_Q-\cdot))\subset A_j(x,z_Q)\subset A_1\cup A_2.
$$
We get that
\begin{align*}
m_\alpha(A_1)&\le C 2^M\red{l_Q}(\epsilon_{j+1}+2^M\red{l_Q})^{n-1}\prod_{i=1}^n (z_{Q,i}+\epsilon_{j+1}+2^M \red{l_Q})^{2\alpha_i+1}\\
&\le C(\epsilon_{j+1})^{n-1}\red{l_Q}\prod_{i=1}^n (z_{Q,i}+\epsilon_{j+1}+2^M \red{l_Q})^{2\alpha_i+1}\\
&\le C(\epsilon_{j+1})^{n-1}\red{l_Q}\prod_{i=1}^n (z_{Q,i}+2^m \red{l_Q})^{{2\alpha_i+1}}
\end{align*}
and
\begin{align*}
|\mathcal{R}^i_{\alpha,loc}(z_Q,y)|&\le\frac{C}{(2^{m-1}\red{l_Q})^n\prod_{i=1}^n(z_{Q,i}+\epsilon_{j+1}-\red{l_Q})^{2\alpha_i+1}}\\
&\le\frac{C}{(2^{m-1}\red{l_Q})^n\prod_{i=1}^n (z_{Q,i}+2^m \red{l_Q})^{2\alpha_i+1}}, \quad y\in A_1\cap ((0,\infty)^n\setminus(3Q)). 
\end{align*}
We choose $J_j\in \N$ such that $2^{J_j-1}\red{l_Q}\le \epsilon_{j}\le 2^{J_j}\red{l_Q}.$ It is clear that $J_j>m$. We obtain that
\begin{align*}
m_\alpha(A_2)&\le C 2^M\red{l_Q} (\epsilon_{j}+2^M \red{l_Q})^{n-1}\prod_{i=1}^n (z_{Q,i}+\epsilon_{j}+2^M \red{l_Q})^{2\alpha_i+1}\\
&\le C \red{l_Q}\epsilon_{j}^{n-1}\prod_{i=1}^n (z_{Q,i}+\epsilon_{j})^{2\alpha_i+1}\le C 2^{J_j(n-1)}(\red{l_Q})^n\prod_{i=1}^n (z_{Q,i}+\epsilon_{j})^{2\alpha_i+1},
\end{align*}
and
\begin{align*}
|\mathcal{R}^i_{\alpha,loc}(z_Q,y)|&\le\frac{C}{(2^{J_j}\red{l_Q})^n\prod_{i=1}^n(z_{Q,i}+|z_{Q}-y|)^{2\alpha_i+1}}\\
&\le\frac{C}{(2^{J_j}\red{l_Q})^n\prod_{i=1}^n (z_{Q,i}+\epsilon_{j}-2^M \red{l_Q}-\red{l_Q})^{2\alpha_i+1}} \\
&\le \frac{C}{(2^{J_j}\red{l_Q})^n\prod_{i=1}^n (z_{Q,i}+\epsilon_{j})^{2\alpha_i+1}},
\quad y\in A_2\cap ((0,\infty)^n\setminus(3Q)). 
\end{align*}
It follows that
$$
\beta_j\le C (2^{-m}+2^{-J_j})\le C 2^{-m}.
$$
Then, 
$$
\sum_{j\in J_{m,3}}\beta_j^\rho\le C 2^{-m}.
$$
By putting together all the above estimates, we conclude that
$$
 |\V_\rho(\{\mathcal{R}_{i,\alpha;\epsilon,loc}\}_{\epsilon>0})(f_2)(x)-c|\le C \|f\|_{L^\infty((0,\infty)^n)}, \quad  \red{ x\in Q.}
$$
We get
$$
\frac{1}{m_\alpha(Q)}\int_Q |\V_\rho(\{\mathcal{R}_{i,\alpha;\epsilon,loc}\}_{\epsilon>0})(f_2)(x)-c|dm_\alpha(x)\le C \|f\|_{L^\infty((0,\infty)^n)}.
$$
Thus we have proved that $\V_\rho(\{\mathcal{R}_{i,\alpha;\epsilon,loc}\}_{\epsilon>0})$ is bounded from $L^\infty((0,\infty)^n)$ into \newline $BMO((0,\infty)^n,m_\alpha)$.

The operator  $\V_\rho(\{\mathcal{R}_{i,\alpha;\epsilon,loc}\}_{\epsilon>0})$ is sublinear and positive. We have just proved that $\V_\rho$ maps $L^\infty((0,\infty)^n)+L^2((0,\infty)^n,m_\alpha)$ into  $ L^1_{loc}((0,\infty)^n,m_\alpha)$. By using \cite[Lemma 5.1]{Mas}, we can apply an interpolation argument to see that $\V_\rho(\{\mathcal{R}_{i,\alpha;\epsilon,loc}\}_{\epsilon>0})$ is bounded from $L^q((0,\infty)^n,m_\alpha)$ into itself, for every $2<q<\infty$.

From the arguments in Lemma  {\ref{Claim3}}, we obtain that $\V_\rho(\{\mathcal{R}_{i,\alpha;\epsilon,loc}\}_{\epsilon>0})$ is bounded from $L^q((0,\infty)^n,\nu_\alpha)$ into itself, for every $1<q<\infty$ and from $L^1((0,\infty)^n,\nu_\alpha)$ into $L^{1,\infty}((0,\infty)^n,\nu_\alpha)$.
\edproof

\subsection{Proof of Theorem \ref{teo1.3} ii)}

\textcolor{white}{v}
\vspace{0.2cm }

 {The $L^p$-boundedness properties of the oscillation operator 
$\OO(\{\mathcal{R}_{i,\alpha;\epsilon}\}_{\epsilon>0},\{\epsilon_j\}_{j\in\N})$ can be proved in an analogous way as we did for the variation operator $\V_\rho(\{\mathcal{R}_{i,\alpha;\epsilon,loc}\}_{\epsilon>0})$. We leave the details to the interested reader.}

\edproof

\subsection{Proof of Theorem \ref{teo1.3} iii)}

\textcolor{white}{v}
\vspace{0.2cm }

Finally, the $L^p$-boundedness properties of the jump operators $ \lambda \Lambda(\{{R}_{i,\alpha;\epsilon}\}_{\epsilon>0},\lambda)^{1/\rho}$ with $\rho>0$, can be deduced from the corresponding properties for $\V_\rho(\{{R}_{i,\alpha;\epsilon}\}_{\epsilon>0})$, see \cite[p. 6712]{JSW}.

\edproof
\vspace{0.2cm}

{\bf Acknowledgments.}
	The first author was partially supported by PID2019-106093GB-I00 and the second author 
was partially supported by ERCIM `Alain Bensoussan' Fellowship Programme.
\vspace{0.1cm}

%
%
%
%
%
%
    
\bibliographystyle{siam}
\bibliography{references}

\end{document}